\documentclass{aic}

\aicAUTHORdetails{%
  title = {Induced subgraphs and tree decompositions\\
  IX. Grid theorem for perforated graphs}, %% please capitalize all significant words
  author = {Bogdan Alecu, Maria Chudnovsky, Sepehr Hajebi, and Sophie Spirkl},
  plaintextauthor = {Bogdan Alecu, Maria Chudnovsky, Sepehr Hajebi, Sophie Spirkl},
  plaintexttitle = {Induced subgraphs and tree decompositions IX. Grid theorem for perforated graphs}, %%  title without math or LaTeX
  runningtitle = {Induced subgraphs and tree decompositions IX.}, 
  copyrightauthor = {B. Alecu, M. Chudnovsky, S. Hajebi and S. Spirkl},
  keywords = {Tree decompositions, Treewidth, Induced subgraphs, graph minors.},
}   %%% END \aicAUTHORdetails

\aicEDITORdetails{%
   year={2025},
   number={3},
   received={24 May 2023},   % received date: example: 7 January 2017
   published={12 March 2025},  % published date
   doi={10.19086/aic.2025.3},      % XXX = number of paper, e.g. aic006 for paper#6
}   %%% END \aicEDITORdetails

\usepackage[shortlabels]{enumitem}
\usepackage{mathdots}
\usepackage{dsfont}

\newtheorem{theorem}{Theorem}[section]
\newtheorem{lemma}[theorem]{Lemma}

\newtheorem{corollary}[theorem]{Corollary}

\DeclareMathOperator{\tw}{tw}

\DeclareMathOperator{\cher}{{\rm \textsf{Cher}}}

\def\dd{\hbox{-}}   

\newcommand{\mf}{\mathfrak}
\newcommand{\mca}{\mathcal}
\newcommand{\poi}{\mathbb{N}}

\newcounter{tbox}
\newcommand{\sta}[1]{\medskip\medskip\refstepcounter{tbox}\noindent{\parbox{\textwidth}{(\thetbox) \emph{#1}}}\vspace*{0.3cm}}
\newcommand{\mylongtitle}[1]{%
  \ifodd\value{page}%   
    \protect\parbox{0.97\linewidth}{#1}\hfill%
  \else%
    \hfill\protect\parbox{0.97\linewidth}{#1}%
  \fi%
}

\usepackage{tikz}
\usetikzlibrary{decorations.pathmorphing}
\usetikzlibrary{arrows,positioning}

\tikzset{snake it/.style={decorate, decoration=snake}}

\definecolor{mblack}{RGB}{40, 40, 43}
\begin{document}

\begin{frontmatter}[classification=text]
%% EDITOR: this will force the keywords to appear right after the Abstract.
%%   If the abstract is too long and would force the keywords off the
%%   front page, please comment out % [classification=text] above
%%   This way the keywords will be floated on the bottom of the first page
%%   even though the Abstract spills over to the next page.

%%% AUTHOR: Title goes here.  This line is optional.  You must use it
%%   if title has footnote attached or requires nontrivial typesetting,
%%   e.g., inclusion of linebreaks to force nice layout.
\title{Induced subgraphs and tree decompositions\\
IX. Grid theorem for perforated graphs} %% please capitalize all significant words

%%% AUTHOR:
%%% List all authors. If you wish, place grant acknowledgements in \thanks.
%%% In brackets include a short tag for each author.
\author[alecu]{Bogdan Alecu\thanks{Supported by DMS-EPSRC Grant EP/V002813/1.}}
\author[chudnov]{Maria Chudnovsky\thanks{Supported by NSF-EPSRC Grant DMS-2120644 and by AFOSR grant FA9550-22-1-0083.}}
\author[hajebi]{Sepehr Hajebi}
\author[spirkl]{Sophie Spirkl\thanks{We acknowledge the support of the Natural Sciences and Engineering Research Council of Canada (NSERC), [funding reference number RGPIN-2020-03912].
Cette recherche a \'et\'e financ\'ee par le Conseil de recherches en sciences naturelles et en g\'enie du Canada (CRSNG), [num\'ero de r\'ef\'erence RGPIN-2020-03912]. This project was funded in part by the Government of Ontario. This project was funded in part by the Government of Ontario.}}

%%% AUTHOR: Abstract goes here
\begin{abstract}
The celebrated Erd\H os-P\'osa Theorem, in one formulation, asserts that for every $c\in \poi$, graphs with no subgraph (or equivalently, minor) isomorphic to the disjoint union of $c$ cycles have bounded treewidth. What can we say about the treewidth of graphs containing no \textit{induced} subgraph isomorphic to the disjoint union of $c$ cycles?

Let us call these graphs \textit{$c$-perforated}. While $1$-perforated graphs have treewidth one, complete graphs and complete bipartite graphs are examples of $2$-perforated graphs with arbitrarily large treewidth. But there are sparse examples, too: Bonamy, Bonnet, D\'{e}pr\'{e}s, Esperet, Geniet, Hilaire, Thomass\'{e} and Wesolek constructed $2$-perforated graphs with arbitrarily large treewidth and no induced subgraph isomorphic to $K_3$ or $K_{3,3}$; we call these graphs \textit{occultations}. Indeed, it turns out that a mild (and inevitable) adjustment of occultations provides examples of $2$-perforated graphs with arbitrarily large treewidth and arbitrarily large girth, which we refer to as \textit{full occultations}.

Our main result shows that the converse also holds: for every $c\in \poi$, a $c$-perforated graph has large treewidth if and only if it contains, as an induced subgraph, either a large complete graph, or a large complete bipartite graph, or a large full occultation. This distinguishes $c$-perforated graphs, among graph classes purely defined by forbidden induced subgraphs, as the first to admit a grid-type theorem incorporating obstructions other than subdivided walls and their line graphs. More generally, for all $c,o\in \poi$, we establish a full characterization of induced subgraph obstructions to bounded treewidth in graphs containing no induced subgraph isomorphic to the disjoint union of $c$ cycles, each of length at least $o+2$.
\end{abstract}
\end{frontmatter}

%%% AUTHOR: body of paper starts here
\section{Introduction}\label{sec:intro}

\subsection{Background} Graphs in this paper have finite vertex sets, no loops and no parallel edges. Let $G=(V(G),E(G))$ be a graph. For $X \subseteq V(G)$, we denote by $G[X]$ the subgraph of $G$ induced by $X$. In this paper, we use induced subgraphs and their vertex sets interchangeably. For graphs $G$ and $H$, we say $G$ \emph{contains} $H$ if $G$ has an induced subgraph isomorphic to $H$, and we say $G$ is \emph{$H$-free} if $G$ does not contain $H$. A class of graphs (which is just another term for a ``set'' of graphs quotiented by isomorphism) is \textit{hereditary} if it is closed under taking induced subgraphs.

The \textit{treewidth} of a graph $G$, denoted $\tw(G)$, is the smallest integer $w\in \poi$ for which there is a tree $T$ as well as a subtree of $T$ assigned to each vertex of $G$, such that the subtrees corresponding to adjacent vertices of $G$ intersect, and each vertex of $T$ belongs to at most $w+1$ subtrees corresponding to the vertices of $G$. This may be viewed as a natural extension of the classical ``Helly property'' for subtrees \cite{GLhelly, Hornhelly}: \textit{given a tree $T$ and a collection of pairwise intersecting subtrees of $T$, some vertex of $T$ belongs to all subtrees}. The treewidth in turn measures how few subtrees each vertex of $T$ can be guaranteed to appear in, provided only \textit{some} pairs of subtrees intersect.

The systematic study of treewidth was originated by Robertson and Seymour as a highlight of their graph minors project. Roughly speaking, graphs of small treewidth are ``measurably fattened forests.'' This brings on a wealth of nice structural
\cite{RS-GMV} and algorithmic \cite{Bodlaender1988DynamicTreewidth} properties for them, which partly explains the enduring popularity of the subject, as well. It also motivates understanding graphs of large treewidth. In particular, it is of enormous interest to identify graph classes in which large treewidth can be certified ``locally'' by a subconfiguration, still of (relatively) large treewidth yet structurally simpler that the host graph. The prototypical result of this sort is the foundational ``Grid Theorem'' of Robertson and Seymour \cite{RS-GMV}, Theorem~\ref{wallminor} below, which characterizes all subgraph-closed (and also minor-closed) graph classes of bounded treewidth. In effect, the Grid Theorem says that every graph of sufficiently large treewidth contains as a subgraph some subdivision of a highly symmetrical graph of large treewidth called a ``wall.'' For  $t\in \poi$, we denote by $W_{t \times t}$ the {\em $t$-by-$t$ wall}, which is the $t$-by-$t$ hexagonal grid and has treewidth $t$ (see Figure~\ref{fig:5x5wall}, and also \cite{wallpaper}).
\begin{theorem}[Robertson and Seymour \cite{RS-GMV}]\label{wallminor}
For every $t\in \poi$, there is a constant $w=w(t)\in \poi$ such that 
every graph of treewidth more than $w$
contains $W_{t \times t}$ as a minor, or equivalently, a subdivision of $W_{t \times t}$ as a subgraph.
\end{theorem}

\begin{figure}[t!]
\centering
\includegraphics[scale=0.7]{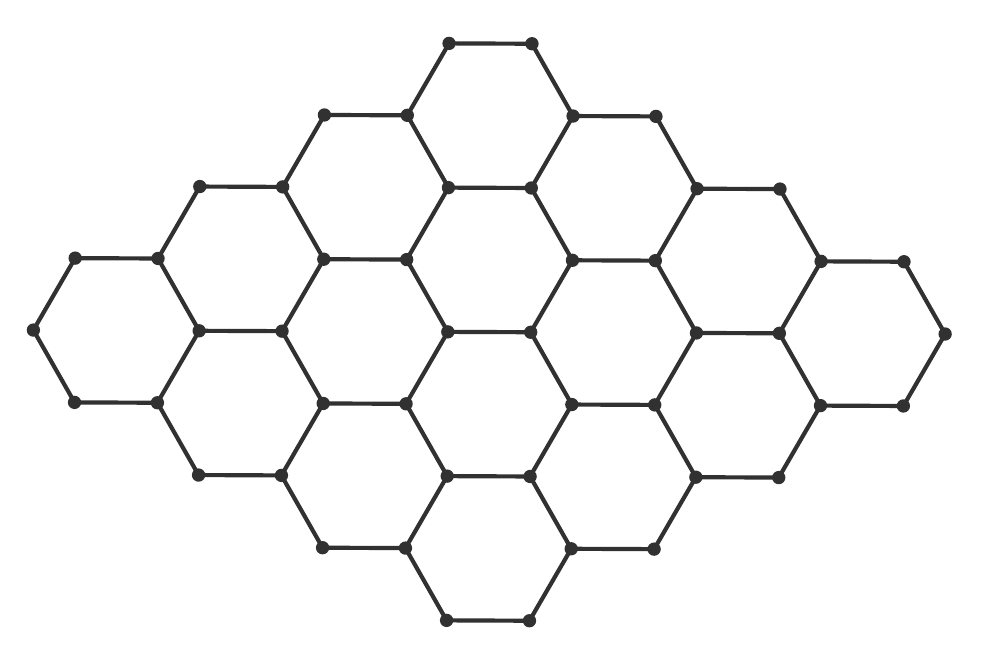}
\caption{The graph $W_{5 \times 5}$.}
\label{fig:5x5wall}
\end{figure}
What could be a ``Grid Theorem'' for induced subgraphs? This question lies at the heart of a recent trend in structural graph theory, which aims at bridging the gap between the two pillars of the subject: graph minors and the theory of induced subgraphs. Despite the immense body of research in both areas, however, the above question remains wide open, to the extent that (more or less) the only reliable clue at the moment is provided by the ``basic obstructions,'' which are known to be necessary yet quite far from sufficient for a full answer.

Let us be more precise. For $t\in \poi$, we say a graph $H$ is a {\em $t$-basic obstruction} if $H$ is either the complete graph
$K_t$, or the complete bipartite graphs $K_{t,t}$, or a subdivision of $W_{t \times t}$ mentioned above, or the line graph of
a subdivision of $W_{t \times t}$, where the {\em line graph} $L(F)$ of a graph $F$ is the graph with vertex set $E(F)$, such that two vertices of $L(F)$ are adjacent if the corresponding edges of $F$ share an end.  The basic obstructions are known to have arbitrarily large treewidth; in fact, for $t\in \poi$, the complete graph $K_{t+1}$, the complete bipartite graph $K_{t,t}$, all subdivisions of $W_{t\times t}$ and line graphs of all subdivisions of $W_{t\times t}$ all have treewidth $t$. So an exhaustive list of unavoidable induced subgraphs of graphs with large treewidth must contain, for some $t$, an induced subgraph of a $t$-basic obstruction of each type. It is also of note that there are hereditary classes for which the basic obstructions do comprise a full list of induced subgraph obstructions to bounded treewidth. We call such classes ``clean.'' In technical terms, for $t\in \poi$, we say a graph $G$ is \textit{$t$-clean} if $G$ does not contain a $t$-basic obstruction. A graph class $\mca{C}$ is \textit{clean} if for every $t\in \poi$, there is a constant $w(t)\in \poi$ (depending on $\mca{C}$) such that every $t$-clean graph in $\mca{G}$ has treewidth at most $w(t)$. For instance, Korhonen \cite{Korhonen} proved that every graph class of bounded maximum degree is clean. With Abrishami, we recently extended this to a full characterization of graphs $H$ for which the class of $H$-free graphs is clean:
\begin{theorem}[Abrishami, Alecu, Chudnovsky, Hajebi and Spirkl \cite{twvii}]\label{thm:tw7}
   The class of all $H$-free graphs is clean if and only if $H$ is a subdivided star forest, that is, a forest in which every component has at most one vertex of degree more than two.
\end{theorem}
Nevertheless, as mentioned earlier, the basic obstructions cannot carry all the load. There are at least three constructions certifying this fact. We list them below in the chronological order of discovery. They all consist of graphs with arbitrarily large treewidth which are $3$ or $4$-clean, and the third one is central to the context of this paper, which we will elaborate on later.
\begin{itemize}
    \item The so-called ``layered-wheels'' of Sintiari and Trotignon \cite{layered-wheels} consisting of even-hole-free graphs and theta-free graphs.
    \item A construction first popularized by Davies \cite{Davies2}, also found earlier by Pohoata \cite{Pohoata}.
    \item A construction by Bonamy, Bonnet, D\'{e}pr\'{e}s, Esperet, Geniet, Hilaire, Thomass\'{e} and Wesolek \cite{deathstar}, consisting of graphs with ``bounded induced cycle packing number.''
\end{itemize}
All of this goes to show that an ultimate grid-type theorem for induced subgraphs must encompass some ``non-basic obstructions,'' an exact description of which remains unknown. More dramatically, until now there was no hereditary class for which there is a grid-type theorem involving any non-basic obstruction. Our main result introduces the first such class.

\subsection{Motivation and (a taste of) the main result}
In order to formally state our main result, Theorem~\ref{mainthmgeneral}, we need quite a few definitions, and so we postpone it to Section~\ref{sec:asterism}. Instead, our goal here is to motivate Theorem~\ref{mainthmgeneral} and give the exact statement of a weakening that captures its essence. This takes a brief digression to the world of the ``Erd\H{o}s-P\'{o}sa Theorem.''

Observe that graphs with no subgraph ismorphic to a cycle are forests, which in turn are exactly the graphs with treewidth $1$. Erd\H{o}s and P\'{o}sa famously extended this simple fact to graphs excluding the disjoint union of any prescribed number of cycles as a subgraph, showing that such graphs remain proportionately close to being a forest:
\begin{theorem}[Erd\H{o}s and P\'{o}sa \cite{ErdosPosa}]\label{ErdosPosa}
    For every   $c\in \poi$, there is a constant $h \in \poi \cup \{0\}$ such that in every graph $G$ with no subgraph (or equivalently, minor) isomorphic to the disjoint union of $c$ cycles, there exists a subset $X\subseteq V(G)$ with $|X|\leq h$ such that $G\setminus X$ is a forest.
\end{theorem}
It follows from Theorem~\ref{ErdosPosa} that for every  $c\in \poi$, graphs with no subgraph isomorphic to the disjoint union of $c$ cycles have bounded treewidth (indeed, the latter statement is equivalent Theorem~\ref{ErdosPosa}; see \cite{RS-GMV}). 
It is therefore natural to ask: what can be said about the treewidth of graphs with no \textit{induced} subgraph isomorphic to the disjoint union of many cycles? For $c\in \poi$, we say a graph $G$ is \textit{$c$-perforated} if $G$ does not contain (as an induced subgraph) the disjoint union of $c$ cycles. Then, $1$-perforated graphs still have treewidth $1$. However, even for $2$-perforated graphs, bounded treewidth is far from a realistic expectation: for every   $t\in \poi$, the complete graph $K_t$ and the complete bipartite graph $K_{t,t}$ are both $2$-perforated basic obstructions. In contrast,  one may readily observe that for every $c\in \poi$,  sufficiently large subdivided walls and their line graphs are not $c$-perforated.

One may then speculate that for all  $c,t\in \poi$, every $c$-perforated graph containing neither $K_t$ nor $K_{t,t}$ has bounded treewidth. This turns out not to be true either, yet the reason is not as simple. Bonamy, Bonnet, D\'{e}pr\'{e}s, Esperet, Geniet, Hilaire, Thomass\'{e} and Wesolek \cite{deathstar} provided, for every   $s\in \poi$, a beautiful construction of $2$-perforated graphs with treewidth at least $s-1$ and containing neither $K_3$ not $K_{3,3}$ (see also \cite{NSSperforated} for another work on perforated graphs). We call these graphs  ``$s$-occultations'' and we will give a detailed description of them in a moment. But let us first state our main result: surprisingly enough, the above construction completes the picture of induced subgraph obstructions to bounded treewidth in $c$-perforated graphs for \textit{every} $c$. More precisely, we introduce a slight modification of $s$-occultations called \textit{full $s$-occultations}, and show that they are all we need to establish a grid-type theorem for $c$-perforated graphs:
\begin{theorem}\label{mainthm}
    For all   $c,t\in \poi$ and $s \in \poi \cup \{0\}$, there is a constant $\tau=\tau(c,s,t)\in \poi$ such that every $c$-perforated graph of treewidth more than $\tau$ contains either $K_t$, or $K_{t,t}$, or a full $s$-occultation.
\end{theorem}
We remark that, to the best of our knowledge, Theorem~\ref{mainthm} is the first of its kind, in the sense that it offers the first non-trivial grid-type theorem for a hereditary class which involves non-basic obstructions (we also note a result from \cite{woodcirclegraphs} which characterizes the induced subgraph obstructions to bounded treewidth in the class of ``circle graphs,'' though we view that more relatable in the context of ``vertex-minors.'') More generally, the main result of this paper, Theorem~\ref{mainthmgeneral}, is an extension of Theorem~\ref{mainthm} which completely describes, for all   $c,o\in \poi$, the obstructions to bounded treewidth in \textit{$(c,o)$-perforated graphs}, that is, graphs with no induced subgraph isomorphic to the disjoint union of $c$ cycles, each of length at least $o+2$. The extension is in fact direct enough that Theorems~\ref{mainthm} and \ref{mainthmgeneral} are identical when $o=1$.

It is also worth mentioning that 
$s$-occultations were constructed in \cite{deathstar} mainly to show that the bound in their main result, Theorem~\ref{thm:deathstarlog} below, is asymptotically sharp. In view of Theorem~\ref{thm:deathstarlog}, excluding complete graphs and complete bipartite graphs in $c$-perforated graphs \textit{does} make a difference by bringing the treewidth down to being logarithmic in the number of vertices. This is of almost equal algorithmic significance as bounded treewidth, and also puts $c$-perforated graphs on the short list of hereditary classes known to have logarithmic treewidth (see \cite{logpaper} for another example of such a class).
\begin{theorem}[Bonamy, Bonnet, D\'{e}pr\'{e}s, Esperet, Geniet, Hilaire, Thomass\'{e} and Wesolek \cite{deathstar}]\label{thm:deathstarlog}
    For all   $c,t\in \poi$, every $c$-perforated graph $G$ containing neither $K_t$ nor $K_{t,t}$ has treewidth at most $\mca{O}(\log |V(G)|)$.
\end{theorem}

\subsection{Occultations: first impression} 
In intuitive terms, an $s$-occulation is a graph consisting of $s$ pairwise non-adjacent vertices $x_1,\ldots, x_s$ and a path $L$ on $2^s+1$ vertices disjoint from $\{x_1,\ldots, x_s\}$, such that $x_1$ is adjacent to the middle vertex of $L$, and for each $2\leq i\leq s$,  the vertex $x_i$ is adjacent to the middle vertices of the ``gaps'' in $L$ created by the neighbors of $x_1,\ldots, x_i$ (see Figure~\ref{fig:deathstar}).

To make this precise, we need a few definitions. For an integer $n$, we write $[n]$ for the set all of positive integers less than or equal to $n$ (so $[n]=\emptyset$ if $n\leq 0$). Let $G$ be a graph. A \textit{stable set in $G$} is a set of pairwise non-adjacent vertices. A {\em path in $G$} is an induced subgraph of $G$ which is a path. If $P$ is a path in $G$, we write $P = p_1 \dd \cdots \dd p_k$ to mean that $V(P) = \{p_1, \dots, p_k\}$ and $p_i$ is adjacent to $p_j$ if and only if $|i-j| = 1$. We call the vertices $p_1$ and $p_k$ the \emph{ends of $P$}, and the \emph{interior of $P$} is the set $P \setminus \{p_1, p_k\}$. The \textit{length} of a path is its number of edges.

   Let us now formally define occultations and full occultations. Given $s\in \poi$, an \textit{$s$-occultation} is a graph $\mf{o}$ whose vertex set can be partitioned into a stable set $S$ in $\mf{o}$ of cardinality $s$ and a path $L$ in $\mf{o}$ with the following specifications.
\begin{enumerate}[(O1), leftmargin=15mm, rightmargin=7mm]
\item\label{O1} No two vertices in $S$ have a common neighbor in $L$. 
\item\label{O2} The ends of $L$ have no neighbor in $S$.
\item\label{O3} For some bijection $\pi:[s]\rightarrow S$, the following holds. Let $i\in [s]$ and let $P$ be a path in $L$ of non-zero length where
 \begin{itemize}
     \item every end of $P$ that is not an end of $L$ has a neighbor in $\pi([i-1])$; and
     \item no vertex in the interior of $P$ has a neighbor in $\pi([i-1])$.
 \end{itemize}
 Then $\pi(i)$ has \textbf{exactly} one neighbor in the interior of $P$. In particular, $\pi(i)$ has exactly $2^{i-1}$ neighbors in $L$ and $P$ has non-empty interior. Said differently, along $L$, $\pi(i)$ has exactly one neighbour ``between'' every two successive vertices which are either an end of $L$ or a neighbour of a vertex in $\pi([i-1])$.
 \item\label{O4} No vertex in $L$ has degree $2$ in $\mf{o}$. In particular, $L$ has length $2^s$.
\end{enumerate}
See Figure~\ref{fig:deathstar}.
\begin{figure}
\centering
\includegraphics[scale=0.8]{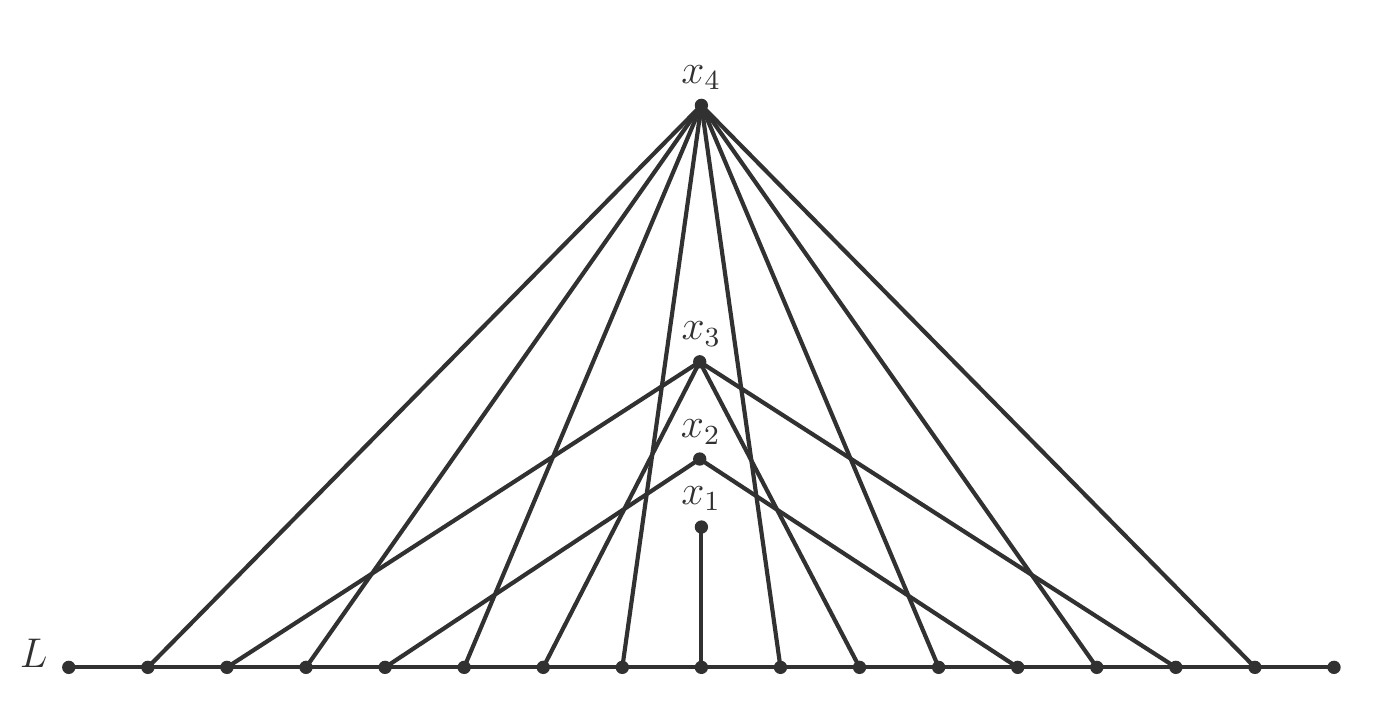}

\caption{A $4$-occultation where $\pi(i)=x_i$ for $i=1,2,3,4$.}
\label{fig:deathstar}
\end{figure}
As mentioned before, these graphs were introduced in \cite{deathstar} for the first time\footnote{Not verbatim. Also, in their version no vertex in $L$ is allowed to have degree $1$ in $\mf{o}$ (so it is obtained from our version by removing the ends of $L$). This clearly does not interfere with the validity of Theorem~\ref{thm:deathstarproperties}.} to provide a lower bound counterpart to Theorem~\ref{thm:deathstarlog}:

\begin{theorem}[Bonamy, Bonnet, D\'{e}pr\'{e}s, Esperet, Geniet, Hilaire, Thomass\'{e} and Wesolek \cite{deathstar}]\label{thm:deathstarproperties}
    For every   $s\in \poi$, every $s$-occultation is a $2$-perforated graph containing neither $K_3$ nor $K_{3,3}$ and with treewidth at least $s-1$.
\end{theorem}
Even so, the occultations defined above are not quite ready yet to be admitted as an outcome of a grid-type theorem for induced subgraphs. For instance, note that subdividing the edges of a graph $G$ does not preserve subgraphs and induced subgraphs of $G$, while it does not change the treewidth of $G$ either. This is why the ``subgraph version'' of Theorem~\ref{wallminor} as well as results concerning clean classes like Theorem~\ref{thm:tw7}  deal with subdivided walls, as opposed to the ``minor version'' of Theorem~\ref{wallminor} which involves bona fide walls. In the same vein, one may subdivide the edges of the path $L$ in an occultation $\mf{o}$, and observe that the resulting graph still satisfies Theorem~\ref{thm:deathstarproperties}. This, in other words, shows that we cannot require \ref{O4} of an occultation-like obstruction which is expected to yield bounded treewidth when forbidden (in conjunction with the basic obstructions). Similarly, when it comes to extracting an occultation as an induced subgraph in a graph of large treewidth, the word ``exactly'' in \ref{O3} is too much to ask.

Luckily, these two turn out to be the only culprits: we may define our ``full occultations'' to appear in Theorem~\ref{mainthm} as graphs with literally the same definition as occultations, except the condition \ref{O4} must be lifted, and the word ``exactly'' in \ref{O3} must be replaced by ``at least.'' More precisely, for $s  \in \poi \cup \{0\}$, a \textit{full $s$-occultation} is a graph $\mf{o}$ the vertex set of which can be partitioned into a stable set $S$ of cardinality $s$ and a path $L$ in $\mf{o}$ with the following specifications.
\begin{enumerate}[(FO1), leftmargin=17mm, rightmargin=7mm]
\item\label{FO1} No two vertices in $S$ have a common neighbor in $L$. 
\item\label{FO2} The ends of $L$ have no neighbor in $S$.
 \item\label{FO3} For some bijection $\pi:[s]\rightarrow S$, the following holds. Let $i\in [s]$ and let $P$ be a path in $L$ of non-zero length where
 \begin{itemize}
     \item each end of $P$ that is not an end of $L$ has a neighbor in $\pi([i-1])$; and
     \item no vertex in the interior of $P$ has a neighbor in $\pi([i-1])$.
 \end{itemize}
 Then $\pi(i)$ has \textbf{at least} one neighbor in the interior of $P$. In particular, $\pi(i)$ has a neighbor in $L$ and $P$ has non-empty interior.
\end{enumerate}
See Figure~\ref{fig:fulldeathstar}. 
\begin{figure}
\centering

\includegraphics[scale=0.8]{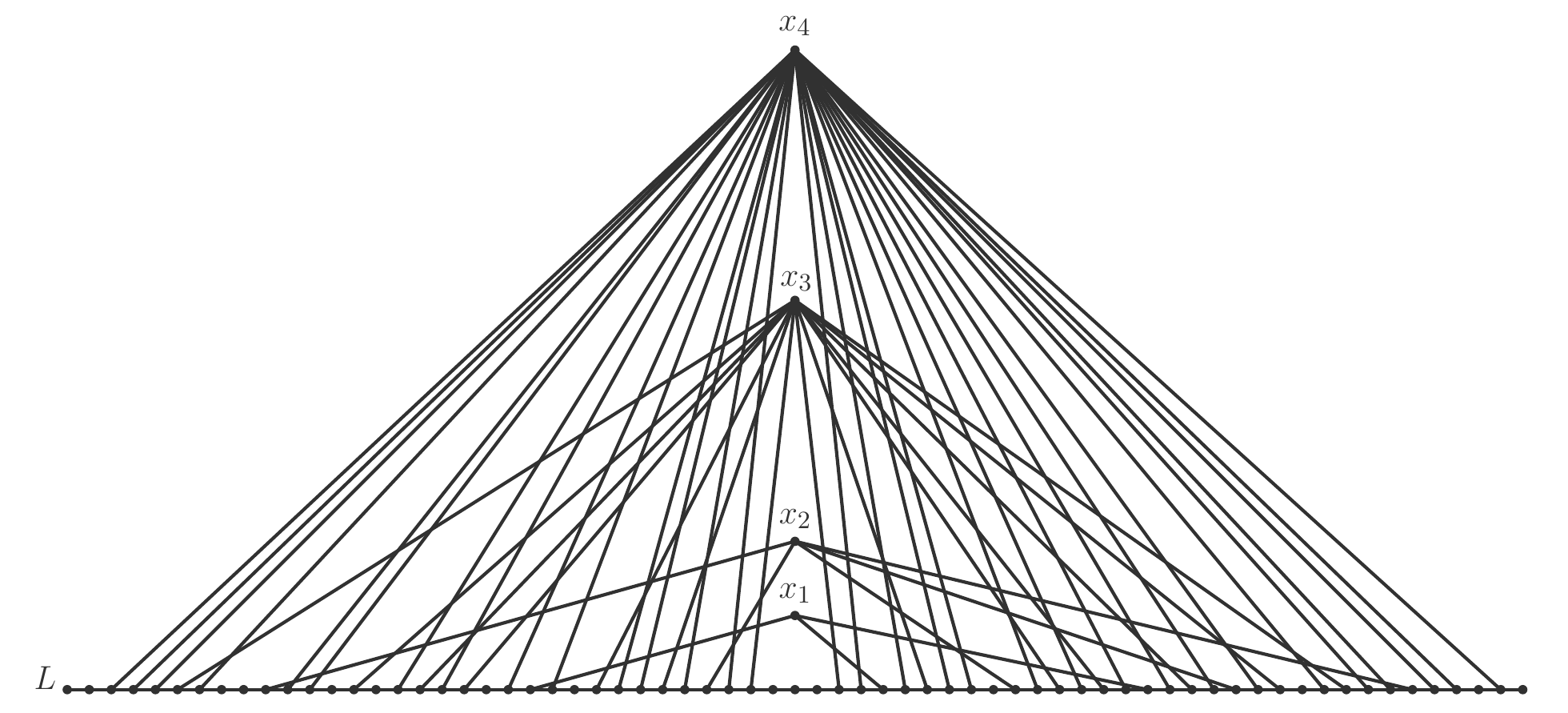}

\caption{A full $4$-occultation with $\pi(i)=x_i$ for $i=1,2,3,4$.}
\label{fig:fulldeathstar}
\end{figure}
As we will show in Theorem~\ref{thm:fulldeathstarproperties}, the  difference between occultations and full occultations is mild enough that it has almost no effect on Theorem~\ref{thm:deathstarproperties} remaining true for full occultations.
In fact, we will define full occultations once again in Section~\ref{sec:asterism} using ``asterisms,'' a term we employ extensively in this paper as it is a great technical fit to our proofs. This in particular allows for upgrading to a parametrized relaxation of full $s$-occultations called ``full $(s,o)$-occultations,'' which we then show to be the right substitute for full $s$-occultations in extending Theorem~\ref{mainthm} to Theorem~\ref{mainthmgeneral}.

 \subsection{Outline of the proof}\label{sec:outline} We now briefly describe our proof ideas and the organization of the paper. Broadly speaking, the proof of Theorem~\ref{mainthm} (or Theorem~\ref{mainthmgeneral}, rather) consists of three steps. Let $G$ be a $c$-perforated graph of sufficiently large treewidth which contains neither $K_t$ nor $K_{t,t}$. We wish to show that $G$ contains a full $s$-occultation.
 
  First, we show that $G$ contains a very large ``constellation,'' that is, a complete bipartite induced minor model where each ``branch set'' on one ``side'' is a vertex and each branch set on the other side is a path in $G$. To that end, we invoke a useful result from an earlier paper in this series \cite{twvii}, namely Theorem~\ref{noblocksmalltw_wall}, which, in essence, says that every graph class of bounded ``local connectivity'' is clean. This allows us to obtain, for a very large integer $M$, a collection of $M$ pairwise disjoint induced subgraphs $G_1,\ldots, G_M$ of $G$ such that for each $i$, one may find in $G_i$ two vertices $x_i$ and $y_i$ as well as, for another large integer $m$, a collection of $m$ pairwise internally disjoint (long) paths in $G_i$ from $x_i$ to $y_i$. From here, the proof goes on a roller coaster ride of Ramsey-type arguments to show that if $G$ excludes the desired constellation, then there are $c$ distinct $G_i$'s in each of which one may find a (long) induced cycle $H_i$, with no edges in $G$ between distinct $H_i$'s. But this violates the assumption that $G$ is $c$-perforated.

 Next, we show that if $G$ contains a huge constellation, then $G$ contains an approximate version of a full occultation called an ``interrupted asterism,'' in which \ref{FO3} is guaranteed to hold only if the ends of the path $P$ have no common neighbor in $S$. Here is an intuitive exposition of the argument: note that for each path $L$ in the ``path side'' $\mca{L}$ of the constellation, one may define a graph on the ``vertex side'' $V$ called the ``transition graph'' by making two vertices in $V$ adjacent if there is a path $R$ in $G$ between them with its interior contained in $L$ such that no other vertex in $V$ has a neighbor in the interior of $R$. Then we can easily arrange for most paths in $\mca{L}$ to impose the same transition graph $\mathsf{T}$ on $V$. There are now two possibilities. If $\mathsf{T}$ contains a matching of cardinality $c$, then the union of corresponding paths through two distinct elements of $\mca{L}$ gives a collection of $c$ pairwise disjoint cycles in $G$ with no edges between them, which is impossible. It follows that $\mathsf{T}$ admits a small vertex cover $X\subseteq V$. From the definition of the transition graph, every path between two distinct vertices in $V\setminus X$ must be ``interrupted'' by a neighbor of a vertex in $X$. But now $X$ starts to behave like the vertex $\pi(s)$ from the definition of a full $s$-occultation (up to the above-mentioned relaxation of \ref{FO3}), which we can delete and move on by induction on $s$. This of course takes much more work to make precise.

 Third, we turn the approximate full occultation obtained in the second step into a genuine one. To accomplish this, we need the ``interrupted'' asterism to be ``invaded'' as well, which means \ref{FO3} is now required to hold even if the ends of $P$ have a common neighbor in $S$. Otherwise, the union of $P$ and the common neighbor of its ends in $S$ form an induced cycle $H$ in $G$, and we can convince a ``big'' portion of the rest of the asterism to be disjoint from $H$ and have no neighbors in $H$, and then proceed by induction on $c$. This makes the last step, again, modulo a fair body of details to be checked, relatively easier than the previous two.

 The technical circumstances of our proofs, however, demand for these three steps to be taken in reverse order in Sections~\ref{sec:cherry}, \ref{sec:surgery} and \ref{sec:patch}, while Section~\ref{sec:bundle} will be our Ramsey workout to prepare for the next two sections. In Section~\ref{sec:prilim}, we set up our terminology and gather a few results from the literature to be used in subsequent sections. Section~\ref{sec:asterism} is devoted to asterisms and the statement of our main result, Theorem~\ref{mainthmgeneral}. Finally, in Section~\ref{sec:end}, we complete the proof of Theorem~\ref{mainthmgeneral}.
 
\section{Preliminaries}\label{sec:prilim}
Let $G = (V(G),E(G))$ be a graph. For $X \subseteq V(G)\cup E(G)$, $G \setminus X$ denotes the subgraph of $G$ obtained by removing $X$. Note that if $X\subseteq V(G)$, then $G \setminus X$ denotes the subgraph of $G$ induced by $V(G)\setminus X$.  

Let $P$ be a path in $G$. We denote by $P^*$ the interior of $P$ and by $\partial P$ the set of ends of $P$. Similarly, for a collection $\mca{P}$ of paths in $G$, we adapt the notations $V(\mca{P})=\bigcup_{P\in \mca{P}}V(P)$, $\mca{P}^*=\bigcup_{P\in \mca{P}}P^*$ and $\partial \mca{P}=\bigcup_{P\in \mca{P}}\partial P$.

A {\em cycle in $G$} is an induced subgraph of $G$ that is a cycle. If $H$ is a cycle in $G$, we write $H= c_1 \dd \cdots \dd c_k\dd c_1$ to mean that $V(C) = \{c_1, \dots, c_k\}$ and $c_i$ is adjacent to $c_j$ if and only if $|i-j|\in \{1,k-1\}$. The \textit{length} of a cycle is its number of edges.

Let $x\in V(G)$.
We denote by $N_G(x)$ the set of all neighbors of $x$ in $G$, and by $N_G[x]$ the set $N_G(x)\cup \{x\}$. For an induced subgraph $H$ of $G$, we define $N_H(x)=N_G(x) \cap H$, $N_H[x]=N_G[x]\cap H$. Also, for $X\subseteq G$, we denote by $N_G(X)$ the set of all vertices in $G\setminus X$ with at least one neighbor in $X$, and define $N_G[X]=N_G(X)\cup X$. 

Let $X,Y \subseteq V(G)$ be disjoint. We say $X$ is \textit{complete} to $Y$ if all edges with an end in $X$ and an end in $Y$ are present in $G$, and $X$ is \emph{anticomplete}
to $Y$ if no edges between $X$ and $Y$ are present in $G$. 

By a \textit{subdivision} of $G$, we mean a graph $G'$ obtained from $G$ by replacing the edges of $G$ by pairwise internally disjoint paths of non-zero length between the corresponding ends. Let $r  \in \poi \cup \{0\}$. A $(\leq r)$-\textit{subdivision} of $G$ is a subdivision of $G$ in which the path replacing each edge has length at most  $r+1$.

Let us now mention a few results from the literature, beginning with two versions of Ramsey's Theorem. Given a set $X$ and $q  \in \poi \cup \{0\}$, we denote 
by $2^X$ the power set of $X$ and by $\binom{X}{q}$ the set of all $q$-subsets of $X$.

\begin{theorem}[Ramsey \cite{multiramsey}]\label{multiramsey}
For all   $n \in \poi \cup \{0\}$ and $q,r\in \poi$, there is a constant $\rho(n,q,r)\in \poi$ with the following property. Let  $U$ be a set of cardinality at least $\rho(n,q,r)$ and let $W$ be a non-empty set of cardinality at most $r$. Let $\Phi:\binom{U}{q}\rightarrow W$ be a map. Then there exist $i\in W$ and $Z\subseteq U$ with $|Z|=n$ such that for every $q$-subset $A$ of $Z$, we have $\Phi(A)=i$.
\end{theorem}

\begin{theorem}[Graham, Rothschild and Spencer \cite{productramsey}, see also \cite{KP}]\label{productramsey}
For all   $n \in \poi \cup \{0\}$ and $q,r\in \poi$, there is a constant $\nu(n,q,r)\in \poi$ with the following property. Let $U_1,\ldots, U_n$ be $n$ sets, each of cardinality at least $\nu(n,q,r)$ and let $W$ be a non-empty set of cardinality at most $r$. Let $\Phi$ be a map from the Cartesian product $U_1\times \cdots \times U_n$ into $W$. Then there exist $i\in W$ and $Z_j\subseteq U_j$ with $|Z_j|=q$ for each $j\in [n]$, such that for every $z\in Z_1\times \cdots\times Z_n$, we have $\Phi(z)=i$.
\end{theorem}

We also need the following result, which is a direct consequence of Theorem 5 in  \cite{dvorak}. See also Theorem 3 in \cite{lozin}, which has been discovered after --  and follows from -- Theorem 5 in  \cite{dvorak}.

\begin{theorem}[Dvo\v{r}\'{a}k, see Theorem 5 in \cite{dvorak}, Lozin and Razgon, see Theorem 3 in \cite{lozin}]\label{dvorak}
For every graph $H$ and all   $d \in \poi \cup \{0\}$ and $t\in \poi$, there is a constant $m=m(H,d,t)\in \poi$ with the following property. Let $G$ be a graph with no induced subgraph isomorphic to a subdivision of $H$. Assume that $G$ contains a $(\leq d)$-subdivision of $K_m$ as a subgraph. Then $G$ contains either $K_t$ or $K_{t,t}$.
\end{theorem}

Let $k\in \poi$ and let $G$ be a graph. A \textit{strong $k$-block} in $G$ is a set $B$ of at least $k$ vertices in $G$ such that for every $2$-subset $\{x,y\}$ of $B$, there exists a collection $\mca{P}_{\{x,y\}}$ of at least $k$ distinct and pairwise internally disjoint paths in $G$ from $x$ to $y$, where for every two distinct $2$-subsets $\{x,y\}, \{x',y'\}\subseteq B$ of $G$, we have $V(\mca{P}_{\{x,y\}})\cap V(\mca{P}_{\{x',y'\}})=\{x,y\}\cap \{x',y'\}$.

In \cite{twvii}, with Abrishami we proved the following:

\begin{theorem}[Abrishami, Alecu, Chudnovsky, Hajebi and Spirkl \cite{twvii}]\label{noblocksmalltw_wall}
  For every $k\in \poi$, the class of all graphs with no strong $k$-block is clean. Equivalently, for all $k,t\in \poi$, there is a constant $w=w(k,t)\in \poi$ such that every $t$-clean graph with no strong $k$-block has treewidth at most $w$.
\end{theorem}

Given a graph $G$ and $d\in \poi$, a \textit{$d$-stable set} in $G$ is a set $S\subseteq V(G)$ such that for every two distinct vertices $u,v\in S$, there is no path of length at most $d$ in $G$ from $u$ to $v$. It is proved in \cite{twvii} that:

\begin{theorem}[Abrishami, Alecu, Chudnovsky, Hajebi and Spirkl \cite{twvii}]\label{distance}
    For all   $d,k\in \poi$ and $m\geq 2$, there is a constant $\kappa=\kappa(d,k,m)\in \poi$ with the following property. Let $G$ be a graph and $B$ be a strong $\kappa$-block in $G$. Assume that $G$ does not contain a $(\leq d)$-subdivision of $K_m$ as a subgraph. Then there exists $A\subseteq G$ with $B'\subseteq B\setminus A$ such that $B'$ is both a strong $k$-block and a $d$-stable set in $G\setminus A$.
\end{theorem}

For   $c,o\in \poi$, a graph $G$ is said to be \textit{$(c,o)$-perforated} if there are no $c$ pairwise disjoint cycles in $G$, each of length at least $o+2$. It follows that $G$ is $c$-perforated if and only if $G$ is $(c,1)$-perforated. Also, one may observe that subdivisions of $W_{5co\times 5co}$ and their line graphs are not $(c,o)$-perforated. It follows from Theorem~\ref{noblocksmalltw_wall} that:

\begin{corollary}\label{noblocksmalltw_polycyclefree}
For all   $c,k,o,t\in \poi$, there is a constant $\xi=\xi(c,k,o,t)$ such that every $(c,o)$-perforated graph of treewidth more than $\xi$ contains either $K_t$, or $K_{t,t}$, or a strong $k$-block.
  \end{corollary}

  A vertex $v$ in a graph $G$ is said to be a \textit{branch vertex} if $v$ has degree more than two. By a {\em caterpillar} we mean a tree $T$ with maximum degree three such that there is a path $P$ in $T$ containing all branch vertices of $T$ (this is not standard for two reasons: a caterpillar is usually allowed to be of arbitrary maximum degree, and the path $P$ from the definition often contains all vertices of degree more than one). By a \textit{subdivided star} we mean a graph isomorphic to a subdivision of the complete bipartite graph $K_{1,\delta}$ for some $\delta\geq 3$. In other words, a subdivided star is a tree with exactly one branch vertex, which we call its \textit{root}. For a graph $H$, a vertex $v$ of $H$ is said to be \textit{simplicial} if $N_H(v)$ is a clique in $H$. We denote by $\mca{Z}(H)$ the set of all simplicial vertices of $H$. Note that for every tree $T$, $\mca{Z}(T)$ is the set of all leaves of $T$. Also, if $H$ is the line graph of a tree $T$, then $\mca{Z}(H)$ is the set of all vertices in $H$ corresponding to the edges in $T$ which are incident with the leaves of $T$.
  
  The following is proved in \cite{twvii}; see Figure~\ref{fig:connectifier} (see also \cite{thesis} for a self-contained proof with the explicit bound).
  
\begin{theorem}[Abrishami, Alecu, Chudnovsky, Hajebi and Spirkl \cite{twvii}]\label{connectifier}

 Let $h\in \poi$ and let $G$ be a graph. Let $S \subseteq V(G)$  with $|S|\geq h^{8h^4}$ 
  such that $S$ is contained in a connected component of $G$. Then there is an induced subgraph $H$ of $G$ with $|V(H)\cap S|=h$ for which one of the following holds.

\begin{enumerate}[\rm (a)]
\item\label{thm:minimalconnectedgeneral_a} $H$ is a path with ends in $S$.

\item\label{thm:minimalconnectedgeneral_b} $H$ is a subdivided star with root $r$ such that $\mca{Z}(H)\subseteq V(H)\cap S\subseteq \mca{Z}(H)\cup \{r\}$.

\item\label{thm:minimalconnectedgeneral_c} $H$ is the line graph of a subdivided star with $V(H)\cap S=\mca{Z}(H)$.

\item\label{thm:minimalconnectedgeneral_d} $H$ is either a caterpillar or the line graph of a caterpillar, and $V(H)\cap S=\mca{Z}(H)$.
    \end{enumerate}
  \end{theorem}
  \begin{figure}[t!]
      \centering
      \includegraphics[scale=0.8]{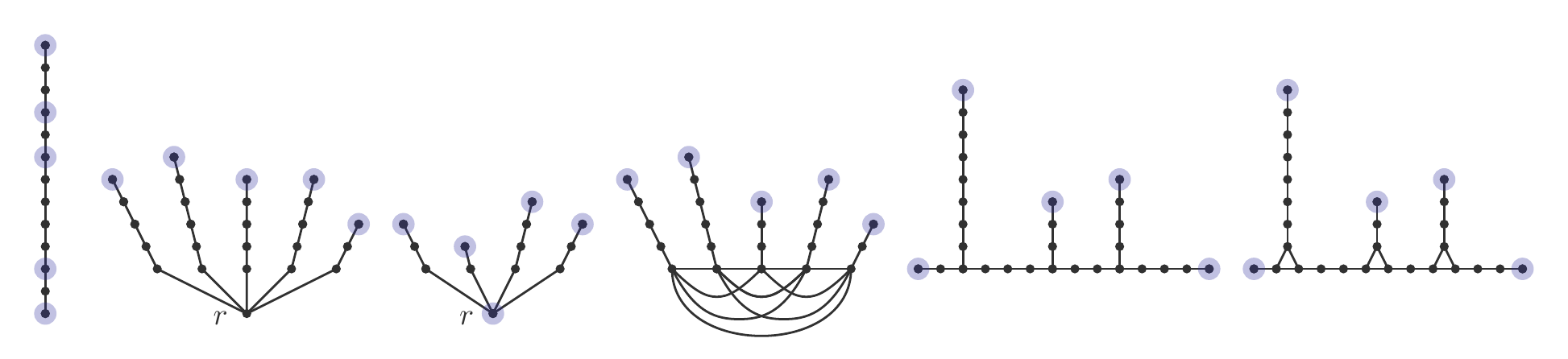}
      \caption{Outcomes of Theorem~\ref{connectifier} when $h=5$. From left to right: the first for \ref{thm:minimalconnectedgeneral_a}, the second for  \ref{thm:minimalconnectedgeneral_b} when $r\notin S$, the third for  \ref{thm:minimalconnectedgeneral_b} when $r\in S$, the fourth for \ref{thm:minimalconnectedgeneral_c}, and the fifth and the sixth for \ref{thm:minimalconnectedgeneral_d}.}
      \label{fig:connectifier}
  \end{figure}

\section{Asterisms}\label{sec:asterism}

In this section, we state our main result in formal terms, beginning with a definition which is of critical importance in the remainder of this paper. Let $G$ be a graph and let $s \in \poi \cup \{0\}$. By an \textit{$s$-asterism in $G$} we mean a pair $\mf{a}=(S_{\mf{a}}, L_{\mf{a}})$ where $S_{\mf{a}}$ is a stable set in $G$ with $|S_{\mf{a}}|=s$ and $L_{\mf{a}}$ is a path in $G\setminus S_{\mf{a}}$, such that every vertex in $S_{\mf{a}}$ has a neighbor in $L_{\mf{a}}^*$ and $S_{\mf{a}}$ is anticomplete to $\partial L_{\mf{a}}$.  An \textit{ordered $s$-asterism in $G$} consists of an $s$-asterism $\mf{a}$ along with a bijection $\pi_{\mf{a}}:[s]\rightarrow S_{\mf{a}}$. There are several convenient notations and definitions associated with asterisms, and we prefer to collect them all below. See also Figure~\ref{fig:asterism}.

Fix an (ordered) $s$-asterism $\mf{a}$ in $G$. We denote by $V(\mf{a})$ the vertex set $S_{\mf{a}}\cup V(L_{\mf{a}})$. By an \textit{$\mf{a}$-route} we mean a path in $G$ with ends in $S_{\mf{a}}$ and interior contained in $L_{\mf{a}}$. An $\mf{a}$-route $R$ is \textit{minimal} if there is no $\mf{a}$-route $Q$ with $Q^*$ properly contained in  $R^*$. We say $\mf{a}$ is \textit{ample} if no two vertices in $S_{\mf{a}}$ have a common neighbor in $L_{\mf{a}}$. More generally, given $d \in \poi \cup \{0\}$, we say $\mf{a}$ is \textit{$d$-ample} if every $\mf{a}$-route has length at least $d+2$. For instance, every asterism in $G$ is $0$-ample, and $\mf{a}$ is $1$-ample if and only if $\mf{a}$ is ample.
 
 By an \textit{$\mf{a}$-piece} we mean a path $P$ in $L_{\mf{a}}$ of non-zero length such that every end of $P$ that is not an end of $L_{\mf{a}}$ has a neighbor in $S_{\mf{a}}$, and $P^*$ is anticomplete to $S_{\mf{a}}$. An $\mf{a}$-piece $P$ is \textit{internal} if $P$ is contained in $L_{\mf{a}}^*$; otherwise $P$ is \textit{external} (so the entire path $L_{\mf{a}}$ is the only external $\mf{a}$-piece if $S_{\mf{a}}=\emptyset$, there are exactly two external $\mf{a}$-pieces if $S_{\mf{a}}\neq\emptyset$). An $\mf{a}$-piece $P$ is said to be \textit{open} if the ends of $P$ have no common neighbor in $S_{\mf{a}}$, and \textit{closed} if the ends of $P$ have a common neighbor in $S_{\mf{a}}$. It follows that every external $\mf{a}$-piece is open. In fact, one may observe that an $\mf{a}$-piece $P$ is open if and only if either $P$ is external or $P$ is the interior of a minimal $\mf{a}$-route.
 \begin{figure}
\centering \includegraphics{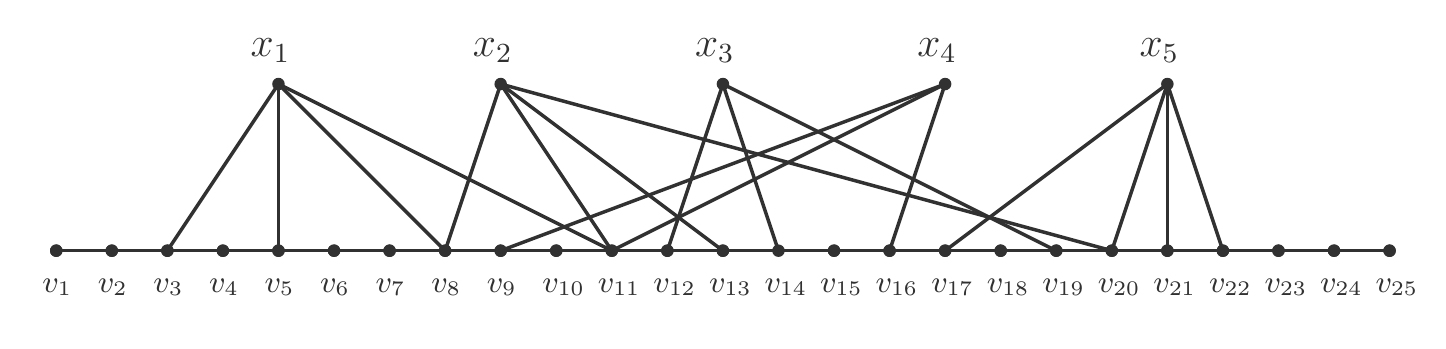}
 \caption{A $5$-asterism $\mf{a}$ with $S_{\mf{a}}=\{x_1,\ldots, x_5\}$ and $L_{\mf{a}}=v_1\dd \cdots \dd v_{25}$. Note, for instance, that  $x_2\dd v_{13}\dd v_{14}\dd v_{15}  \dd v_{16}\dd x_4$ is an $\mf{a}$-route that is not minimal, and $x_2\dd v_{13}\dd v_{14}\dd x_3$ is an $\mf{a}$-route that is minimal. Also, there are exactly 15 $\mf{a}$-pieces; 13 of which internal and $v_1\dd v_{2}\dd v_{3}$ and $v_{22}\dd v_{23}\dd v_{24}\dd v_{25}$ are the two external $\mf{a}$-pieces. For instance, $v_{17}\dd v_{18}\dd v_{19}$ is an internal $\mf{a}$-piece which is open, and $v_{5}\dd v_{6}\dd v_{7}\dd v_{8}$ is a closed $\mf{a}$-piece (which is necessarily internal).}
 \label{fig:asterism}
 \end{figure}

For every subset $X$ of $S_{\mf{a}}$, we denote by $\mf{a}|X$ the (ordered) $|X|$-asterism $(X,L_{\mf{a}})$ such that, in case $\mf{a}$ is ordered, for all distinct $x,x'\in X$, we have $\pi_{\mf{a}|X}(x)>\pi_{\mf{a}|X}(x')$ if and only if $\pi_{\mf{a}}(x)>\pi_{\mf{a}}(x')$.

Assume that $\mf{a}$ is ordered. For every integer $i$, we write $\mf{a}^{i}=\mf{a}|\pi_{\mf{a}}([i])$. We say $\mf{a}$ is \textit{interrupted} if for every $i\in [s]$:
\medskip

\begin{enumerate}[(INT), leftmargin=19mm]
\item\label{INT} $\pi_{\mf{a}}(i)$ has at least one neighbor in each open $\mf{a}^{i-1}$-piece. 
\end{enumerate}
\medskip

Also, we say $\mf{a}$ is \textit{invaded} if for every $i\in [s]$:
\medskip

\begin{enumerate}[(\rm{INV}), leftmargin=19mm]
    \item\label{INV} $\pi_{\mf{a}}(i)$ has at least one neighbor in each closed $\mf{a}^{i-1}$-piece.
\end{enumerate}
\medskip

Let us now re-define full occultations in terms of asterisms. Given a graph $G$ and $s \in \poi \cup \{0\}$, a \textit{full $s$-occultation in $G$} is an ample ordered $s$-asterism $\mf{o}$ in $G$ which is both interrupted and invaded, that is, for every $i\in [s]$, $\pi_{\mf{o}}(i)$ has at least one neighbor in the interior of every $\mf{o}^{i-1}$-piece. It is straightforward to check that this definition is equivalent to the one mentioned in Section~\ref{sec:intro}, modulo the convenient nuance of viewing full occultations as induced subgraphs of a fixed graph $G$ with their ``$(S,L)$ partition'' given, rather than independently defined graphs.

As promised, our main result is an extension of Theorem~\ref{mainthm} to a characterization of obstructions to bounded treewidth in $(c,o)$-perforated graphs for every $o\in \poi$. It turns out that Theorem~\ref{mainthm} remains true for $(c,o)$-perforated graphs at the cost of ``relaxing the invadedness of full occultations.'' To be more precise, given a graph $G$,  $s \in \poi \cup \{0\}$ and $o\in \poi$, we say an ordered $s$-asterism $\mf{a}$ in $G$ is \textit{$o$-invaded} if for every $i\in [s]$:
\medskip

\begin{enumerate}[(OI), leftmargin=17mm, rightmargin=7mm]
    \item\label{OI} $\pi_{\mf{a}}(i)$ has at least one neighbor in each closed $\mf{a}^{i-1}$-piece of length at least $o$.
\end{enumerate}
\medskip

So $\mf{a}$ is invaded if and only if $\mf{a}$ is $1$-invaded. By a \textit{full $(s,o)$-occultation in $G$}, we mean an ample ordered $s$-asterism $\mf{o}$ in $G$ which is both interrupted and $o$-invaded. In particular, $\mf{o}$ is a full $(s,1)$-occultation in $G$ if and only if $\mf{o}$ is a full $s$-occultation in $G$. There is also an analogue of Theorem~\ref{thm:deathstarproperties} for full $(s,o)$-occultations:

\begin{theorem}\label{thm:fulldeathstarproperties}
   Let $s \in \poi \cup \{0\}$, $o\in \poi$ and let $G$ be a graph. Then the following hold.
    \begin{enumerate}[\rm (a)]
       \item\label{fulldeathstarproperties_a} Let $g \in \poi \cup \{0\}$ be and let $\mf{o}$ be a full $(g+s,o)$-occultation in $G$. Then $\mf{o}^s$ is a full $(s,o)$-occultation in $G$ and $G[V(\mf{o}^s)]$ has girth more than $g+2$. 
        \item\label{fulldeathstarproperties_b} Let $\mf{o}$ be a full  $(s,o)$-occultation in $G$. Then $G[V(\mf{o})]$ is $(2,o)$-perforated and has treewidth at least $s-1$.
    \end{enumerate}
\end{theorem}
\begin{proof}
We leave the proof of \ref{fulldeathstarproperties_a} to the reader as it is easy, and will only give a proof \ref{fulldeathstarproperties_b}, which is almost identical to the proof of Theorem~\ref{thm:deathstarproperties} in \cite{deathstar}. The assertion is trivial for $s=0$. Let $s\in \poi$ and let $G'=G[V(\mf{o})]$. Then $G'$ contains a subdivision of an $s$-occultation as a subgraph, and so by Theorem~\ref{thm:deathstarproperties}, $G'$ has treewidth at least $s-1$. Now suppose for a contradiction that there are two disjoint and anticomplete cycles $H_1$ and $H_2$ in $G'$, each of length at least $o+2$. It follows that for each $i\in \{1,2\}$, both $H_i\cap S_{\mf{o}}$ and $H_i\cap L_{\mf{o}}$ are non-empty. Let $i\in [s]$ be maximum with $\pi_{\mf{o}}(i)\in H_1\cup H_2$, and without loss of generality, assume that $\pi_{\mf{o}}(i)\in H_2$. Let $P$ be a connected component of $H_1[V(H_1)\cap V(L_{\mf{o}})]$. Then we may write $\partial P=\{u,v\}$ such that for some choice of $j,k\in [i-1]$, $u$ is adjacent $\pi_{\mf{o}}(j)$ in $G$ and $v$ is adjacent $\pi_{\mf{o}}(k)$ in $G$. It follows that either $P$ contains the interior of minimal $\mf{o}^{i-1}$-route, that is, $P$ contains an open $\mf{o}^{i-1}$-piece $P'$, or $j=k$ and $P$ is a closed $\mf{o}^{i-1}$-piece. In the former case, since $\mf{o}$ is interrupted, by \ref{INT}, $\pi_{\mf{o}}(i)\in V(H_2)$ has a neighbor in $P'\subseteq P\subseteq H_1$. But this violates the assumption that $H_1$ and $H_2$ are anticomplete in $G$. Also, in the latter case, we have $H_1=\pi_{\mf{o}}(j)\dd u\dd P\dd v\dd \pi_{\mf{o}}(j)$. Since $H_1$ has length at least $o+2$, it follows $P$ has length at least $o$. More precisely, $P$ is a closed $\mf{o}^{i-1}$-piece of length at least $o$. But then since $\mf{o}$ is $o$-invaded, by \ref{OI}, $\pi_{\mf{o}}(i)\in V(H_2)$ has a neighbor in $P\subseteq H_1$, again a contradiction with $H_1$ and $H_2$ being anticomplete. This completes the proof of Theorem~\ref{thm:fulldeathstarproperties}.
\end{proof}

We can now state the main result of this paper:
\begin{theorem}\label{mainthmgeneral}
    For all   $c,o,t\in \poi$ and $s \in \poi \cup \{0\}$, there is a constant $\tau=\tau(c,o,s,t)\in \poi$ such that for every $(K_t, K_{t,t})$-free $(c,o)$-perforated graph $G$ of treewidth more than $\tau$, there is a full $(s,o)$-occultation in $G$.
\end{theorem}
In view of Theorem~\ref{thm:fulldeathstarproperties}\ref{fulldeathstarproperties_b}, Theorem~\ref{mainthmgeneral} provides a grid-like theorem for the class of $(c,o)$-perforated graphs for all $c,o\in \poi$. Moreover, using Theorem~\ref{thm:fulldeathstarproperties}\ref{fulldeathstarproperties_a}, we may arrange for the full occultation in the outcome of Theorem~\ref{mainthmgeneral} to be of arbitrarily large girth (as opposed to complete and complete bipartite graphs which have small girth):

\begin{corollary}\label{mainthmgeneralcorollary}
    For all   $c,o,t\in \poi$ and $g,s \in \poi \cup \{0\}$, there is a constant $\tau=\tau(c,g,o,s,t)\in \poi$ such that for every $(K_t, K_{t,t})$-free  $(c,o)$-perforated graph $G$, there is a full $(s,o)$-occultation $\mf{o}$ in $G$ such that $G[V(\mf{o})]$ has girth more than $g+2$.
\end{corollary}

 \section{The cherry on top}\label{sec:cherry}

Recall that every full occultation is both ample and interrupted. Our goal in this section is to show that in perforated graphs, a qualitative converse holds, too:
\begin{theorem}\label{thm:interrupted_to_occultation}
     Let $c,o\in \poi$, $s \in \poi \cup \{0\}$ and let $G$ be a $(c,o)$-perforated graph. Assume that there exists a $2$-ample, interrupted ordered $s^c$-asterism $\mf{a}$ in $G$. Then there exists a full $(s,o)$-occultation $\mf{o}$ in $G$ with $S_{\mf{o}}\subseteq S_{\mf{a}}$ and $L_{\mf{o}}\subseteq L_{\mf{a}}$.
 \end{theorem}

The proof of Theorem~\ref{thm:interrupted_to_occultation} calls for a few definitions and lemmas. Let $G$ be a graph and let $x\in V(G)$. Let $\mf{a}'$ be an (ordered) asterism in $G$ with $V(\mf{a}')\subseteq V(G)\setminus \{x\}$. We say $x$ is a \textit{cherry on top of $\mf{a}'$ in $G$} if:
\begin{enumerate}[(CH1), leftmargin=19mm, rightmargin=7mm]
    \item\label{CH1} $x$ is anticomplete in $G$ to $\partial L_{\mf{a}'}$; and
    \item\label{CH2} $x$ has a neighbor in every open $\mf{a}'$-piece. In particular, $x$ has a neighbor in $L_{\mf{a}'}$.
\end{enumerate}

\begin{figure}[t!]
    \centering
    \includegraphics[scale=1]{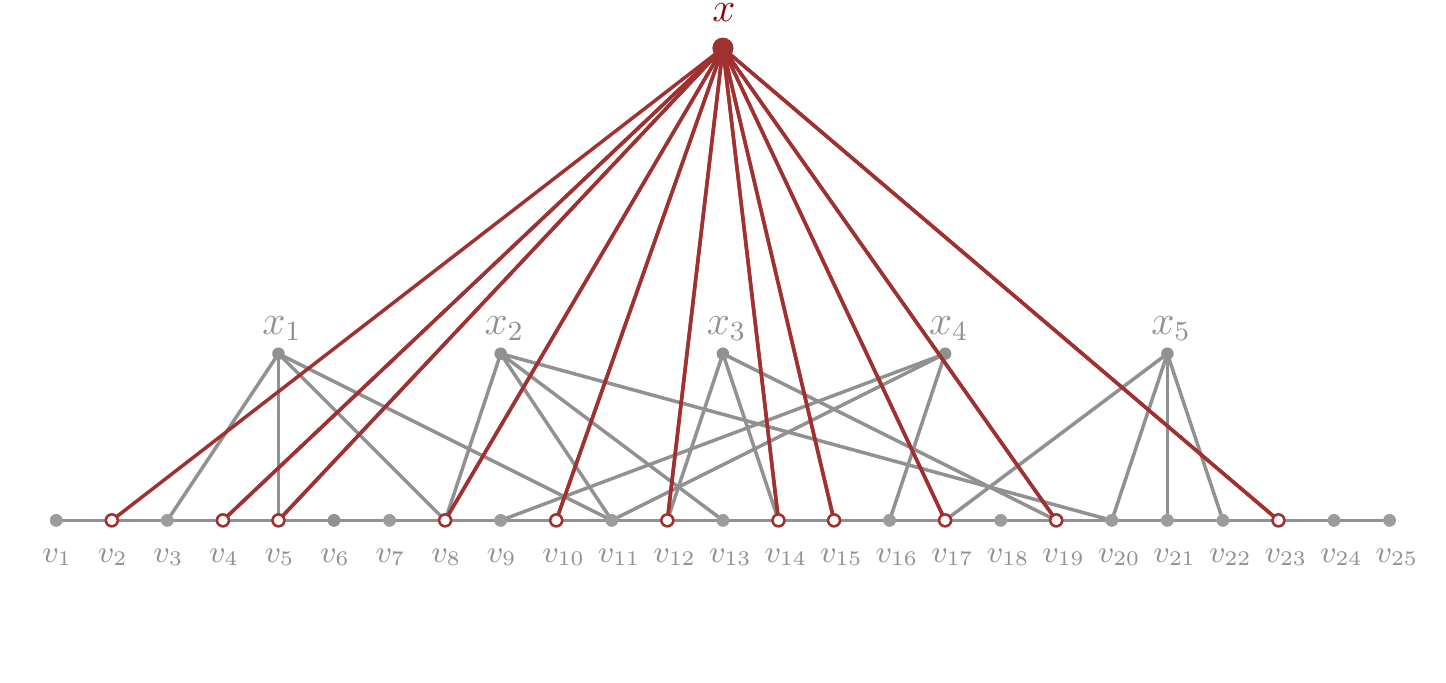}
    \caption{The vertex $x$ is a cherry on top of the asterism $\mf{a}$ from Figure~\ref{fig:asterism}.}
    \label{fig:cherry}
\end{figure}

See Figure~\ref{fig:cherry}. Now, let $G$ be graph, let $x\in V(G)$, let $s' \in \poi \cup \{0\}$ and let $\mf{a}'$ be an ordered $s'$-asterism in $G$ with $V(\mf{a}')\subseteq V(G)\setminus \{x\}$ such that $x$ is a cherry on top of $\mf{a}'$ in $G$. Then by \ref{CH1} and \ref{CH2}, $\cher(\mf{a}',x)=(S_{\mf{a}'}\cup \{x\},L_{\mf{a}'})$ is an ordered $(s'+1)$-asterism in $G$ with $\pi_{\cher(\mf{a}',x)}(s'+1)=x$ and $\pi_{\cher(\mf{a}',x)}(i)=\pi_{\mf{a}'}(i)$ for all $i\in [s']$. In addition, by \ref{CH2}, $\cher(\mf{a}',x)$ is interrupted if and only if $\mf{a}'$ is interrupted. 

One may note that the notion of a ``cherry on top'' arises naturally from viewing interrupted ordered asterisms as ordered asterisms which can be constructed by ``successively adding cherries on top.'' Said more carefully, it is straightforward to observe that given a graph $G$ and $s \in \poi \cup \{0\}$, an ordered $s$-asterism $\mf{a}$ in $G$ is interrupted if and only if for every $i\in [s]$, $\pi_{\mf{a}}(i)$ is a cherry on top of $\mf{a}^{i-1}$ in $G$.

Accordingly, our first lemma,  which we will use both here and in Section~\ref{sec:surgery}, provides a tool for growing interrupted ``sub-asterisms'' of a given asterism by adding one cherry on top at a time. It roughly says the following: let $G$ be a graph, let $\mf{a}$ be an asterism in $G$ which is ``ample enough'' and let $x\in S_{\mf{a}}$. Let $S'\subseteq S_{\mf{a}}\setminus \{x\}$ such that $x$ is a cherry on top of $\mf{a}|S'=(S', L_{\mf{a}})$ in $G$. Then the same is also inherited by an ``optimally'' chosen subpath $L'\subseteq L_{\mf{a}}$, that is, $x$ is a cherry on top of the asterism $(S', L')$ in $G$. To make this notion of ``optimality'' precise, let $G$ be a graph, let $\mf{a}$ be an asterism in $G$ which may or may not be ordered, let $x\in S_{\mf{a}}$ and let $s' \in \poi \cup \{0\}$. By an \textit{$(\mf{a},x,s')$-candidate in $G$} we mean an interrupted ordered $s'$-asterism $\mf{a}'$ in $G$ with $S_{\mf{a}'}\subseteq S_{\mf{a}}\setminus \{x\}$ and $L_{\mf{a}'}\subseteq L_{\mf{a}}$, such that:

\begin{enumerate}[(CA),leftmargin=19mm, rightmargin=7mm]
    \item\label{CA} for every interrupted ordered $s'$-asterism $\mf{a}''$ in $G$ with $S_{\mf{a}'}=S_{\mf{a}''}$, $\pi_{\mf{a}'}=\pi_{\mf{a}''}$ and $L_{\mf{a}'}\subsetneq L_{\mf{a}''}\subseteq L_{\mf{a}}$, there exists $i\in [s'-1]$ such that $\pi_{\mf{a}'}(i)=\pi_{\mf{a}''}(i)$ has a neighbor in $L_{\mf{a}''}\setminus L_{\mf{a}'}$ (and so $s'\geq 2$).
\end{enumerate}

In other words, the path $L_{\mf{a}'}$ is maximal if one requires $N_{L_{\mf{a}'}}(S_{\mf{a}'}\setminus \{\pi_{\mf{a}'}(s')\})$ to remain the same. We deduce that:

\begin{lemma}\label{lem:interrupted_top}
    Let $s\in \poi$. Let $G$ be a graph and let $\mf{a}$ be a $2$-ample $s$-asterism in $G$. Let $x\in S_{\mf{a}}$ and let $\mf{a}'$ be an $(\mf{a},x,s-1)$-candidate in $G$. Assume that $x$ is a cherry on top of $\mf{a}|S_{\mf{a}'}$ in $G$. Then $x$ is a cherry on top of $\mf{a'}$ in $G$.  Consequently, $\cher(\mf{a}',x)$ is a $2$-ample, interrupted ordered $s$-asterism in $G$ with $S_{\cher(\mf{a}',x)}\subseteq S_{\mf{a}}$ and $L_{\cher(\mf{a}',x)}\subseteq L_{\mf{a}}$.
 \end{lemma}
 \begin{proof}
We need to show that $\mf{a}'$ and $x$ satisfy \ref{CH1} and \ref{CH2}. Let us begin with the following:
 
     \sta{\label{st:ends_work}Let $u\in \partial L_{\mf{a}'}\setminus \partial L_{\mf{a}}$. Then the unique neighbor of $u$ in $L_{\mf{a}}\setminus L_{\mf{a}'}$ has a neighbor in $\pi_{\mf{a}'}([s-2])$.}

Suppose not. Let $u'$ be the unique neighbor of $u$ in $L_{\mf{a}}\setminus L_{\mf{a}'}$. Then $u'$ is anticomplete to $\pi_{\mf{a}'}([s-2])$. Since $u$ is not an end of $L_{\mf{a}}$ and $\mf{a}'$ is an $(\mf{a},x,s-1)$-candidate in $G$, it follows from \ref{CA} that $S_{\mf{a}'}\neq \emptyset$ (indeed, we have $s\geq 3$).  Also, $u'$ is adjacent to $x'=\pi_{\mf{a}'}(s-1)$, as otherwise $\mf{a}''=(S_{\mf{a}'}, L_{\mf{a}'}\cup \{u'\})$ is an interrupted ordered $(s-1)$-asterism in $G$ violating \ref{CA}. Let $u''$ be the end of $L_{\mf{a}}$ for which $u\dd L_{\mf{a}}\dd u''$ contains $u'$. Since $x'$ is adjacent to $u'$ and $x'$ is not adjacent to $u''$, traversing $u'\dd L_{\mf{a}}\dd u''$ from $u'$ to $u''$, we may choose the first vertex $w$ which is not adjacent to $x'$. It follows that $u'\neq w$ but $u''$ and $w$ might be the same. Let $w'$ be the unique neighbor of $w$ in $u'\dd L_{\mf{a}}\dd w$ (so $u'$ and $w'$ might be the same).  Since $\mf{a}$ is $2$-ample and $x'$ is complete to $u'\dd L_{\mf{a}}\dd w'$, it follows that $\pi_{\mf{a}'}([s-2])$ is anticomplete to $u'\dd L_{\mf{a}}\dd w$. Now, let $v$ be the end of $L_{\mf{a'}}$ distinct from $u$ and let $L''=v\dd L_{\mf{a}}\dd w$. Then $\mf{a}''=(S_{\mf{a}'}, L'')$ is an interrupted ordered $(s-1)$-asterism in $G$ with $S_{\mf{a}'}=S_{\mf{a}''}$, $\pi_{\mf{a}'}=\pi_{\mf{a}''}$ and  $L_{\mf{a}'}\subsetneq L''=L_{\mf{a}''}\subseteq L_{\mf{a}}$, for which $\pi_{\mf{a}''}([s-2])=\pi_{\mf{a}'}([s-2])$ is anticomplete to $L_{\mf{a}''}\setminus L_{\mf{a}'}=u'\dd L_{\mf{a}}\dd w$. But then by \ref{CA}, $\mf{a}'$ is not an $(\mf{a},x,s-1)$-candidate in $G$, a contradiction. This proves \eqref{st:ends_work}.
\medskip

From \eqref{st:ends_work} and the fact that $\mf{a}$ is $2$-ample, it follows that $x$ is anticomplete to the ends of $L_{\mf{a}'}$, and so $\mf{a}'$ and $x$  satisfy \ref{CH1}. Also, we claim that:

\sta{\label{st:x_does_CH3}Let $P$ be an open $\mf{a}'$-piece. Then $x$ has a neighbor in $P$.}

First, assume that $P$ is an internal open $\mf{a}'$-piece. Then $P$ is an open $\mf{a}|S_{\mf{a}'}$-piece. Since $x$ is a cherry on top of $\mf{a}|S_{\mf{a}'}$ in $G$, it follows from \ref{CH2} that $x$ has a neighbor in $P$, as desired. Next, assume that $P$ is an external $\mf{a}'$-piece. Then $P$ and $L_{\mf{a}'}$ share at least one end, say $u$. By \eqref{st:ends_work}, either $u$ is an end of $L_{\mf{a}}$, or the unique neighbor $u'$ of $u$ in $L_{\mf{a}}\setminus L_{\mf{a}'}$ is adjacent to $\pi_{\mf{a}'}(i)$ for some $i\in [s-2]$. In the former case, $P$ is an external $\mf{a}|S_{\mf{a}'}$-piece, and so $P$ is an open $\mf{a}|S_{\mf{a}'}$-piece. Again, since $x$ is a cherry on top of $\mf{a}|S_{\mf{a}'}$ in $G$, it follows from \ref{CH2} that $x$ has a neighbor in $P$. In the latter case, traversing $L_{\mf{a}'}$ starting at $u$, let $u''$ be the first vertex with a neighbor in $S_{\mf{a}'}$. Since $\mf{a}'$ is interrupted, it follows that $u''$ is a neighbor of $\pi_{\mf{a}'}(s-1)$, and so there exists an $\mf{a}|S_{\mf{a}'}$-route $R$ from $\pi_{\mf{a}'}(s-1)$ to $\pi_{\mf{a}'}(i)$ such that $P=R^*\setminus \{u'\}$. Note that since $\mf{a}$ is $2$-ample, the ends of $R^*$ have no common neighbor in $S_{\mf{a}'}$, and so $R^*$ is an open $\mf{a}|S_{\mf{a}'}$-piece. Therefore, since $x$ is a cherry on top of $\mf{a}|S_{\mf{a}'}$ in $G$, it follows from \ref{CH2} that $x$ has a neighbor in $R^*$. On the other hand, since $x''\in S_{\mf{a}}\setminus \{x\}$ is adjacent to $u'$ and $\mf{a}$ is $2$-ample, it follows that $x$ is not adjacent to $u'$. But now $x$ has a neighbor in $P$. This proves \eqref{st:x_does_CH3}.
\medskip

By \eqref{st:x_does_CH3}, $\mf{a}'$ and $x$  satisfy \ref{CH2}. This completes the proof of Lemma~\ref{lem:interrupted_top}.
\end{proof}

We also need the following, which is an application of Lemma~\ref{lem:interrupted_top}:

\begin{lemma}\label{lem:occultation_top}
     Let $o\in \poi$ and let $r,r',s$ be integers such that $r>r'\geq s-1\geq 0$. Let $G$ be a graph and let $\mf{a}$ be a $2$-ample, interrupted ordered $r$-asterism in $G$. Assume that $\pi_{\mf{a}}(r)$ has a neighbor in every closed $\mf{a}^{r'}$-piece of length at least $o$. Assume also that there exists a full $(s-1,o)$-occultation $\mf{o}'$ in $G$ with $S_{\mf{o}'}\subseteq S_{\mf{a}^{r'}}$ and $L_{\mf{o}'}\subseteq L_{\mf{a}}$. Then there exists a full $(s,o)$-occultation $\mf{o}$ in $G$ with $S_{\mf{o}}\subseteq S_{\mf{a}}$ and $L_{\mf{o}}\subseteq L_{\mf{a}}$.
 \end{lemma}
\begin{proof}
We write $x=\pi_{\mf{a}}(r)$. Choose a full $(s-1,o)$-occultation $\mf{o}'$ in $G$ with $S_{\mf{o}'}\subseteq S_{\mf{a}^{r'}}$ and  $L_{\mf{o}'}\subseteq L_{\mf{a}}$, such that $L_{\mf{o}'}$ is maximal with respect to inclusion. In particular,  $\mf{o}'$ is ample, interrupted and $o$-invaded with $S_{\mf{o}'}\subseteq S_{\mf{a}^{r'}}\subseteq S_{\mf{a}}\setminus \{x\}$. We further deduce that: 

\sta{\label{st:x_is_cherry} $\mf{o}'$ is an $(\mf{a},x,s-1)$-candidate in $G$ and $x$ is a cherry on top of $\mf{a}|S_{\mf{o}'}$ in $G$.}

Note that from the maximality of $L_{\mf{o}'}$, it follows immediately that $\mf{a}$, $\mf{o}'$ and $x$ satisfy \ref{CA}. Therefore, $\mf{o}'$ is an $(\mf{a},x,s-1)$-candidate in $G$. It remains to show that $x$ is a cherry on top of $\mf{a}|S_{\mf{o}'}$. To that end, we need to argue that $\mf{a}|S_{\mf{o}'}$ and $x$ satisfy \ref{CH1} and \ref{CH2}. Observe that \ref{CH1} follows immediately from the fact that $L_{\mf{a}|S_{\mf{o}'}}=L_{\mf{a}}$. For \ref{CH2}, let $P$ be an open $\mf{a}|S_{\mf{o}'}$-piece. Since $S_{\mf{o}'}\subseteq S_{\mf{a}^{r'}}\subseteq S_{\mf{a}^{r-1}}$, it follows that $P$ contains an open $\mf{a}^{r-1}$-piece $P'$. But now since $\mf{a}$ is interrupted, it satisfies \ref{INT} for $i=r$, that is, $x=\pi_{\mf{a}}(r)$ has a neighbor in the open $\mf{a}^{r-1}$-piece $P'$, and so $x$ has a neighbor in $P$. This proves \eqref{st:x_is_cherry}.
\medskip

In view of \eqref{st:x_is_cherry}, we can apply Lemma~\ref{lem:interrupted_top} to $\mf{a}, \mf{o}'$ and $x$, and deduce that $\mf{o}=\cher(\mf{o'},x)$ is a $2$-ample, interrupted ordered $s$-asterism in $G$ with $S_{\mf{o}}=S_{\mf{o}'}\cup \{x\}\subseteq S_{\mf{a}}$ and $L_{\mf{o}}\subseteq L_{\mf{a}}$. Also, we have:

\sta{\label{f_is_o_invaded}$\mf{o}$ is $o$-invaded.}

We need to prove that $\mf{o}$ satisfies \ref{OI} for every $i\in [s]$. This is immediate for $i\in [s-1]$ as $\mf{o}'$ is $o$-invaded. For $i=s$, let $P$ be a closed $\mf{o}^{s-1}$-piece of length at least $o$. Our goal is to show that $\pi_{\mf{o}}(s)=x$ has a neighbor in $P$. Since $S_{\mf{o}^{s-1}}\subseteq S_{\mf{a}^{r'}}\subseteq S_{\mf{a}^{r-1}}$, it follows that either $P$ is a closed $\mf{a}^{r'}$-piece, or $P$ contains an open $\mf{a}^{r'}$-piece, which in turn implies that $P$ contains an open $\mf{a}^{r-1}$-piece $P'$. In the former case, $\pi_{\mf{o}}(s)=x$ has a neighbor in $P$ due to the assumption of Lemma~\ref{lem:occultation_top} that $\pi_{\mf{a}}(r)=x$ has a neighbor in every closed $\mf{a}^{r'}$-piece of length at least $o$. In the latter case, since $\mf{a}$ is interrupted, $\mf{a}$ satisfies \ref{INT} for $i=r$. In particular, $x=\pi_{\mf{a}}(r)$ has a neighbor in the open $\mf{a}^{r-1}$-piece $P'$, and so $x$ has a neighbor in $P$. This proves \eqref{f_is_o_invaded}.
\medskip

In conclusion, we have shown that $\mf{o}$ is a $2$-ample, interrupted and $o$-invaded ordered $s$-asterism in $G$ with $S_{\mf{o}}\subseteq S_{\mf{a}}$ and $L_{\mf{o}}\subseteq L_{\mf{a}}$. Hence, $\mf{o}$ is a full $(s,o)$-occultation in $G$  with $S_{\mf{o}}\subseteq S_{\mf{a}}$ and $L_{\mf{o}}\subseteq L_{\mf{a}}$. This completes the proof of Lemma~\ref{lem:occultation_top}.
\end{proof}

We can now prove the main result of this section:

\begin{proof}[Proof of Theorem~\ref{thm:interrupted_to_occultation}]
We proceed by induction on $c+s\geq 1$. The result is immediate for $s\in \{0,1\}$ as in this case $\mf{a}$ is a full $(s,o)$-occultation in $G$. So we may assume that $s\geq 2$, which in turn implies that $c+s\geq 3$.

 Let $G$ be a $(c,o)$-perforated graph, let $\mf{a}$ be a $2$-ample, interrupted ordered $s^c$-asterism in $G$, and suppose for a contradiction that there is no full $(s,o)$-occultation $\mf{o}$ in $G$ with $S_{\mf{o}}\subseteq S_{\mf{a}}$ and $L_{\mf{o}}\subseteq L_{\mf{a}}$. Write $r=s^c$, $r'=(s-1)^c$ and $x=\pi_{\mf{a}}(r)$. It follows that $r>r'\geq 1$. We deduce that:

\sta{\label{st:missed_piece} There exists a closed $\mf{a}^{r'}$-piece $P$ of length at least $o$ such that $x$ is anticomplete to $P$.}

Suppose not. Then $x=\pi_{\mf{a}}(r)$ has a neighbor in every closed $\mf{a}^{r'}$-piece $P$ of length at least $o$. Note that $\mf{a}^{r'}$ is an ordered $r'$-astersim in $G$. Also, $\mf{a}^{r'}$ is both $2$-ample and interrupted, because $\mf{a}$ is. Consequently, by the induction hypothesis, there is a full $(s-1,o)$-occultation $\mf{o}'$ in $G$ with $S_{\mf{o}'}\subseteq S_{\mf{a}^{r'}}$ and $L_{\mf{o}'}\subseteq L_{\mf{a}^{r'}}=L_{\mf{a}}$. But now by Lemma~\ref{lem:occultation_top}, there exists a full $(s,o)$-occultation $\mf{o}$ in $G$ with $S_{\mf{o}}\subseteq S_{\mf{a}}$ and $L_{\mf{o}}\subseteq L_{\mf{a}}$, a contradiction. This proves \eqref{st:missed_piece}.
\medskip

Henceforth, let $P$ be as in \eqref{st:missed_piece}. Then there exists a vertex $z\in S_{\mf{a}^{r'}}=\pi_{\mf{a}}([r'])$ which is adjacent to both ends of $P$. It follows $H=P\cup \{z\}$ is a cycle in $G$ of length at least $o+2$. As a result, we have $c\geq 2$. Moreover, we claim that:

\sta{\label{st:top_anti_to_C} $S_{\mf{a}}\setminus \{z\}$ is anticomplete to $P$, and so to $H$.}

By \eqref{st:missed_piece}, $x$ is anticomplete to $P$. Suppose for a contradiction that $x'\in S_{\mf{a}}\setminus \{x,z\}$ has a neighbor in $P$. Then $P$ contains the interior of an $\mf{a}$-route $R$ from $x'$ to $z$. Since $x',z\in S_{\mf{a}}\setminus \{x\}=S_{\mf{a}^{r-1}}$, it follows that $R$ is an $\mf{a}^{r-1}$-route, and so $R^*$ is an open $\mf{a}^{r-1}$-piece. Therefore, since $\mf{a}$ is interrupted,  by \ref{INT} for $i=r$, $x=\pi_{\mf{a}}(r)$ has a neighbor in the interior of the open $\mf{a}^{r-1}$-piece $R^*$. But then $x$ has a neighbor in $P$, a contradiction. This proves \eqref{st:top_anti_to_C}.
\medskip

Now, since $c,s\geq 2$, we have:
$$r-r'=s^c-(s-1)^c=\sum_{i=1}^{c}s^{c-i}(s-1)^{i-1}\geq s^{c-1}+s^{c-2}(s-1);$$
which in turn implies that:
$$r-(r'+1
)\geq s^{c-1}+s^{c-2}(s-1)-1\geq s^{c-1}>1.$$
In particular, for $S=\pi_{\mf{a}}([r]\setminus [r - s^{c-1}])$, we have $S\subseteq \pi_{\mf{a}}([r]\setminus [r'+1])$ with $|S|=s^{c-1}>1$ and  $x=\pi_{\mf{a}}(r)\in S$. Let $y=\pi_{\mf{a}}(r'+1)$. It follows that $y,z\notin S$. By \eqref{st:top_anti_to_C}, $S$ is anticomplete to $H$.

Next, note that since both $y,z$ have neighbors in $L_{\mf{a}}$, there exists an $\mf{a}$-route $R$ from $y$ to $z$. Let $Q=R^*$. Since $\mf{a}$ is $2$-ample, it follows that $Q^*\neq \emptyset$. Also, by \eqref{st:top_anti_to_C}, $P$ and $Q$ share at most one vertex (which would be a common end of $P$ and $Q$). Let $\partial Q=\{u,v\}$ such that $y$ and $z$ are adjacent to $u
$ and $v$, respectively. Let $L=Q\setminus (N_{Q}[u]\cup N_{Q}[v])$. Then we have:

\sta{\label{st:new_ast_anti} $L$ is anticomplete to $H$. Also, every vertex in $S$ has a neighbor in $L^*$, and $S$ is anticomplete to $\partial L$. In particular, we have $L\neq \emptyset$.}

The first assertion is immediate from the definition of $L$. For the second assertion,
let $x'\in S$. Then we have $x'=\pi_{\mf{a}}(i)$ for some $i\in [r]\setminus [r'+1]$. Since $R$ is an $\mf{a}$-route from $y$ to $z$, it follows that $R$ is an $\mf{a}^{i-1}$-route, and so $Q$ contains an open $\mf{a}^{i-1}$-piece $Q'$. Therefore, by the assumption that $\mf{a}$ is interrupted, $x'$ has a neighbor in $Q'\subseteq Q$. On the other hand, note that $y,z\in S_{\mf{a}}\setminus \{x'\}$, $y$ is adjacent to $u$ and $z$ is adjacent to $v$. Therefore, since $\mf{a}$ is $2$-ample, it follows that $x'$ has a neighbor in $L^*$ and $S$ is anticomplete to the ends of $L$. This proves \eqref{st:new_ast_anti}.
\medskip

We are almost done. Let $G'=G[S\cup L]$. From \eqref{st:new_ast_anti}, we conclude that $\mf{a}'=(S,L)$ is an ordered $s^{c-1}$-asterism in $G'$ with $\pi_{\mf{a}'}(i)=\pi_{\mf{a}}(r-s^{c-1}+i)$ for every $i\in [s^{c-1}]$. Also, $\mf{a}'$ is both interrupted and $2$-ample, as so is $\mf{a}$. Recall the assumption that there is no full $(s,o)$-occultation $\mf{o}$ in $G$ with $S_{\mf{o}}\subseteq S_{\mf{a}}$ and $L_{\mf{o}}\subseteq L_{\mf{a}}$. It follows that there is no full $(s,o)$-occultation $\mf{o}$ in $G'$ with $S_{\mf{o}}\subseteq S_{\mf{a}'}$ and $L_{\mf{o}}\subseteq L_{\mf{a}'}$. Hence, by the induction hypothesis applied to $G'$ and $\mf{a}'$, $G'$ contains $c-1$ pairwise disjoint and anticomplete cycles $H_1,\ldots, H_{c-1}$, each of length at least $o+2$. On the other hand, by \eqref{st:top_anti_to_C} and \eqref{st:new_ast_anti}, $V(G')$ is anticomplete to $H$ in $G$. But now $H,H_1,\ldots, H_{c-1}$ is a collection of $c$ pairwise disjoint and anticomplete cycles in $G$, each of length at least $o+2$, a contradiction with the assumption that $G$ is $(c,o)$-perforated. This completes the proof of Theorem~\ref{thm:interrupted_to_occultation}.
 \end{proof}

\section{Bundles and Constellations}\label{sec:bundle}
Broadly, our proofs in this paper deal with subgraphs that consist of several pairwise vertex-disjoint induced paths, together with a few other vertices marked as ``special'' (which may or may not belong to the paths). We call these subgraphs ``bundles'' and the goal of this section is to prove certain results to control the extra edges within a bundle preventing it from being induced, and also the edges between bundles preventing them from being anticomplete. The main result is Theorem~\ref{thm:bundlethm} below, of which the proof is in multiple steps and relies on various Ramsey-type arguments. Intuitively, when the extra edges are prevalent, we find an induced subgraph (called a ``plain constellation'') which is the right place to look for the full occultation, and when the extra edges are rare, we find several pairwise disjoint and anticomplete bundles, which we then use to find several pairwise disjoint and anticomplete (long) cycles.

Let us now give the exact definitions. Let $G$ be a graph and let $l\in \poi$. By an \textit{$l$-polypath in $G$} we mean a set $\mca{L}$ of $l$ pairwise disjoint paths in $G$. Given an $l$-polypath $\mca{L}$ in $G$, we say $\mca{L}$ is \textit{plain} if every two distinct paths $L,L'\in \mca{L}$ are anticomplete in $G$. Also, two polypaths $\mca{L}$ and $\mca{L}'$ in $G$ are said to be \textit{disentangled} if $V(\mca{L})\cap V(\mca{L}')=\emptyset$. For $s \in \poi \cup \{0\}$, an \textit{$(s,l)$-bundle in $G$} is a pair $\mf{b}=(S_{\mf{b}},\mca{L}_{\mf{b}})$ where $S_{\mf{b}}\subseteq V(G)$ with $|S_{\mf{b}}|=s$ and $\mca{L}_{\mf{b}}$ is an $l$-polypath in $G$ (note that $S_{\mf{b}}$ and $V(\mca{L}_{\mf{b}})$ are not necessarily disjoint.) By an \textit{$(\leq s, l)$-bundle in $G$}, we mean an $(s',l)$-bundle in $G$ for some integer $0 \leq s'\leq s$. Given an $(s,l)$-bundle $\mf{b}$ in $G$, we write $V(\mf{b})=S_{\mf{b}}\cup V(\mca{L}_{\mf{b}})$, and we say that $\mf{b}$ is \textit{plain} if the $l$-polypath $\mca{L}_{\mf{b}}$ is plain. Also, two bundles $\mf{b}$ and $\mf{b}'$ in a graph $G$ are said to be \textit{disentangled} if $V(\mf{b})\cap V(\mf{b}')=\emptyset$. 

As mentioned above, the main result of this section, Theorem~\ref{thm:bundlethm} below, is a Ramsey-type result concerning pairwise disentangled plain bundles in perforated graphs. This involves a specific type of bundles which we need to define separately. Given a graph $G$ and   $s \in \poi \cup \{0\}$ and $l\in \poi$,  an \textit{$(s,l)$-constellation in $G$} is an $(s,l)$-bundle $\mf{c}=(S_{\mf{c}},\mca{L}_{\mf{c}})$ in $G$ such that $S_{\mf{c}}$ is a stable set (of cardinality $s$) in $G\setminus V(\mca{L}_{\mf{c}})$, and every $s\in S_{\mf{c}}$ has a neighbor in every path $L\in \mca{L}_{\mf{c}}$ (in particular, the vertices in $S_{\mf{c}}$ may be adjacent to the ends of a path $L\in \mca{L}_{\mf{c}}$, and so $(S_\mf{c},L)$ is not necessarily an $s$-asterism in $G$). See Figure~\ref{fig:constellation}. See also \cite{cocks} where the somewhat similar notion of a ``$t$-sail'' is studied.
\begin{figure}[t!]
    \centering
    \includegraphics[scale=0.7]{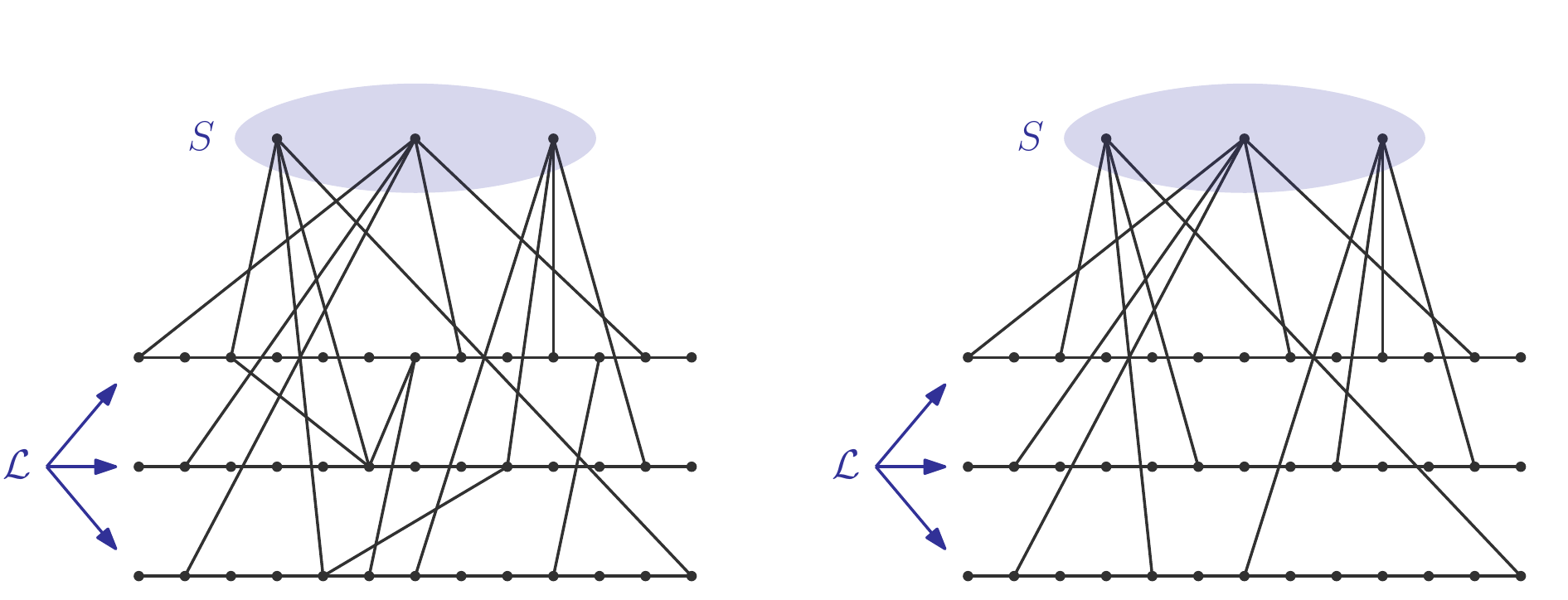}
    \caption{A $(3,3)$-constellation which is not plain (left) and one which is plain (right).}
    \label{fig:constellation}
\end{figure}

\begin{theorem}\label{thm:bundlethm}
    \sloppy For all   $a,b,c,h,l,o,s,t\in \poi$, there are constants $\Theta=\Theta(a,b,c,h,l,o,s,t)\in \poi$ and $\theta=\theta(a,b,c,h,l,o,s,t)\in \poi$ with the following property. Let $G$ be a $(c,o)$-perforated graph and let $\mf{T}$ be a collection of $\Theta$ pairwise disentangled plain $(\leq b,\theta)$-bundles in $G$. Then one of the following holds.
    \begin{enumerate}[\rm (a)]
        \item\label{thm:bundlethm_a} $G$ contains either $K_t$ or $K_{t,t}$.
        \item\label{thm:bundlethm_b} There exists a plain $(s,l)$-constellation in $G$.
        \item\label{thm:bundlethm_c} There exists $\mf{A}\subseteq \mf{T}$ with $|\mf{A}|=a$ and $\mca{H}_{\mf{b}}\subseteq \mca{L}_{\mf{b}}$ with $|\mca{H}_{\mf{b}}|=h$ for each $\mf{b}\in \mf{A}$, such that for all distinct $\mf{b},\mf{b}'\in \mf{A}$, $S_{\mf{b}}\cup V(\mca{H}_{\mf{b}})$ is anticomplete to $S_{\mf{b}'}\cup V(\mca{H}_{\mf{b}'})$ in $G$.
    \end{enumerate}
\end{theorem}
Our way towards the proof of Theorem~\ref{thm:bundlethm} passes through an assortment of definitions and lemmas, often of independent interest and even with applications in future sections (and papers). We begin with two lemmas which capture our main use of Theorems~\ref{multiramsey} and \ref{productramsey}. 

The first lemma is quite similar to Theorem~\ref{thm:bundlethm}, but still not as strong: the constellation in the second outcome is not plain, and the third outcome is not able to rule out the edges between the ``$\mca{L}$-sides'' of the bundles:

\begin{lemma}\label{lem:bundle}
   \sloppy For all   $b,f,g,m,n,t\in \poi$ and $s \in \poi \cup \{0\}$, there is a constant $\beta=\beta(b,f,g,m,n,s,t)\in \poi$ with the following property. Let $G$ be a graph and let $\mf{B}$ be a collection of $\beta$ pairwise disentangled $(\leq b,2b(g-1)+f)$-bundles in $G$. Then one of the following holds.
    \begin{enumerate}[\rm (a)]
        \item\label{lem:bundle_a} $G$ contains either $K_t$ or $K_{t,t}$.
        \item\label{lem:bundle_b} There exist  $\mf{N}\subseteq \mf{B}$ with $|\mf{N}|=n$ as well as $S\subseteq \bigcup_{\mf{b}\in \mf{B}\setminus \mf{N}} S_{\mf{b}}$ with $|S|=s$, such that for every $\mf{b}\in \mf{N}$, there exists $\mca{G}_{\mf{b}}\subseteq \mca{L}_{\mf{b}}$ with $|\mca{G}_{\mf{b}}|=g$ for which $(S, \mca{G}_{\mf{b}})$ is an $(s,g)$-constellation in $G$. As a result, $(S,\bigcup_{\mf{b}\in \mf{N}}\mca{G}_{\mf{b}})$ is an $(s,gn)$-constellation in $G$. 
        \item\label{lem:bundle_c} There exist  $\mf{M}\subseteq \mf{B}$ with $|\mf{M}|=m$ and $\mca{F}_{\mf{b}}\subseteq \mca{L}_{\mf{b}}$ with $|\mca{F}_{\mf{b}}|=f$ for each $\mf{b}\in \mf{M}$, such that for all distinct $\mf{b},\mf{b}'\in \mf{M}$, $S_{\mf{b}}$ is anticomplete to $S_{\mf{b}'}\cup V(\mca{F}_{\mf{b}'})$ in $G$.
    \end{enumerate}
\end{lemma}
\begin{proof}
Let $l=2b(g-1)+f$ and let $\rho=\rho(\max\{m,n+s,2t\},2,2^{b^2+2bl})$
 be as in Theorem~\ref{multiramsey}. Let
$$\beta=\beta(b,f,g,m,n,s,t)=(b+1)\rho.$$
Suppose that Lemma~\ref{lem:bundle}\ref{lem:bundle_a} and Lemma~\ref{lem:bundle}\ref{lem:bundle_b} do not hold. From the choice of $\beta$, it follows that there exists $c\in \{0,1,\ldots, b\}$ as well as $\mf{C}\subseteq \mf{B}$ with $|\mf{C}|=\rho$ such that for every $\mf{b}\in \mf{C}$, we have $|S_{\mf{b}}|=c$. Consider the following three pairwise disjoint sets: 
$$X=\{x_i:i\in [c]\};$$
$$X'=\{x'_i:i\in [c]\};$$
$$\mca{L}=\{L_1,\ldots, L_l\}.$$
Let $\mca{W}$ be the set of all vertex-labelled 3-partite graphs with vertex set $X\cup X'\cup \mca{L}$ and $3$-partition $(X, X', \mca{L})$. So we have $|\mca{W}|=2^{c^2+2cl}\leq 2^{b^2+2bl}$.

Next, let us write $\mf{C}=\{\mf{b}_1,\ldots, \mf{b}_{\rho}\}$, and for every $i\in [\rho]$, let $S_{\mf{b}_i}=\{x_1^i,\ldots, x_{c}^i\}$ and
$\mca{L}_{\mf{b}_i}=\{L_1^i,\ldots, L_l^i\}$. For every $2$-subset $
\{i,j\}$ of $[\rho]$ with $i<j$, let $W_{i,j}$ be the unique graph in $\mca{W}$ with the following specifications.
\begin{enumerate}[(W1), leftmargin=19mm, rightmargin=7mm]
    \item\label{W1} For all $p,q\in [c]$, we have $x_px'_q\in  E(W_{i,j})$ if and only if we have $ x_p^ix^j_q\in E(G)$.
    \item\label{W2} For all $p\in [c]$ and $q\in [l]$, we have $x_pL_q\in E(W_{i,j})$ if and only if $x_p^i$ has a neighbor in $L^{j}_q$ in $G$.
   \item\label{W3} For all $p\in [c]$ and $q\in [l]$, we have $x'_pL_q\in E(W_{i,j})$ if and only if $x_p^j$ has a neighbor in $L^i_q$ in $G$.
\end{enumerate}
Now, let $\Phi:\binom{[\rho]}{2}\rightarrow \mca{W}$ be the map with $\Phi(\{i,j\})=W_{i,j}$ for every $2$-subset $
\{i,j\}$ of $[\rho]$ with $i<j$. It follows that $\Phi$ is well-defined. The choice of $\rho$ then allows an application of Theorem~\ref{multiramsey}, which implies that there exists $I\subseteq [\rho]$ with $|I|=\max\{m,n+s,2t\}$ as well as $W\in \mca{W}$ such that for every $2$-subset $
\{i,j\}$ of $I$ with $i<j$, we have $\Phi(\{i,j\})=W_{i,j}=W$. In particular, one may pick $I_1,J_1,I_2,J_2, I_3,J_3,M\subseteq I$ such that 
\begin{itemize}
    \item $|I_1|=|J_1|=t$ and $\max I_1< \min J_1$;
    \item $|I_2|=s$, $|J_2|=n$ and $\max I_2< \min J_2$;
    \item $|I_3|=s$, $|J_3|=n$ and $\max J_3< \min I_3$; and
    \item $|M|=m$.
\end{itemize}

We claim that:

\sta{\label{st:Xanticomplete} The sets $\{S_{\mf{b}_i}:i\in I\}$ are pairwise anticomplete in $G$.}

In view of \ref{W1}, it suffices to show that $X$ and $X'$ are anticomplete in $W$. Note that if $x_px'_p\in E(W)$ for some $p\in [c]$, then by \ref{W1}, $G[\{x_p^i:i\in I_1\}]$ is isomorphic to $K_t$ and so Lemma~\ref{lem:bundle}\ref{lem:bundle_a} holds, a contradiction. It follows that for every $p\in [c]$, the vertices $x_p$ and $x'_p$ are non-adjacent in $W$. This, along with \ref{W1} and the fact that $ W_{i,j}=W$ for all $i,j\in I$ with $i<j$, implies that $\{x_p^i:i\in I\}$ is a stable set in $G$. Now, assume that $x_px'_q\in E(W)$ for some $p,q\in [c]$. It follows that $p$ and $q$ are distinct and $\{x_p^i:i\in I_1\}$ and $\{x_q^j:j\in J_1\}$ are both stable sets in $G$. But then by \ref{W1},  $G[\{x_p^i:i\in I_1\}\cup \{x_q^j:i\in J_1\}]$
is isomorphic to $K_{t,t}$, and so Lemma~\ref{lem:bundle}\ref{lem:bundle_a} holds, again a contradiction. This proves \eqref{st:Xanticomplete}.

\sta{\label{st:findM} There is an $f$-subset $F$ of $[l]$ such that $X\cup X'$ is anticomplete to $\{L_{q}:q\in F
\}$ in $W$.}

Suppose not. Then, in the graph $W$, there are fewer than $f$ vertices in $\mca{L}$ with no neighbor in $X\cup X'$, and so the number of edges in the graph $W$ with an end in $X\cup X'$ and an end in $\mca{L}$ is greater than $|\mca{L}|-f$. On the other hand, recall that $|X|=|X'|=c$ and
$$|\mca{L}|=l=2b(g-1)+f\geq 2c(g-1)+f=(g-1)|X\cup X'|+f.$$
It follows that the number of edges in the graph $W$ with an end in $X\cup X'$ and an end in $\mca{L}$ is greater than $(g-1)|X\cup X'|$. Therefore, there exists a vertex in $X\cup X'$ which has at least $g$ neighbors in $\mca{L}$ (in $W$). In other words, there exists $p\in [c]$ and $F'\subseteq [l]$ with $|F'|=g$ such that one of $x_p$ and $x'_p$ is complete to $\{L_{q}:q\in F'\}$ in $W$. This, along with \ref{W2} and \ref{W3}, implies that there exists $k\in \{2,3\}$ such that for every $i\in I_k$, every $j\in J_k$ and every $q\in F'$, $x_p^i\in S_{\mf{b}_i}$ has a neighbor in $L_q^j\in \mca{L}_{\mf{b}_j}$. Let $S=\{x^i_p:i\in I_k\}$. Let $\mf{N}=\{\mf{b}_j:j\in J_k\}$; so we have $|\mf{N}|=n$.  For every $j\in J_k$, let $\mca{G}_{\mf{b}_j}=\{L_q^j:q\in F'\}\subseteq \mca{L}_{\mf{b}_j}$; so $\mca{G}_{\mf{b}_j}$ is a $g$-polypath in $G$. It follows that $S\subseteq \bigcup_{\mf{b}\in \mf{B}\setminus \mf{N}} S_{\mf{b}}$, and by \eqref{st:Xanticomplete},  $S$ is a stable set in $G\setminus (\bigcup_{\mf{b}\in \mf{N}}V(\mca{G}_{\mf{b}}))$. Also, by construction (and by the choice of $k$, in particular), for every $\mf{b}\in \mf{N}$ and every path $L\in \mca{G}_{\mf{b}}$, every vertex in $S$ has a neighbor in $L$ in $G$. Thus, $(S, \mca{G}_{\mf{b}})$ is an $(s,g)$-constellation in $G$. But now $\mf{N}, S$ and $\{\mca{G}_{\mf{b}}:\mf{b}\in \mf{N}\}$ satisfy \ref{lem:bundle_b}, a contradiction. This proves \eqref{st:findM}.
\medskip

We now finish the proof. Let $\mf{M}=\{\mf{b}_j:j\in M\}$; then $|\mf{M}|=m$. Let $F$ be as in \eqref{st:findM}, and for every $j\in M$, define the $f$-polypath $\mca{F}_{\mf{b}_j}=\{L_q^j:q\in F\}\subseteq \mca{L}_{\mf{b}_j}$. For all distinct $i,j\in M$, it follows from \eqref{st:Xanticomplete} that $S_{\mf{b}_i}$ is anticomplete to $S_{\mf{b}_j}$ in $G$. It also follows from \eqref{st:findM} and the definition of the graph $W=W_{i,j}$ (specifically from \ref{W2} and \ref{W3}) that $S_{\mf{b}_i}$ is anticomplete to $V(\mca{F}_{\mf{b}_j})$ in $G$. In conclusion, we have shown that for all distinct $i,j\in M$, $S_{\mf{b}_i}$ is anticomplete to $S_{\mf{b}_j}\cup V(\mca{F}_{\mf{b}_j})$. Hence, $\mf{M}$ and $\{\mca{F}_{\mf{b}}:\mf{b}\in \mf{M}\}$ satisfy \ref{lem:bundle_c}.  This completes the proof of Lemma~\ref{lem:bundle}.
\end{proof}

The second lemma will help eradicate the edges between paths, further pushing \ref{lem:bundle}\ref{lem:bundle_b} and \ref{lem:bundle}\ref{lem:bundle_c} towards \ref{thm:bundlethm}\ref{thm:bundlethm_b}  and \ref{thm:bundlethm}\ref{thm:bundlethm_c}, respectively.

\begin{lemma}\label{lem:polypathvspolypath}
    For all   $a,g\in \poi$, there is a constant $\varphi=\varphi(a,g)\in \poi$ with the following property. Let $G$ be a graph and let $\mca{F}_1,\ldots, \mca{F}_a$ be a collection of $a$ pairwise disentangled $\varphi$-polypaths in $G$. Then for every $i\in [a]$,  there exists a $g$-polypath $\mca{G}_{i}\subseteq \mca{F}_i$, such that for all distinct $i,i'\in [a]$, either $V(\mca{G}_i)$ is anticomplete to $V(\mca{G}_{i'})$ in $G$, or for every $L\in \mca{G}_i$ and every $L'\in \mca{G}_{i'}$, there is an edge in $G$ with an end in $L$ and an end in $L'$.
\end{lemma}
\begin{proof}
    Define $\varphi=\varphi(a,g)=\nu(a,g,2^{a^2})$,
    where $\nu(\cdot,\cdot,\cdot)$ is as in Theorem~\ref{productramsey}. Consider a set $X=\{x_1,\ldots, x_a\}$, and let $\mca{W}$ be the set of all vertex-labeled graphs with vertex set $X$; so we have $|\mca{W}|\leq 2^{a^2}$. For every $\mathsf{F}=(L_1,\ldots, L_a)\in \mca{F}_1\times \cdots\times \mca{F}_a$, define $\Phi(\mathsf{F})$ to be the unique graph in $\mca{W}$ in which for all distinct $i,i'\in [a]$, we have $x_ix_{i'}\in E(\Phi(\mathsf{F}))$ if and only if there is an edge in $G$ with an end in $L_i$ and an end in $L_{i'}$. Then $\Phi:\mca{F}_1\times \cdots\times \mca{F}_a\rightarrow \mca{W}$ is a well-defined map. By the choice of $\varphi$, we can apply Theorem~\ref{productramsey}, and deduced that there exists $W\in \mca{W}$ as well as  $\mca{G}_i\subseteq \mca{F}_i$ with $|\mca{G}_i|=g\in \poi$ for each $i\in [a]$, such that for every $\mathsf{F}\in \mca{G}_1\times \cdots\times \mca{G}_a$, we have $\Phi(\mathsf{F})=W$.
    
    Now, let $i,i'\in [a]$ be distinct. It follows immediately that, if $x_ix_{i'}\notin E(W)$, then $V(\mca{G}_i)$ is anticomplete to $V(\mca{G}_{i'})$ in $G$, and if $x_ix_{i'}\in E(W)$, then for every $L\in \mca{G}_i$ and every $L'\in \mca{G}_{i'}$, there is an edge in $G$ with an end in $L$ and an end in $L'$. This completes the proof of Lemma~\ref{lem:polypathvspolypath}.
\end{proof}

The next lemma will be of extensive use in the remainder of this paper (and several forthcoming ones). Nevertheless, it is indeed based on a rather intuitive idea, which we briefly discuss here. Suppose $\mf{a}$ is an asterism with $S_{\mf{a}}$ sufficiently large, while every vertex in $L_{\mf{a}}$ is adjacent to only a few vertices in $S_{\mf{a}}$ (the latter is what we will refer to as ``meagerness''). Then, for each vertex in $x\in S_{\mf{a}}$, we consider the shortest path in  $L_{\mf{a}}$ which contains all neighbors of $x$ in $L_{\mf{a}}$. The key observation is that the intersection graph of these paths is an interval graph. If this interval graph contains a large stable set, then we obtain a large ``sub-asterism'' of $\mf{a}$ in which the neighbors of the vertices in the stable set appear on the path in order. We call the latter a ``syzygy,'' and syzygies (at least when ``glued together'') are useful for finding several pairwise disjoint and anticomplete long induced cycles (this will become clearer once we get to the fourth lemma in this section). Back to the interval graph, we may assume that it contains a large clique. Consequently, there is a vertex $u\in L_{\mf{a}}$ such that several vertices in $S_{\mf{a}}$ have neighbors in $L_{\mf{a}}$ on either side of $u$. But recall that $u$ itself has only a few neighbors in $S_{\mf{a}}$. So we obtain two disjoint and anticomplete paths in $L_{\mf{a}}$ such that several vertices in $S_{\mf{a}}$ have neighbors in both of them. This is the seed out of which we will grow a plain constellation, and as mentioned before, a plain constellation is exactly where we should be looking for a full occultation.

The exact definitions are as follows. Given a graph $G$ and $a\in \poi$, an \textit{$a$-syzygy} in $G$ is an ordered $a$-asterism $\mf{s}$ in $G$ such that for some end $u$ of $L_{\mf{s}}$, the following holds.
 \begin{enumerate}[(SY), leftmargin=19mm, rightmargin=7mm]
     \item\label{SY} For all $i,j\in [a]$ with $i<j$, every neighbor $v_i\in L_{\mf{s}}$ of $\pi_{\mf{s}}(i)$ and every neighbor $v_j\in L_{\mf{s}}$ of $\pi_{\mf{s}}(j)$, the path in $L_{\mf{s}}$ from $u$ to $v_j$ contains $v_i$ in its interior. In other words, traversing $L_{\mf{s}}$ starting at $u$, all neighbors of $\pi(i)$ appear ``before'' all neighbors of $\pi(j)$.
 \end{enumerate}
 \begin{figure}
\centering
\includegraphics[scale=0.7]{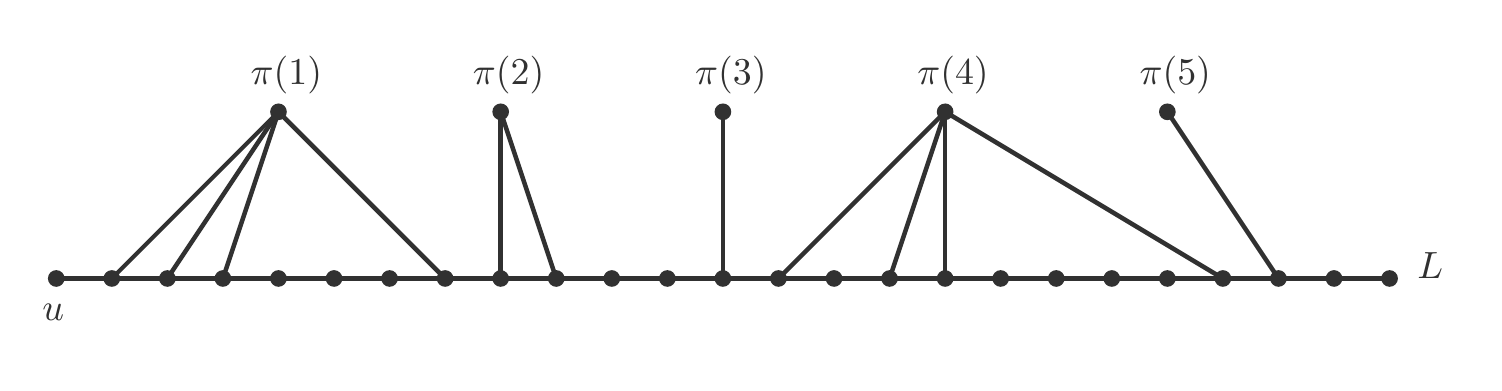}
\caption{A $5$-syzygy.}
\label{fig:syzygy}
\end{figure}
See Figure~\ref{fig:syzygy}. In particular, if $\mf{s}$ is an $a$-syzygy in $G$, then  $\mf{s}$ is ample, and for every non-empty subset $X$ of $S_\mf{s}$, $\mf{s}|X$ is a $|X|$-syzygy in $G$.

Recall that $d$-ample asterisms were introduced in Section~\ref{sec:asterism} as an extension of ample asterisms. Here is another notion extending ample asterisms, but the other way around. For $d \in \poi \cup \{0\}$ and an (ordered) asterism $\mf{a}$ in a graph $G$, we say $\mf{a}$ is \textit{$d$-meager} if every vertex in $L_{\mf{a}}$ has at most $d$ neighbors in $S_{\mf{a}}$. It follows that $\mf{a}$ is $0$-meager if and only if $S_{\mf{a}}=\emptyset$ and $1$-meager if and only if $\mf{a}$ is ample.

We also use the following folklore fact which is a direct consequence of interval graphs being perfect (we omit further details).
 
\begin{lemma}[Berge \cite{berge}]\label{lem:interval}
Let $a,b \in \poi \cup \{0\}$ and let $G$ be an interval graph on $ab$ vertices. Then $G$ contains either a stable set of cardinality $a$ or a clique of cardinality $b$. 
\end{lemma}

\begin{lemma}\label{lem:ast_to_syzygy}
     Let $a,l\in \poi$, $d,s \in \poi \cup \{0\}$ and let $G$ be a graph. Assume that there exists a $d$-meager (ordered) $(a^{l-1}(s+d(l-1)))$-asterism  $\mf{g}$ in $G$. Then one of the following holds.
      \begin{enumerate}[\rm (a)]
        \item\label{lem:ast_to_syzygy_a} There exists an $a$-syzygy $\mf{s}$ in $G$ with $S_{\mf{s}}\subseteq S_{\mf{g}}$ and $L_{\mf{s}}\subseteq L_{\mf{g}}$.
        \item\label{lem:ast_to_syzygy_b} There exists a plain $(s,l)$-constellation  $\mf{c}$ in $G$ such that $S_{\mf{c}}\subseteq S_{\mf{g}}$ and $L\subseteq L_{\mf{g}}^*$ for every $L\in \mca{L}_{\mf{c}}$.
      \end{enumerate}
\end{lemma}

  \begin{proof}For fixed $a,d$ and $s$, we proceed by induction on $l$. Note that if $l=1$,  then $(S_{\mf{g}},\{L_{\mf{g}^*}\})$ is a plain $(s,1)$-constellation in $G$ satisfying Lemma~\ref{lem:ast_to_syzygy}\ref{lem:ast_to_syzygy_b}. Thus, we may assume that $l\geq 2$.

 Let $\partial L_{\mf{g}}=\{u_1,u_2\}$. For every vertex $x\in S_{\mf{g}}$, traversing $L_{\mf{g}}$ from $u_1$ to $u_2$, let $v_x$ and $w_x$ be the first and the last neighbor of $x$ in $L_{\mf{g}}$, and let $P_x$ be the unique path in $L_{\mf{g}}$ with ends $v_x,w_x$.
  
  Define $\mathsf{G}$ to be the graph with vertex set $S_{\mf{g}}$ such that for distinct $x,y\in S_{\mf{g}}$, we have $xy\in E(\mathsf{G})$ if and only if $P_x\cap P_y\neq\emptyset$. It is readily seen that $\mathsf{G}$ is an interval graph on $a^{l-1}(s+d(l-1))$ vertices. Thus, by Lemma~\ref{lem:interval}, $\mathsf{G}$ contains either a stable set $A$ of cardinality $a$ or a clique $B$ of cardinality $$a^{l-2}(s+d(l-1))=a^{l-2}(s+d(l-2))+da^{l-2}\geq a^{l-2}(s+d(l-2))+d.$$
  In the former case, we may write $A=\{x_1,\ldots, x_{a}\}$ such that, for all distinct $i,j\in [a]$, the path in $L_{\mf{g}}$ from $u_1$ to $v_{x_j}$ contains $v_{x_i}$ if and only if $i<j$. But then $\mf{s}=(A,L_{\mf{g}})$ is an ordered $a$-asterism in $G$ with $\pi_{\mf{a}}(i)=x_i$ for every $i\in [a]$, and $\mf{s}$ together with the end $u_1$ of $L_{\mf{s}}=L_{\mf{g}}$ satisfy \ref{SY}. In other words, $\mf{s}$ is an $a$-syzygy in $G$ with $S_{\mf{s}}\subseteq S_{\mf{g}}$ and $L_{\mf{s}}\subseteq L_{\mf{g}}$, and so \ref{lem:ast_to_syzygy}\ref{lem:ast_to_syzygy_a} holds.

   Now, assume that $\mathsf{G}$ contains a clique $B$ of cardinality at least $a^{l-2}(s+d(l-2))+d$. It follows that there exists a vertex $u\in L^*_{\mf{g}}$ such that for every $x\in B$, we have $u\in P_x$. Let $L_i=u_i\dd L_{\mf{g}}\dd u$ for $i\in \{1,2\}$. Since $\mf{g}$ is $d$-meager, it follows that $|N_{B}(u)|\leq d$, and so there exists $B'\subseteq B$ with $|B'|=a^{l-2}(s+d(l-2))$ such that for every $i\in \{1,2\}$, every vertex in $B'$ has a neighbor in $L_i^*$, and $B'$ is anticomplete to $\partial L_i$. It follows that, for every $i=\{1,2\}$, $\mf{g}_i=(B',L_i)$ is an $(a^{l-2}(s+d(l-2)))$-asterism in $G$ which is $d$-meager, as so is $\mf{g}$. From the induction hypothesis applied to $\mf{g}_1$, we deduce that either here exists an $a$-syzygy $\mf{s}$ in $G$ with $S_{\mf{s}}\subseteq S_{\mf{g}_1}=B'\subseteq S_{\mf{g}}$ and $L_{\mf{s}}\subseteq L_{\mf{g}_1}=L_1\subseteq L_{\mf{g}}$, or there exists a plain $(s,l-1)$-constellation $\mf{c}_1$ in $G$ such that $S_{\mf{c}_1}\subseteq S_{\mf{g}_1}=B'\subseteq S_{\mf{g}}$ and $L\subseteq L_{\mf{g}_1}^*=L_1^*\subseteq L_{\mf{g}}^*$ for every $L\in \mca{L}_{\mf{c}_1}$. In the former case, Lemma~\ref{lem:ast_to_syzygy}\ref{lem:ast_to_syzygy_a} holds, as required. In the latter case, $\mf{c}=(S_{\mf{c}_1}, \mca{L}_{\mf{c}_1}\cup \{L_2^*\})$ is a plain $(s,l)$-constellation in $G$ such that $S_{\mf{c}}=S_{\mf{c}_1}\subseteq S_{\mf{g}}$ and $L\subseteq L_1^*\cup L_2^*=L_{\mf{g}}^*$ for every $L\in \mca{L}_{\mf{c}}$, and so Lemma~\ref{lem:ast_to_syzygy}\ref{lem:ast_to_syzygy_b} holds. This completes the proof of Lemma~\ref{lem:ast_to_syzygy}.
  \end{proof}
  
We need one more lemma, an application of Lemma~\ref{lem:ast_to_syzygy}, which in turn calls for one more definition.
Let $G$ be a graph and let $g\in \poi$. A \textit{$g$-gemini in $G$} is a pair $(\mf{g}_1,\mf{g}_2)$ of ordered $g$-asterisms in $G$ with the following specifications.
  \begin{enumerate}[(G1), leftmargin=19mm, rightmargin=7mm]
  \item\label{G1} We have $V(\mf{g}_1)\cap V(\mf{g}_2)= S_{\mf{g}_1}\cap S_{\mf{g}_2}$.
  \item\label{G2} $V(\mf{g}_1)\setminus V(\mf{g}_2)$ is anticomplete to $V(\mf{g}_2)\setminus V(\mf{g}_1)$.
      \item\label{G3} There exists a plain $g$-polypath $\mca{Q}=\{Q_1,\ldots, Q_{g}\}$ in $G\setminus (L_{\mf{g}_1}\cup L_{\mf{g}_2})$, such that for every $i\in [g]$:
      \begin{itemize}
          \item we have $\partial Q_i=\{\pi_{\mf{g}_1}(i), \pi_{\mf{g}_2}(i)\}$; and
          \item $Q_i\setminus \{\pi_{\mf{g}_1}(i)\}$ is anticomplete to $L_{\mf{g}_2}$ and $Q_i\setminus \{\pi_{\mf{g}_2}(i)\}$ is anticomplete to $L_{\mf{g}_1}$.
      \end{itemize}
In particular, for every $x\in S_{\mf{g}_1}\cap S_{\mf{g}_2}$, we have $\pi_{\mf{g}_1}^{-1}(x)=\pi_{\mf{g}_2}^{-1}(x)$.
  \end{enumerate}
  \begin{figure}
\centering
\includegraphics[scale=0.9]{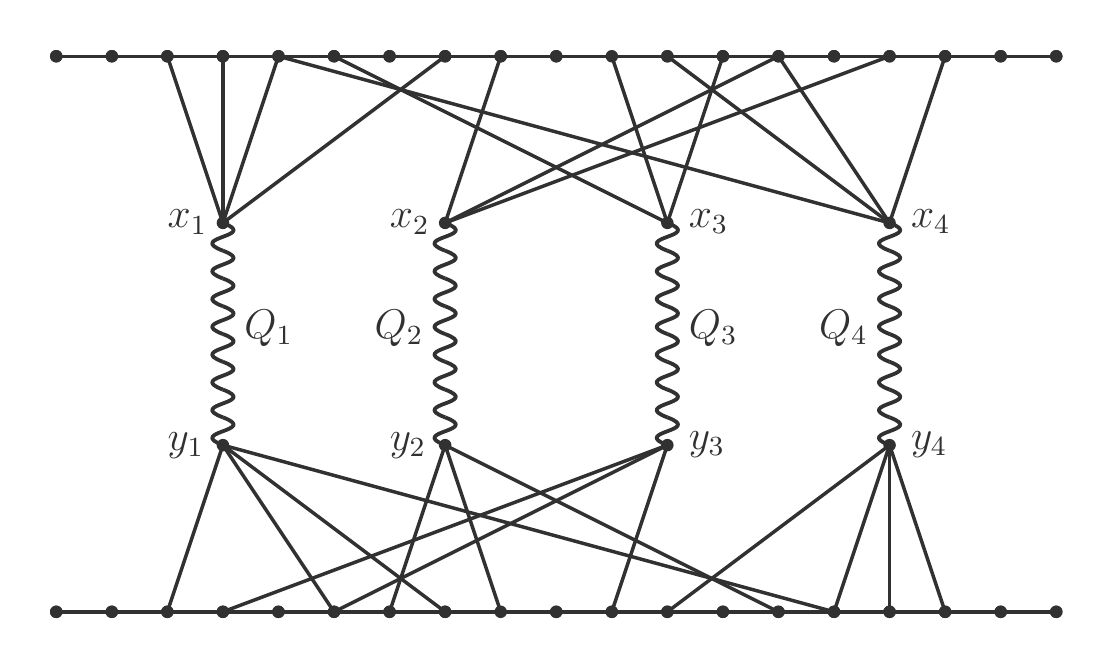}

\caption{A $4$-gemini $(\mf{g}_1,\mf{g}_2)$ with $\pi_{\mf{g}_1}(i)=x_i$ and $\pi_{\mf{g}_2}(i)=y_i$ for $i=1,2,3,4$. The paths $Q_1, Q_2, Q_3, Q_4$ are depicted in squiggly lines to mean they are of arbitrary (possibly zero) length.}
\label{fig:gemini}
\end{figure}

See Figure~\ref{fig:gemini}. It follows that if $(\mf{g}_1,\mf{g}_2)$ is a $g$-gemini in $G$, then for every non-empty subset $J$ of $[g]$, $(\mf{g}_1|\pi_{\mf{g}_1}(J),\mf{g}_2|\pi_{\mf{g}_2}(J))$ is a $|J|$-gemini in $G$.

The following shows that a perforated graph containing a large gemini must also contain a large plain constellation. 

\begin{lemma}\label{lem:gemini_to_constellation}
     For all   $c,l,o,s\in \poi$ and $d \in \poi \cup \{0\}$, there is a constant $\gamma=\gamma(c,d,l,o,s)\in \poi$ with the following property. Let $G$ be a $(c,o)$-perforated graph. Assume that there exists a $\gamma$-gemini $(\mf{g}_1,\mf{g}_2)$ in $G$ where both $\mf{g}_1$ and $\mf{g}_2$ are $d$-meager. Then there exists a plain $(s,l)$-constellation in $G$.
\end{lemma}  
  
  \begin{proof} Let $a=(4co)^{(l-1)}(s+d(l-1))\in \poi$ and let 
  $$\gamma=\gamma(c,d,l,o,s)=a^{(l-1)}(s+d(l-1))\geq 1.$$
  Let $(\mf{g}_1,\mf{g}_2)$ be a $\gamma$-gemini in $G$. Suppose for a contradiction that there is no plain $(s,l)$-constellation in $G$. Therefore, applying Lemma~\ref{lem:ast_to_syzygy} to the $d$-meager ordered $\gamma$-asterism $\mf{g}_1$, it follows that there exists an $a$-syzygy $\mf{s}$ in $G$ with $S_{\mf{s}}\subseteq S_{\mf{g}_1}$ and $L_{\mf{s}}\subseteq L_{\mf{g}_1}$. Let $S_2=\pi_{\mf{g}_2}(\pi^{-1}_{\mf{g}_1}(S_{\mf{s}}))$. Then we have $S_2\subseteq S_{\mf{g}_2}$ with $|S_2|=a$. Also, note that the ordered $a$-asterism $\mf{g}_2|S_2$ is $d$-meager, because $\mf{g}_2$ is. So another application of Lemma~\ref{lem:ast_to_syzygy}, this time to $\mf{g}_2|S_2$, implies that there exists a $4co$-syzygy $\mf{a}_2$ in $G$ with $S_{\mf{a}_2}\subseteq S_{\mf{g}_2|S_2}=S_2\subseteq S_{\mf{g}_2}$ and $L_{\mf{a}_2}\subseteq L_{\mf{g}_2|S_2}=L_{\mf{g}_2}$. Let $S_1=\pi_{\mf{g}_1}(\pi^{-1}_{\mf{g}_2}(S_{\mf{a}_2}))\subseteq S_{\mf{g}_1}$ and let $\mf{a}_1=\mf{s}|S_1$. Then $\mf{a}_1$ is a $4co$-syzygy in $G$ with $S_{\mf{a}_1}=S_1\subseteq S_{\mf{g}_1}$ and $S_{\mf{a}_1}=L_{\mf{s}}\subseteq L_{\mf{g}_1}$. 
  
  To sum up, we have shown that $(\mf{a}_1,\mf{a}_2)$ is a $4co$-gemini in $G$ such that for every $i\in \{1,2\}$, $\mf{a}_i$ is a $4co$-syzygy in $G$ with $S_{\mf{a}_i}\subseteq S_{\mf{g}_i}$ and $L_{\mf{a}_i}\subseteq L_{\mf{g}_i}$. Let $J=\{2jo:j\in [2c]\}$, and for every $i\in \{1,2\}$, let $\mf{s}_i=\mf{a}_i|\pi_{\mf{a}_i}(J)$. It follows that both $\mf{s}_1,\mf{s}_2$ are $2o$-ample ordered $2c$-asterisms, and that $(\mf{s}_1,\mf{s}_2)$ is a $2c$-gemini in $G$.
  
  Let $\mca{Q}=\{Q_1,\ldots, Q_{2c}\}$ be the $(2c)$-polypath in $G\setminus (L_{\mf{s}_1}\cup L_{\mf{s}_2})$ which along with $(\mf{s}_1,\mf{s}_2)$ satisfies \ref{G3}. For every $i\in \{1,2\}$ and every $j\in [c]$, let $P_j^i$ be the shortest path in $G[V(\mf{s}_i)]$ from $\pi_{\mf{s}_i}(2j-1)$ to $\pi_{\mf{s}_i}(2j)$. So the interior of $P_j^i$ is contained in $L_{\mf{s}_i}$, and since $\mf{s}_i$ is a $2o$-ample, it follows that $P_j^i$ has length at least $2o+2$. Consequently, for every $j\in [c]$,
$$H_j=\pi_{\mf{s}_1}(2j-1)\dd P_j^1\dd \pi_{\mf{s}_1}(2j)\dd Q_{2j}\dd \pi_{\mf{s}_2}(2j)\dd P_j^2\dd \pi_{\mf{s}_2}(2j-1)\dd Q_{2j-1}\dd \pi_{\mf{s}_1}(2j-1)$$
is a cycle in $G$ of length at least $4o+4$. Moreover, for every $i\in \{1,2\}$, since $\mf{s}_i$ is a $2o$-ample $2c$-syzygy in $G$, by \ref{SY}, the paths $\{P_j^i:j\in [c]\}$ are pairwise disjoint and anticomplete in $G$. Hence, $H_1,\ldots,H_c$ are pairwise disjoint and anticomplete in $G$ (see Figure~\ref{fig:geminilem}). But this violates the assumption that $G$ is $(c,o)$-perforated, and so completes the proof of Lemma~\ref{lem:gemini_to_constellation}.
  \end{proof}
  \begin{figure}
      \centering
      \includegraphics[scale=0.9]{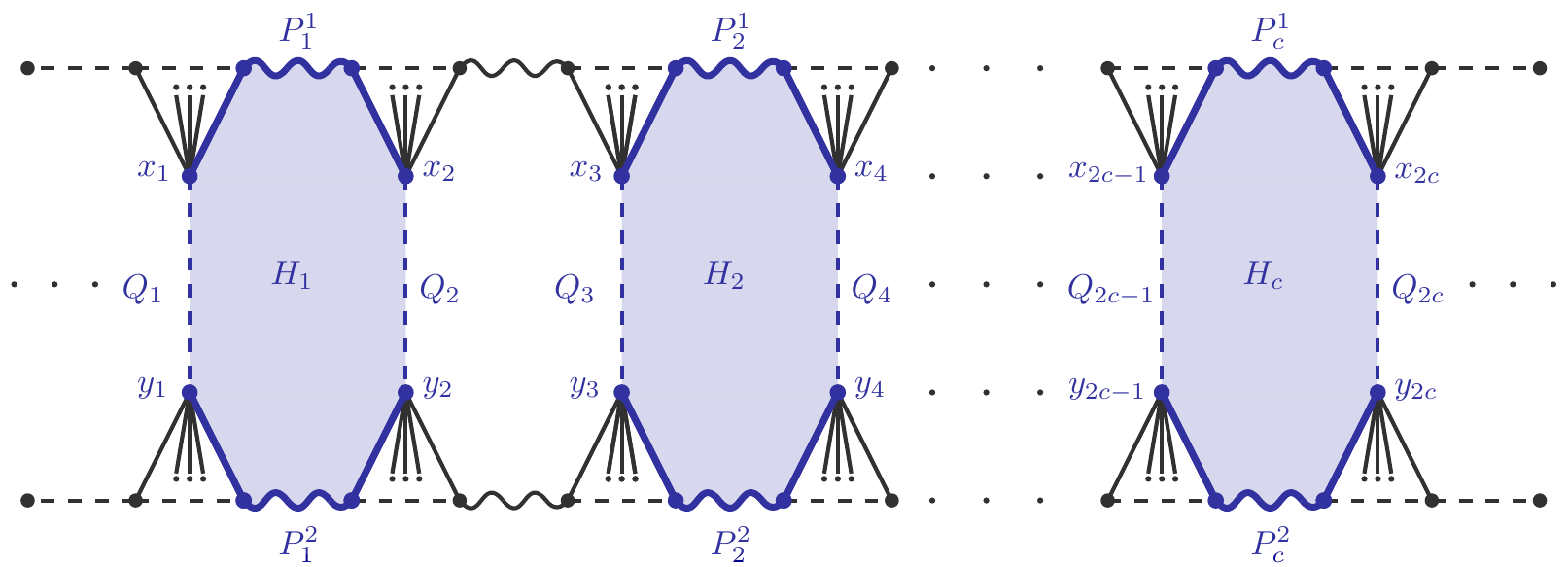}
      \caption{Proof of Lemma~\ref{lem:gemini_to_constellation}, with $\pi_{\mf{s}_1}(i)=x_i$ and $\pi_{\mf{s}_2}(i)=y_i$ for $i=1,\ldots, 2c$.}
      \label{fig:geminilem}
  \end{figure}
  
We are now in a position to prove Theorem~\ref{thm:bundlethm}, which we restate:

\setcounter{section}{5}
\setcounter{theorem}{0}

\begin{theorem}\label{thm:bundlethm}
    \sloppy For all   $a,b,c,h,l,s,o,t\in \poi$ and $s \in \poi \cup \{0\}$, there are constants $\Theta=\Theta(a,b,c,h,l,o,s,t)\in \poi$ and $\theta=\theta(a,b,c,h,l,o,s,t)\in \poi$ with the following property. Let $G$ be a $(c,o)$-perforated graph and let $\mf{T}$ be a collection of $\Theta$ pairwise disentangled plain $(\leq b,\theta)$-bundles in $G$. Then one of the following holds.
    \begin{enumerate}[\rm (a)]
        \item\label{thm:bundlethm_a} $G$ contains either $K_t$ or $K_{t,t}$.
        \item\label{thm:bundlethm_b} There exists a plain $(s,l)$-constellation in $G$.
        \item\label{thm:bundlethm_c} There exists $\mf{A}\subseteq \mf{T}$ with $|\mf{A}|=a$ as well as $\mca{H}_{\mf{b}}\subseteq \mca{L}_{\mf{b}}$ with $|\mca{H}_{\mf{b}}|=h$ for each $\mf{b}\in \mf{A}$, such that for all distinct $\mf{b},\mf{b}'\in \mf{A}$, $S_{\mf{b}}\cup V(\mca{H}_{\mf{b}})$ is anticomplete to $S_{\mf{b}'}\cup V(\mca{H}_{\mf{b}'})$ in $G$.
    \end{enumerate}
\end{theorem}

\begin{proof} Let
$\gamma=\gamma(c,l-1,l,o,s)\geq 1$
be as in Lemma~\ref{lem:gemini_to_constellation} and let $\varphi=\varphi(a,\max\{h,s(\gamma+4l-4)^{l}+1\})\geq 1$
be given by Lemma~\ref{lem:polypathvspolypath}. Let $\beta(\cdot,\cdot,\cdot,\cdot,\cdot,\cdot,\cdot)$ be as in Lemma~\ref{lem:bundle}, and define
$$\Theta=\beta(b,\varphi,l,a,1,s,t);$$
$$\theta=2b(l-1)+\varphi.$$
We show that the above choices of $\Theta$ and $\theta$ satisfy Theorem~\ref{thm:bundlethm}. Suppose for a contradiction that none of the three outcomes of Theorem~\ref{thm:bundlethm} holds.

By the choice of $\Theta$ and $\theta$, we can apply Lemma~\ref{lem:bundle} to $\mf{T}$. Note that Lemma~\ref{lem:bundle}\ref{lem:bundle_a} implies Theorem~\ref{thm:bundlethm}\ref{thm:bundlethm_a}. Also, since all the bundles in $\mf{T}$ are plain, it follows that Lemma~\ref{lem:bundle}\ref{lem:bundle_b} implies Theorem~\ref{thm:bundlethm}\ref{thm:bundlethm_b}. Therefore, Lemma~\ref{lem:bundle}\ref{lem:bundle_c} holds. Specifically, we have: 

\sta{\label{st:bundlelemapplied}There exists $\mf{A}\subseteq \mf{T}$ with $|\mf{A}|=a$ as well as $\mca{F}_{\mf{b}}\subseteq \mca{L}_{\mf{b}}$ with $|\mca{F}_{\mf{b}}|=\varphi$ for each $\mf{b}\in \mf{A}$, such that for all distinct $\mf{b},\mf{b}'\in \mf{A}$, $S_{\mf{b}}$ is anticomplete to $S_{\mf{b}'}\cup V(\mca{F}_{\mf{b}'})$ in $G$.}

Next, by the choice of $\varphi$, we may apply Lemma~\ref{lem:polypathvspolypath} to $\{\mca{F}_{\mf{b}}:\mf{b}\in \mf{A}\}$, and deduce that:

\sta{\label{st:pathlemmaapplied} For every $\mf{b}\in \mf{A}$,  there exists a $\max\{h,s(\gamma+4l-4)^{l}+1\}$-polypath $\mca{G}_{\mf{b}}\subseteq \mca{F}_{\mf{b}}$ such that for all distinct $\mf{b},\mf{b}'\in \mf{A}$ either $V(\mca{G}_{\mf{b}})$ is anticomplete to $V(\mca{G}_{\mf{b}'})$ in $G$, or for every $L\in \mca{G}_{\mf{b}}$ and every $L'\in \mca{G}_{\mf{b}'}$, there is an edge in $G$ with an end in $L$ and an end in $L'$.}

In particular, for every $\mf{b}\in \mf{A}$, one may pick $\mca{H}_{\mf{b}}\subseteq \mca{G}_{\mf{b}}$ with $|\mca{H}_{\mf{b}}|=h$. Note that, for distinct $\mf{b},\mf{b}'\in \mf{A}$, if $V(\mca{G}_{\mf{b}})$ is anticomplete to $V(\mca{G}_{\mf{b}'})$ in $G$, then by \eqref{st:bundlelemapplied}, $S_{\mf{b}}\cup V(\mca{H}_{\mf{b}})$ is anticomplete to $S_{\mf{b}'}\cup V(\mca{H}_{\mf{b}'})$. Consequently, from \eqref{st:pathlemmaapplied} and since Theorem~\ref{thm:bundlethm}\ref{thm:bundlethm_c} is assumed not to hold, it follows that:

\sta{\label{st:gettwogoods} There are distinct $\mf{b},\mf{b}'\in \mf{A}$ such that for every $L\in \mca{G}_{\mf{b}}$ and every $L'\in \mca{G}_{\mf{b}'}$, there is an edge in $G$ with an end in $L$ and an end in $L'$.}

Henceforth, let $\mf{b},\mf{b}'$ be as in \eqref{st:gettwogoods}. Since $|\mca{G}_{\mf{b}}|=|\mca{G}_{\mf{b}'}|\geq s(\gamma+4l-4)^l+1>\gamma+4l-4$, we may choose $\mca{G}\subseteq \mca{G}_{\mf{b}}$ and $\mca{G}'\subseteq \mca{G}_{\mf{b}'}$ with $|\mca{G}|=s(\gamma+4l-4)^l+1$ and $|\mca{G}'|=\gamma+4l-4$. It follows that both $\mca{G}$ and $\mca{G}'$ are plain polypaths in $G$. For every path $L\in \mca{G}$, let us say $L$ is \textit{rigid} if there exists a vertex $x_L\in L$ as well as $\mca{G}'_L\subseteq \mca{G}'$ with $|\mca{G}'_L|=l$ such that $x_L$ has a neighbor in every path in $\mca{G}'_L$. Let $\mca{R}\subseteq \mca{G}$ be the set of all rigid paths. Then:

\sta{\label{st:fewrigidpaths} 
        We have $|\mca{G}\setminus \mca{R}|\geq 2$, that is, there are two distinct paths $L_1,L_2\in \mca{G}$ such that for each $i\in \{1,2\}$, every vertex in $L_i$ has a neighbor in at most $l-1$ paths in $\mca{G}'$.}

Suppose not. Then we have $|\mca{R}|\geq s(\gamma+4l-4)^l$. For every $L\in \mca{R}$, let $x_L\in L$ and $\mca{G}'_L\subseteq \mca{G}'$ with $|\mca{G}'_L|=l$ be as in the definition of a rigid path. Since $|\mca{G}'|=\gamma+4l-4$, it follows that there exists $\mca{S}\subseteq \mca{R}\subseteq \mca{G}$ with $|\mca{S}|=s$ and $\mca{L}\subseteq \mca{G}'$ with $\mca{L}=l$ such that for every $L\in \mca{S}$, we have $\mca{G}'_L=\mca{L}$. Let $S=\{x_L:L\in \mca{S}\}$. Then every vertex in $S$ has a neighbor in every path in $\mca{L}$. But now since both $\mca{G}$ and $\mca{G}'$ are plain, it follows that $(S,\mca{L})$ is a plain $(s,l)$-constellation in $G$, and so Theorem~\ref{thm:bundlethm}\ref{thm:bundlethm_b} holds, a contradiction. This proves \eqref{st:fewrigidpaths}.
\medskip

Let $L_1,L_2\in \mca{G}$ be as in \eqref{st:fewrigidpaths}. Let $\mca{T}$ be the set of all paths $L'\in \mca{G}'$ for which some vertex in $\partial L_1\cup \partial L_2$ has a neighbor in $L'$ in $G$. Then by \eqref{st:fewrigidpaths}, we have $|\mca{T}|\leq 4l-4$. Consequently, we may choose $\gamma$ distinct paths $L'_1, \ldots, L'_{\gamma}\in \mca{G}'$ such that for every $i\in \gamma$, $\partial L_1\cup \partial L_2$ is anticomplete to $L_i$ in $G$. By \eqref{st:gettwogoods}, for every $i\in[\gamma]$, we may choose a shortest path $Q_j$ in $L'_j$ from a vertex $x_j^1\in L'_j$ with a neighbor in $L_1$ to a vertex $x_j^2\in L'_j$ with a neighbor in $L_2$. For $i\in \{1,2\}$, let $S_i=\{x_j^i:j\in [\gamma]\}$. It follows that every vertex in $S_i$ has a neighbor in $L_i$ while $S_i$ is anticomplete to $\partial L_i$. Therefore, $\mf{g}_i=(S_i,L_i)$ is an ordered $\gamma$-asterism in $G$ with $\pi_{\mf{g}_i}(j)=x_j^i$ for every $j\in [\gamma]$. Also, by \eqref{st:fewrigidpaths}, $\mf{g}_i$ is $(l-1)$-meager. 

Now, since both $\mca{G}$ and $\mca{G}'$ are plain, and by the choice of $Q_1,\ldots, Q_{\gamma}$, we conclude that $(\mf{g}_1,\mf{g}_2)$ is a pair of $(l-1)$-meager ordered $\gamma$-asterisms in $G$ satisfying \ref{G1} and \ref{G2}, and that $\mca{Q}=\{Q_1,\ldots, Q_{\gamma}\}$ satisfies \ref{G3}. Hence, $(\mf{g}_1,\mf{g}_2)$ is a $\gamma$-gemini in $G$ where both $\mf{g}_1$ and $\mf{g}_2$ are $(l-1)$-meager. But then Lemma~\ref{lem:gemini_to_constellation} along with the choice of $\gamma$ implies that there exists a plain $(s,l)$-constellation in $G$, and so Theorem~\ref{thm:bundlethm}\ref{thm:bundlethm_b} holds, a contradiction. This completes the proof of Theorem~\ref{thm:bundlethm}.
\end{proof}

\section{Transition graphs: a tale of an intuition coming true}\label{sec:surgery}
In this section, we take a substantial step towards the proof of Theorem~\ref{mainthmgeneral} by showing that:

 \begin{theorem}\label{thm:const_to_fire}
     For all   $c,o,t\in \poi$ and $s \in \poi \cup \{0\}$, there are constants $\Sigma=\Sigma(c,o,s,t) \in \poi \cup \{0\}$ and $\Lambda=\Lambda(c,o,s,t)\in \poi$ with the following property. Let $G$ be a $(c,o)$-perforated graph. Assume that there exists a plain $(\Sigma,\Lambda)$-constellation in $G$. Then $G$ contains either $K_t$ or $K_{t,t}$, or there is a full $(s,o)$-occultation in $G$.
\end{theorem}

The proof of Theorem~\ref{thm:const_to_fire} is centered around a rather intuitive idea, the ``matching/vertex-cover duality in transition graphs'' described in Section~\ref{sec:outline}. Our job here is to inject rigor and the main step will be taken in Lemma~\ref{lem:getinterrupted}. But first we need to add another lemma to our arsenal, which will also be used in the next section.

\begin{lemma}\label{lem:nosubdivision_ample}
    For all   $l,m\in \poi$ and $s \in \poi \cup \{0\}$, there is a constant $\psi=\psi(l,m,s)\in \poi$ with the following property. Let $G$ be a graph. Assume that there exists a $(\psi,l+m^2-1)$-constellation $\mf{c}$ in $G$. Then for every $d\in \poi$, one of the following holds. 
      \begin{enumerate}[\rm (a)]
          \item\label{lem:nosubdivision_ample_a} $G$ contains a $(\leq d)$-subdivision of $K_m$ as a subgraph.
          \item\label{lem:nosubdivision_ample_b} There exists $S\subseteq S_{\mf{c}}$ with $|S|=s$ and $\mca{L}\subseteq \mca{L}_\mf{c}$ with $|\mca{L}|=l$ such that for every $L\in \mca{L}$, $(S,L)$ is a $d$-ample $s$-asterism in $G$.
      \end{enumerate}
  \end{lemma}
  \begin{proof}
Let 
$$\psi=\psi(l,m,s)=\rho(\max\{m,s+2l\},2,2^{l+m^2-1});$$
where $\rho(\cdot,\cdot,\cdot)$ is as in Theorem~\ref{multiramsey}. For every $2$-subset $X$ of $S_{\mf{c}}$, let $\Phi(X)$ be the set of all paths $L\in \mca{L}_{\mf{c}}$ for which there exists a path $R$ in $G$ of length at most $d+1$ with $\partial R=X$ and $R^*\subseteq L$. It follows that the map $\Phi:\binom{S_{\mf{c}}}{2}\rightarrow 2^{\mca{L}_{\mf{c}}}$ is well-defined. From the choice of $\psi$ and Theorem~\ref{multiramsey}, we deduce that there exists $\mca{W}\subseteq \mca{L}_{c}$ as well as $Z\subseteq S_{\mf{c}}$ with $|Z|=\max\{m,s+2l\}$, such that for every $2$-subset $X$ of $Z$, we have $\Phi(X)=\mca{W}$.

First, assume that $|\mca{W}|\geq m^2$. Let $M\subseteq Z$ with $|M|=m$. Then for every $2$-subset $X$ of $M$, one may choose a path $L_{X}\in \mca{W}=\Phi(X)$ such that the paths $\{L_X:X\in \binom{M}{2}\}$ are pairwise distinct, and hence disjoint. Also, from the definition of $\Phi$, it follows that for every $2$-subset $X$ of $M$, there exists a path $R_X$ in $G$ of length at most $d+1$ with $\partial R_X=X$ and $R_X^*\subseteq L_X$. But then $G[\bigcup_{X\in \binom{M}{2}}R_X]$ contains a $(\leq d)$-subdivision of $K_m$ as a (spanning) subgraph, and so Lemma~\ref{lem:nosubdivision_ample}\ref{lem:nosubdivision_ample_a} holds, as desired.

Next, assume that $|\mca{W}|<m^2$. Then there exists $\mca{L}\subseteq \mca{L}_{\mf{c}}\setminus \mca{W}$ with $|\mca{L}|=l$. Pick a subset $S'$ of $Z$ with $|S'|=s+2l$. For every $2$-subset $X$ of $S'$ and every path $L\in \mca{L}$, since we have $\Phi(X)=\mca{W}$, it follows from the definition of $\Phi$ that there is no path $R$ in $G$ of length at most $d+1$ with $\partial L=X$ and $R^*\subseteq L$. In particular, since $d\geq 1$, it follows that every vertex in $V(\mca{L})$ has at most one neighbor in $S'$. Consequently, there exists $S\subseteq S'$ with $|S|=s$ which is anticomplete to $\partial \mca{L}$.  Now for every $L\in \mca{L}$, every vertex in $S$ has a neighbor in $L$ while $S$ is anticomplete to $\partial L$, and for every $2$-subset $X$ of $S$, there is no path $R$ in $G$ of length at most $d+1$ with $\partial L=X$ and $R^*\subseteq L$. Hence, for every $L\in \mca{L}$, $(S,L)$ is a $d$-ample $s$-asterism in $G$, and so $S$ and $\mca{L}$ satisfy Lemma~\ref{lem:nosubdivision_ample}\ref{lem:nosubdivision_ample_b}. This completes the proof of Lemma~\ref{lem:nosubdivision_ample}. \end{proof}

It is time to give a formal definition of transition graphs. Given a graph $G$ and an (ordered) asterism $\mf{a}$ in $G$, the \textit{transition graph $\mathsf{T}_{\mf{a}}$ of $\mf{a}$} is the graph with $V(\mathsf{T}_{\mf{a}})=S_{\mf{a}}$ such that for all distinct $x,y\in S_{\mf{a}}$, we have $xy\in E(\mathsf{T}_{\mf{a}})$ if and only if there exists an $\mf{a}$-route $R$ from $x$ to $y$ such that $S_{\mf{a}}\setminus \{x,y\}$ is anticomplete to $R^*$ (and so to $R$) in $G$; see Figure~\ref{fig:transition}). The following encompasses the bulk of difficulty in the proof of Theorem~\ref{thm:const_to_fire}.

\begin{figure}[t!]
\centering
\includegraphics[scale=1]{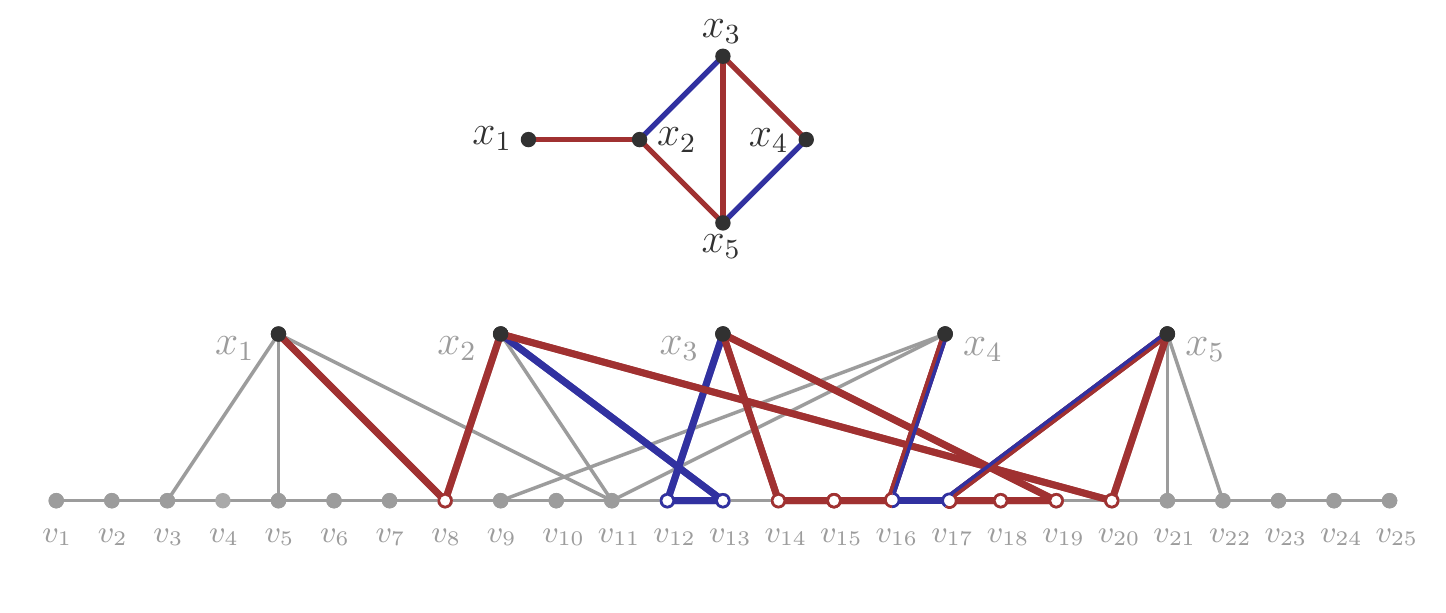}

% \begin{tikzpicture}[scale=0.65,auto=left]
% \tikzstyle{every node}=[inner sep=1.3pt, circle, draw]  
% \centering

% %%%%% S

% \node [label={270:$x_3$},fill] at(1,-1) {};
% \node [label={360:$x_4$},fill] at(2,0) {};
% \node [label=$x_2$,fill] at(1,1) {};
% \node [label={180:$x_1$},fill] at(0,0) {};

% %%%%%%% edges from S

% \draw [thick] (0,0) -- (1,1);
% \draw [thick] (0,0) -- (1,-1);
% \draw [thick] (2,0) -- (1,1);
% \draw [thick] (2,0) -- (1,-1);
% \draw [thick] (1,1) -- (1,-1);

% \end{tikzpicture}

\caption{The transition graph $\mathsf{T}_{\mf{a}}$ of the $5$-asterism $\mf{a}$ in Figure~\ref{fig:asterism} with the $\mf{a}$-routes $x_1\dd v_8\dd x_2$, $x_2\dd v_{13}\dd v_{12}\dd x_3$, $x_2\dd v_{20}\dd x_5$, $x_3\dd v_{14}\dd v_{15}\dd v_{16}\dd x_4$, $x_3\dd v_{19}\dd v_{18}\dd v_{17}\dd x_5$ and $x_4\dd v_{16}\dd v_{17}\dd x_5$ corresponding to the edges of $\mathsf{T}_{\mf{a}}$. Note that the unique $\mf{a}$-route $x_1\dd v_{11}\dd v_{12}\dd x_3$ from $x_1$ to $x_3$ contains neighbors of both  $x_2$ and $x_4$.}
\label{fig:transition}
\end{figure}

 \begin{lemma}\label{lem:getinterrupted}
      For all   $c,o,t\in \poi$ and $s \in \poi \cup \{0\}$, there are constants $\sigma=\sigma(c,o,s,t) \in \poi \cup \{0\}$ and $\lambda=\lambda(c,o,s,t)\in \poi$ with the following property. Let $G$ be a $(c,o)$-perforated graph. Assume that there exists a plain  $(\sigma,\lambda)$-constellation $\mf{c}$ in $G$. Then one of the following holds.
      \begin{enumerate}[\rm (a)]
          \item\label{lem:getinterrupted_a} $G$ contains either $K_t$ or $K_{t,t}$.
        \item\label{lem:getinterrupted_b} There exists an $(o+2)$-ample, interrupted ordered $s$-asterism $\mf{a}$ in $G$ such that $S_{\mf{a}}\subseteq S_{\mf{c}}$, and for some $L\in \mca{L}_\mf{c}$, we have $L_{\mf{a}}\subseteq  L$.
      \end{enumerate}
 \end{lemma}
 \begin{proof}
Let $c,o,t\in \poi$ be fixed. Let $H$ be the disjoint union of $c$ cycles, each of length exactly $o+2$ (so $H$ is the unique $2$-regular graph, up to isomorphism, with exactly $c$ components, each on $o+2$ vertices). Let $m=m(H,o+2,t)$ be as in Theorem~\ref{dvorak} (note that here $m$ only depends on $c,o$ and $t$). Let $\rho(\cdot,\cdot,\cdot)$ be as in Theorem~\ref{multiramsey} and let $\beta(\cdot,\cdot,\cdot,\cdot,\cdot,\cdot)$ be as in Lemma~\ref{lem:bundle}.

Consider the following recursive definition for $\sigma$ and $\lambda$. Let $\sigma(c,o,0,t)=0$ and $\lambda(c,o,0,t)=1$. For every $s\in \poi$, assuming $\sigma'=\sigma(c,o,s-1,t)$ and $\lambda'=\lambda(c,o,s-1,t)$ are defined, let
\medskip
$$\sigma_3=\sigma_3(c,o,s,t)=\beta(2,2,\lambda',c,1,\sigma',t);$$
$$\sigma_2=\sigma_2(c,o,s,t)=\rho(2\sigma_3,2,2^{2c}-1);$$ 
$$\sigma_1=\sigma_1(c,o,s,t)=2c+\sigma_2;$$
$$\lambda_2=\lambda_2(c,o,s,t)=2^{2c+1}\sigma_3(2\lambda'-1)\sigma_2^{2\sigma_3};$$
$$\lambda_1=\lambda_1(c,o,s,t)=2^{\sigma_1^2}\lambda_2.$$
\medskip
Now, define
$$\sigma=\sigma(c,o,s,t)=\psi(\lambda_1,m,{\sigma_1});$$ $$\lambda=\lambda(c,o,s,t)=\lambda_1+m^{2}-1;$$
where $\psi=\psi(\cdot,\cdot,\cdot)$ is as in Lemma~\ref{lem:nosubdivision_ample}. We will prove, by induction on $s$, that the above values of $\sigma$ and $\lambda$ satisfy Lemma~\ref{lem:getinterrupted}. If $s=0$, then Lemma~\ref{lem:getinterrupted}\ref{lem:getinterrupted_b} is immediately seen to hold. Assume that $s\geq 1$.

Let $G$ be a $(c,o)$-perforated graph and let $\mf{c}$ be a plain $(\sigma,\lambda)$-constellation in $G$. Suppose that Lemma~\ref{lem:getinterrupted}\ref{lem:getinterrupted_a} does not hold, that is, $G$ contains neither $K_t$ nor $K_{t,t}$. Since $G$ is $(c,o)$-perforated, it follows that $G$ contains no induced subgraph isomorphic to a subdivision of $H$. Therefore, by Theorem~\ref{dvorak} and the choice of $m$, it follows that $G$ contains no subgraph isomorphic to a $(\leq o+2)$-subdivision of $K_m$. This, along with the choices of $\sigma$ and $\lambda$, as well as Lemma~\ref{lem:nosubdivision_ample}\ref{lem:nosubdivision_ample_b}, implies that there exists $S\subseteq S_{\mf{c}}$ with $|S|=\sigma_1$ and $\mca{L}_1\subseteq \mca{L}_{\mf{c}}$ with $|\mca{L}_1|=\lambda_1$ such that for every $L\in \mca{L}_1$, $\mf{f}_L=(S,L)$ is an $(o+2)$-ample $\sigma_1$-asterism in $G$. In particular, $\mf{f}=(S,\mca{L}_1)$ is a plain $(\sigma_1,\lambda_1)$-constellation in $G$.

Since $|\mca{L}_1|=\lambda_1=2^{\sigma_1^2}\lambda_2$, it follows that:

\sta{\label{sametransition}There exists $\mca{L}_2\subseteq \mca{L}_1$ with $|\mca{L}_2|=\lambda_2$ as well as $E_2\subseteq \binom{S}{2}$, such that for every $L\in \mca{L}_2$, we have $E(\mathsf{T}_{\mf{f}_L})=E_2$.}

Henceforth, let $\mca{L}_2$ and $E_2$ be as in \eqref{sametransition}. Let $\mathsf{T}$ be the graph with $V(\mathsf{T})=S$ and $E(\mathsf{T})=E_2$.  Then \eqref{sametransition} implies that $\mathsf{T}_{\mf{f}_L}=\mathsf{T}$ for every $L\in \mca{L}_2$. We deduce:

\sta{\label{transitionmagic} $\mathsf{T}$ does not have a matching of cardinality $c$.}

Suppose for a contradiction that there exists a matching  $\{x_1x_1',\ldots, x_cx_c'\}\subseteq E(\mathsf{T})$ of cardinality $c$ in $\mathsf{T}$. Since $|\mca{L}_2|=\lambda_2\geq 2$, we may choose two distinct paths $L_1,L_2\in \mca{L}_2$. Then we have $\{x_1x_1',\ldots, x_cx_c'\}\subseteq E(\mathsf{T}_{\mf{f}_{L_1}})=E(\mathsf{T}_{\mf{f}_{L_2}})$. Also,  since $\mf{f}_{L_1}$ and $\mf{f}_{L_2}$ are both $(o+2)$-ample, $\mf{f}$ is plain, and from the definition of the transition graph, it follows that for every $i\in \{1,2\}$ and $j\in [c]$, there exists an $\mf{f}_{L_i}$-route $R_{i,j}$ of length at least $o+4$ from $x_j$ to $x_j'$, such that for all distinct $j,j'\in [c]$, $H_j=R_{1,j}\cup R_{2,j}$ and $H_j'=R_{1,j'}\cup R_{2,j'}$ are disjoint and anticomplete cycles in $G$, each of length at least $2o+8$. In other words, $\{H_j:j\in [c]\}$ is a collection of $c$ pairwise disjoint and anticomplete cycles in $G$, each of length at least $2o+8$ (see Figure~\ref{fig:matchingmagic}). But this violates the assumption that $G$ is $(c,o)$-perforated, and so proves \eqref{transitionmagic}.
\begin{figure}[t!]
    \centering
\includegraphics[scale=0.9]{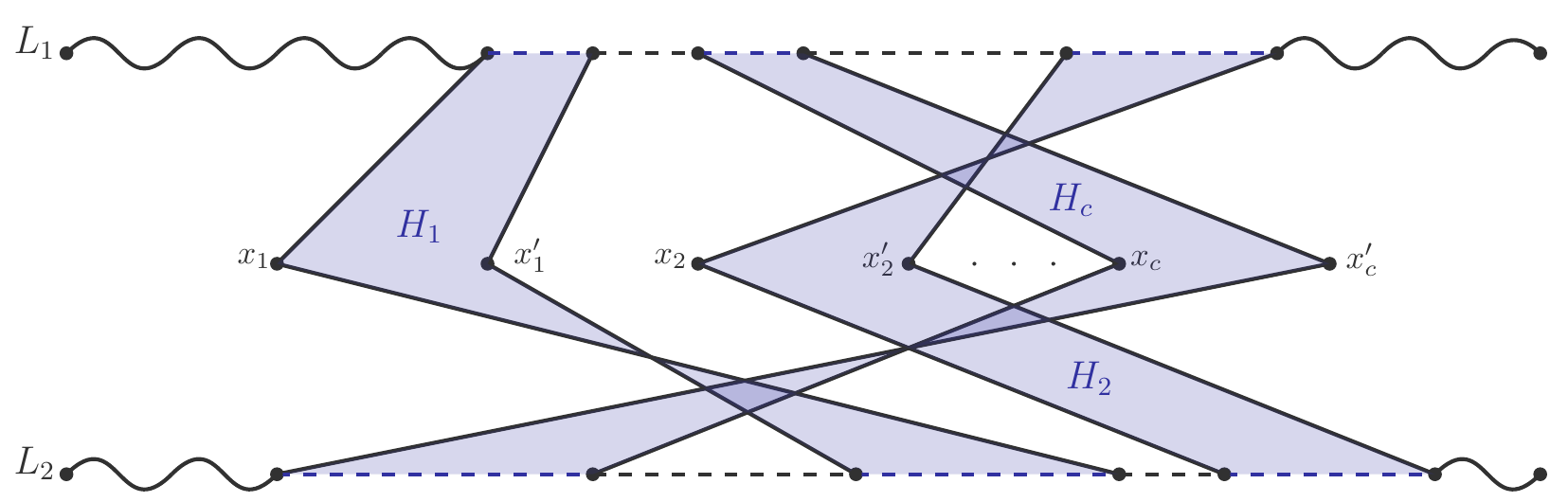}
    \caption{Proof of \eqref{transitionmagic} (each dashed segment represents a path of length at least $o+1$ and each squiggly segment depicts a path of arbitrary yet non-zero length).}
    \label{fig:matchingmagic}
\end{figure}
\medskip

By \eqref{transitionmagic}, there exists a vertex cover $X\subseteq S=V(\mathsf{T})$ for $\mathsf{T}$ with  $|X|<2c$. Consequently, $S\setminus X=V(\mathsf{T})\setminus X$ is a stable set in $\mathsf{T}$. Moreover, we have:

\sta{\label{vertexcover} Let $L\in \mca{L}_2$. Then for every $\mf{f}_L$-route $R$, there exists a vertex $x\in X$ which has a neighbor in $R^*$.}

Suppose not. Then we may pick $L\in \mca{L}_2$ and a minimal $\mf{f}_L$-route $R$ such that $X$ is anticomplete to $R^*$. Let $z,z'\in S$ be the ends of $R$; so $z,z'\in S\setminus X=V(\mathsf{T})\setminus X$. Since $\mf{f}_L$ is $(o+2)$-ample and from the minimality of $R$, it follows that $S\setminus (X\cup \{z,z'\})$ is anticomplete to $R^*$. We conclude that $S\setminus \{z,z'\}=V(\mathsf{T}_{\mf{f}_L})\setminus \{z,z'\}$ is anticomplete to $R^*$. But then from the definition of the transition graph, it follows that $zz'\in E(\mathsf{T}_{\mf{f}_L})=E(\mathsf{T})$, a contradiction with the fact that $V(\mathsf{T})\setminus X$ is a stable set in $\mathsf{T}$. This proves \eqref{vertexcover}.
\medskip

In view of \eqref{vertexcover}, for every $L\in \mca{L}_2$ and all distinct $z,z'\in S\setminus X$, we may choose a minimal non-empty subset $\Phi_{L}(\{z,z'\})$ of $X$ such that for every $\mf{f}_L$-route $R$ from $z$ to $z'$, there exists a vertex $x\in \Phi_{L}(\{z,z'\})$ which has a neighbor in $R^*$. It follows that, for every $L\in \mca{L}_2$, the map $\Phi_{L}: \binom{S\setminus X}{2}\rightarrow 2^{X}\setminus \{\emptyset\}$ is well-defined. Since $|S\setminus X|=\sigma_1-2c=\sigma_2=\rho(2\sigma_3,2,2^{2c}-1)$, applying Theorem~\ref{multiramsey} yields the following:

\sta{\label{presameminimalcover} For every $L\in \mca{L}_2$, there exists $Y_L\subseteq X$ and $Z_{L}\subseteq S\setminus X$ such that:
\begin{itemize}
    \item we have $Y_{L}\neq \emptyset$ and $|Z_{L}|=2\sigma_3$; and
    \item for all distinct $z,z'\in Z_{L}$, we have $\Phi_{L}(\{z,z'\})=Y_{L}$.
\end{itemize}}

Combining \eqref{presameminimalcover} with the fact that $|\mca{L}_2|=\lambda_2=2^{2c+1}\sigma_3(2\lambda'-1)\sigma_2^{2\sigma_3}$, it follows from a pigeonhole argument that for several choices of $L\in \mca{L}_2$ for which the sets $Y_L$ associated with them are all the same, and so are the sets $Z_L$. More precisely, we deduce:

\sta{\label{sameminimalcover}There exist $\mca{L}_3\subseteq \mca{L}_2$, $Y\subseteq X$ and $Z\subseteq S\setminus X$, such that:
\begin{itemize}
    \item we have $|\mca{L}_3|=\sigma_3(4\lambda'-2)$, $Y\neq \emptyset$ and $|Z|=2\sigma_3$; and
    \item for every $L\in \mca{L}_3$ and all distinct $z,z'\in Z$, we have $\Phi_{L}(\{z,z'\})=Y$.
\end{itemize}}

By the first bullet of \eqref{sameminimalcover}, we may fix a partition $\{\mca{L}_3^i:i\in [\sigma_3]\}$ of $\mca{L}_3$ into $(4\lambda'-2)$-subsets, a vertex $y\in Y$, and a partition $\{\{z_i,z_i'\}:i\in [\sigma_3]\}$ of $Z$ into $2$-subsets. By the second bullet of \eqref{sameminimalcover}, for every $L\in \mca{L}_3$ and every $i\in [\sigma_3]$, we have $y\in Y=\Phi_{L}(\{z_i,z_i'\})$. This, together with the minimality of $\Phi_{L}(\{z_i,z_i'\})$, implies that for every $L\in \mca{L}_3$ and $i\in [\sigma_3]$, there exists an $\mf{f}_L$-route $Q_{i,L}$ from $z_i$ to $z_i'$ such that, writing $R_{i,L}=Q^*_{i,L}$, we have that $y$ is the only vertex in $Y$ with a neighbor in $R_{i,L}$.

Now, for every $i\in [\sigma_3]$, let $\mf{b}_i$ be the $(2,4\lambda'-2)$-bundle in $G$ with $S_{\mf{b}_i}=\{z_i,z_i'\}$ and  $\mca{L}_{\mf{b}_i}=\{R_{i,L}:L\in \mca{L}_3^i\}$. We claim that:

\sta{\label{vertexoflargeoutdegree} There exist $i\in [\sigma_3]$, $S'\subseteq Z\setminus \{z_i,z_i'\}$ with $|S'|=\sigma'$ and $\mca{G}\subseteq \mca{L}_{\mf{b}_i}$ with $|\mca{G}|=\lambda'$ such that $(S',\mca{G})$ is a plain $(\sigma',\lambda')$-constellation in $G$.}

Recall that $\sigma_3=\beta(2,2,\lambda',c,1,\sigma',t)$. Thus, we may apply Lemma~\ref{lem:bundle} to $\mf{B}=\{\mf{b}_i: i\in [\sigma_3]\}$. Since $G$ is assumed not to contain $K_t$ or $K_{t,t}$, it follows that Lemma~\ref{lem:bundle}\ref{lem:bundle_a} does not hold. Assume that Lemma~\ref{lem:bundle}\ref{lem:bundle_c} holds. Then there exists $I\subseteq [\sigma_3]$ with $|I|=c$ as well as $L^i_1,L^i_2\in \mca{L}_{3}^i$ for each $i\in I$, such that for all distinct $i,j\in I$, $\{z_i,z_i'\}$ is anticomplete to $\{z_j,z_j'\}\cup R_{j,L^j_1}\cup R_{j,L^j_2}$. Since $\mca{L}_3$ is a plain polypath and $\mf{f}_L$ is $(o+2)$-ample for every $L\in \mca{L}_3$, it follows that for all distinct $i,j\in I$, $Q_{i,L^i_1}\cup Q_{i,L^i_2}$ and $Q_{j,L^j_1}\cup Q_{j,L^j_2}$ are two disjoint and anticomplete cycles in $G$, each of length at least $2o+8$. In other words, $\{Q_{i,L^i_1}\cup Q_{i,L^i_2}:i\in I\}$ is a collection of $c$ pairwise disjoint and anticomplete cycles in $G$, each of length at least $2o+8$, a contradiction with $G$ being $(c,o)$-perforated. It follows that Lemma~\ref{lem:bundle}\ref{lem:bundle_b} holds, that is, there exist $i\in [\sigma_3]$, $S'\subseteq Z\setminus S_{\mf{b}_i}=Z\setminus \{z_i,z_i'\}$ with $|S'|=\sigma'$ and $\mca{G}\subseteq \mca{L}_{\mf{b}_i}$ with $|\mca{G}|=\lambda'$ such that $(S',\mca{G})$ is a $(\sigma',\lambda')$-constellation in $G$. In addition, $(S',\mca{G})$ is plain, because $\mf{f}$ is. This proves \eqref{vertexoflargeoutdegree}.
\medskip

Henceforth, let us fix $i\in [\sigma_3]$  $S'\subseteq Z\setminus \{z_i,z_i'\}$ and $\mca{G}\subseteq \mca{L}_{\mf{b}_i}$ as given by \eqref{vertexoflargeoutdegree}. Now we apply the induction hypothesis to show that:

\sta{\label{ihsurgery} For some $L\in \mca{L}_3^i$, there exists an $(o+2)$-ample, interrupted ordered $(s-1)$-asterism $\mf{a}'$ in $G$ such that
$S_{\mf{a}'}\subseteq S'$ and $L_{\mf{a}'}\subseteq R_{i,L}$.}

By \eqref{vertexoflargeoutdegree}, $\mf{c}'=(S',\mca{G})$ is a plain $(\sigma',\lambda')$-constellation in $G$. Since $G$ is assumed not to contain $K_t$ or $K_{t,t}$, by the induction hypothesis applied to $\mf{c}'$, there exists an $(o+2)$-ample, interrupted ordered $(s-1)$-asterism $\mf{a}'$ in $G$ such that
$S_{\mf{a}'}\subseteq S_{\mf{c}'}=S'$, and for some $L'\in \mca{L}_{\mf{c}'}=\mca{G}\subseteq \mca{L}_{\mf{b}_i}$, we have $L_{\mf{a}'}\subseteq  L'$. From the definition of $\mf{b}_i$, it follows that for some $L\in \mca{L}_3^i$, we have $L_{\mf{a}'}\subseteq R_{i,L}$. This proves \eqref{ihsurgery}.
\medskip

Let $L\in \mca{L}_3^i$ be as in \eqref{ihsurgery}. Then we may choose an $(o+2)$-ample, interrupted ordered $(s-1)$-asterism $\mf{a}'$ in $G$ such that
$S_{\mf{a}'}\subseteq S'$, $L_{\mf{a}'}\subseteq R_{i,L}$, and  $L_{\mf{a}'}$ is maximal with respect to inclusion. By \eqref{vertexoflargeoutdegree}, every vertex in $S'$ has a neighbor in $R_{i,L})$ in $G$. Also, recall that the vertex $y\in X\subseteq S\setminus Z\subseteq S\setminus S'$ has a neighbor in $R_{i,L}=Q^*_{i,L}$. In fact, since $\mf{f}_L$ is $(o+2)$-ample, it follows that every vertex in $S'\cup \{y\}$ has a neighbor in $R_{i,L}^*$ and $S'\cup \{y\}$ is anticomplete to the ends of $\partial R_{i,L}$. As a result $\mf{a}^+=(S'\cup \{y\}, R_{i,L})$ is a $(\sigma'+1)$-asterism in $G$ such that $S_{\mf{a}'}\subseteq S'= S_{\mf{a}^+}\setminus \{y\}$ and  $L_{\mf{a}'}\subseteq R_{i,L}=L_{\mf{a}^+}$.
We further deduce that:

\sta{\label{st:y_to_open_pieces} $\mf{a}'$ is a $(\mf{a}^+,y,s-1)$-candidate in $G$ and $y$ is a cherry on top of $\mf{a}^+|S_{\mf{a}'}$ in $G$.}

From the maximality of $L_{\mf{a}'}\subseteq R_{i,L}$, it follows immediately that $\mf{a}^+$, $\mf{a}'$ and $y$ satisfy \ref{CA}, and so $\mf{a}'$ is a $(\mf{a}^+,y,s-1)$-candidate in $G$. Let us now prove that $y$ is a cherry on top of $\mf{a}^+|S_{\mf{a}'}$ in $G$. We need to show that $\mf{a}^+|S_{\mf{a}'}$ and $y$ satisfy \ref{CH1} and \ref{CH2}. Observe that \ref{CH1} follows from the fact that $L_{\mf{a}^+|S_{\mf{a}'}}=L_{\mf{a}^+}$. To see \ref{CH2}, let $P$ be an open $\mf{a}^+|S_{\mf{a}'}$-piece. It follows that $P$ is the interior an $\mf{f}_L$-route between two vertices in $S'\cup \{z_i,z'_i\}$. On the other hand, since $S'\cup \{z_i,z'_i\}\subseteq Z$ and $L\in \mca{L}_3^i\subseteq \mca{L}_3$, it follows from the second bullet of \eqref{sameminimalcover} that some vertex in $Y$ has a neighbor in $P$. But then from $P\subseteq R_{i,L}\subseteq Q_{i,L}$ and the choice of $Q_{i,L}$, we conclude that $y$ has a neighbor in $P$. This proves \eqref{st:y_to_open_pieces}.
\medskip

We can now finish the proof. Note that $\mf{a}^+$ is $(o+2)$-ample, because $\mf{f}_L$ is. Therefore, since $o+2\geq 3$, in view of \eqref{st:y_to_open_pieces}, we can apply Lemma~\ref{lem:interrupted_top} to $\mf{a}^+, \mf{a}'$ and $y$, and deduce that $\mf{a}=\cher(\mf{a}',y)$ is an $(o+2)$-ample, interrupted ordered $s$-asterism in $G$ with $S_{\mf{a}}\subseteq S_{\mf{a}^+}=S'\cup \{y\}\subseteq S\subseteq S_{\mf{c}}$ and  $L_{\mf{a}}=L_{\mf{a}^+}=R_{i,L}\subseteq L$ where $L\in \mca{L}_3^i\subseteq \mca{L}_3\subseteq \mca{L}_2\subseteq \mca{L}_1\subseteq \mca{L}_{\mf{c}}$. Hence, Lemma~\ref{lem:getinterrupted}\ref{lem:getinterrupted_b} holds. This completes the proof of Lemma~\ref{lem:getinterrupted}.
\end{proof}

Now Theorem~\ref{thm:const_to_fire} becomes almost immediate:

\begin{proof}[Proof of Theorem~\ref{thm:const_to_fire}] \sloppy Let
 $\Sigma=\Sigma(c,o,s,t)=\sigma(c,o,s^c,t)$ and $\Lambda=\Lambda(c,o,s,t)=\lambda(c,o,s^c,t)$,  where $\sigma(\cdot,\cdot,\cdot,\cdot)$ and $\lambda(\cdot,\cdot,\cdot,\cdot)$ are as in Lemma~\ref{lem:getinterrupted}. Let $G$ be a $(c,o)$-perforated graph and let $\mf{c}$ be a plain $(\Sigma,\Lambda)$-constellation in $G$. Assume that $G$ does not contain $K_t$ or $K_{t,t}$. By Lemma~\ref{lem:getinterrupted}\ref{lem:getinterrupted_b}, there exists an $(o+2)$-ample, interrupted ordered $s^c$-asterism in $G$.  Since $o+2\geq 3$, by Theorem~\ref{thm:interrupted_to_occultation},  there is a full $(s,o)$-occultation in $G$. This completes the proof of Theorem~\ref{thm:const_to_fire}.
 \end{proof}

\section{Patch \& Match}\label{sec:patch}
Here we obtain the last main ingredient in the proof of Theorem~\ref{mainthmgeneral}: 

\begin{theorem}\label{thm:block_to_constellation}
For all   $c,l,o,s,t\in \poi$, there is a constant $
       \Omega=\Omega(c,l,o,s,t)\in \poi$ with the following property. Let $G$ be a $(c,o)$-perforated graph. Assume that there exists a strong $\Omega$-block in $G$. Then one of the following holds.
       \begin{enumerate}[\rm (a)]
\item\label{thm:block_to_constellation_a} $G$ contains either $K_t$ or $K_{t,t}$.
\item\label{thm:block_to_constellation_b} There exists a plain $(s,l)$-constellation in $G$.
      \end{enumerate}
  \end{theorem}

We need a couple of definitions and a lemma. Let $G$ be a graph, let $d \in \poi \cup \{0\}$ and let $r\in \poi$. For $X\subseteq V(G)$, by a \textit{$(d,r)$-patch for $X$ in $G$} we mean a $(1,r)$-bundle $\mf{p}$ in $G$ where:
\begin{enumerate}[(P1), leftmargin=19mm, rightmargin=7mm]
     \item\label{P1} we have $S_{\mf{p}}\subseteq V(G)\setminus V(\mca{L}_{\mf{p}})$;
    \item\label{P2} every path $L\in \mca{L}_{\mf{p}}$ has length at least $d$; and
    \item\label{P3} for every $L\in \mca{L}_{\mf{p}}$, one may write $\partial L=\{x_L,y_L\}$ such that $L\cap X=\{x_L\}$ and $N_L(S_\mf{p})=\{y_L\}$.
 \end{enumerate}
 Also, by a \textit{$(d,r)$-match for $X$ in $G$} we mean an $r$-polypath $\mca{M}$ in $G$ such that 

 \begin{enumerate}[(M1), leftmargin=19mm, rightmargin=7mm]
 \item\label{M1} every path $L\in \mca{M}$ has length at least $d$; and
 \item\label{M2} $V(\mca{M})\cap X=\partial \mca{M}$.
 \end{enumerate}
 \begin{figure}[t!]
    \centering
    \includegraphics[scale=0.9]{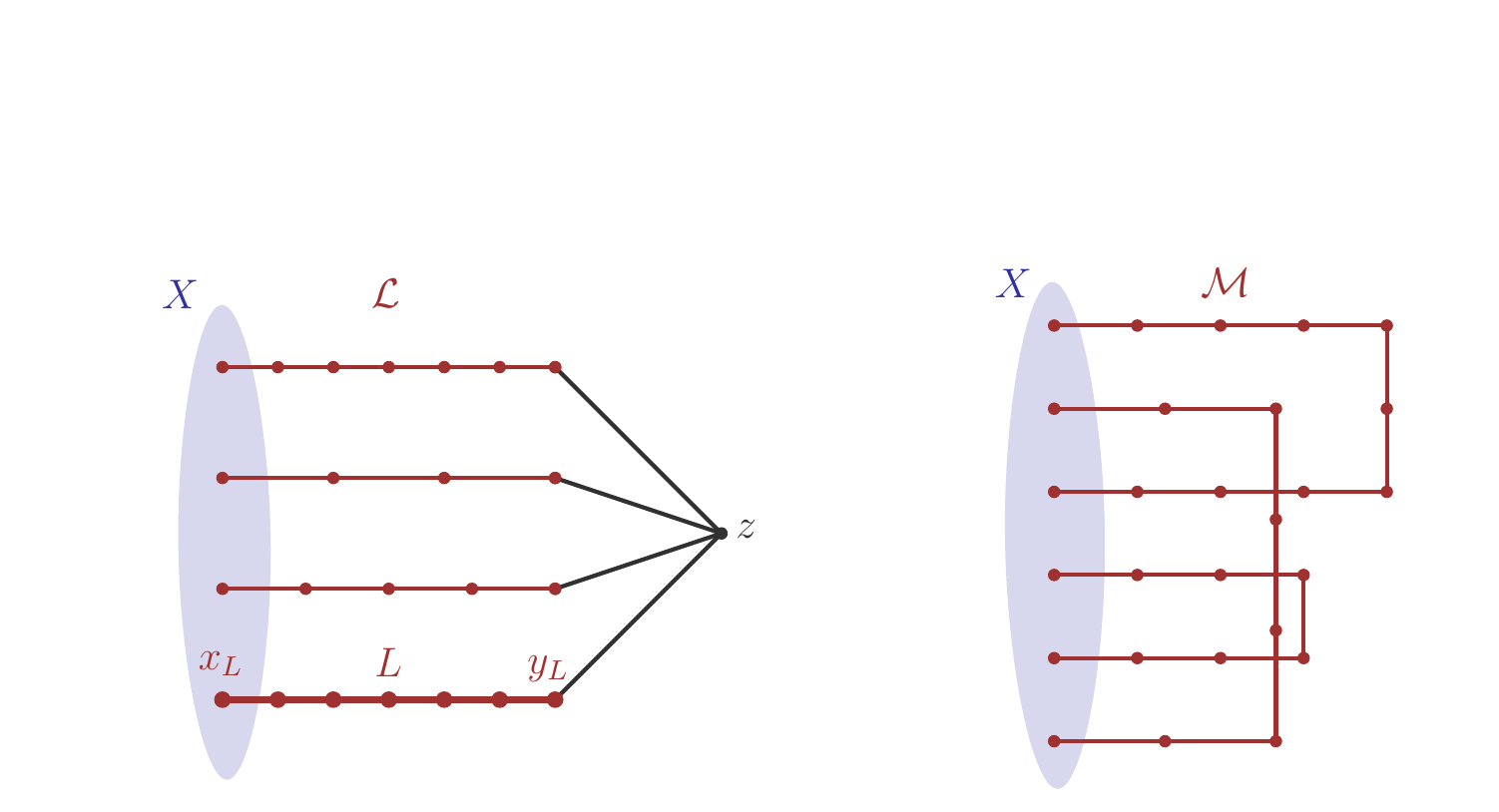}
    \caption{Left: A $(3,4)$-patch $(\{z\},\mca{L})$ for $X$ in $G$. Right: A $(7,3)$-match $\mca{M}$ for $X$ in $G$.}
    \label{fig:patch&match}
\end{figure}
See Figure~\ref{fig:patch&match}. 

 \begin{lemma}\label{lem:digraph}
       \sloppy For all   $d,s \in \poi \cup \{0\}$ and $l,m,r,r'\in \poi$, there is a constant $
       \eta=\eta(d,l,m,r,r',s)\in \poi$ with the following property. Let $G$ be a graph, let $X\subseteq V(G)$ and let $\mf{p}$ be a $(d,\eta)$-patch for $X$ in $G$. Then one of the following holds.
    \begin{enumerate}[\rm (a)]
        \item\label{lem:digraph_a} $G$ contains a $(\leq d+1)$-subdivision of $K_m$ as a subgraph.
        \item\label{lem:digraph_b} There exists a plain $(s,l)$-constellation in $G$.
         \item\label{lem:digraph_c} There exists a plain $(2(d+1),r)$-match $\mca{M}$ for $X$ in $G$ such that $ V(\mca{M})\subseteq V(\mf{p})$.
        \item\label{lem:digraph_d} There exists a plain $(d,r')$-patch $\mf{q}$ for $X$ in $G$ such that $V(\mf{q}) \subseteq V(\mf{p})$.
    \end{enumerate}
  \end{lemma}
  \begin{proof}
 Let $l,m,r,r'\in \poi$ and $s \in \poi \cup \{0\}$ be fixed. In order to define $\eta$,
 first we recursively define a sequence $\{\zeta_{d}: d \in \poi \cup \{0\}\}$, as follows. Let $h=\max\{m,4r,r'+1\}$ and let
 $$\zeta_{0}=h^{8h^4}.$$
 Let $\psi=\psi(1,m,(4r)^{l-1}(s+l-1))$ be as in Lemma~\ref{lem:nosubdivision_ample}. For every $d\in \poi$, let
 $$\zeta_{d}=\beta(1,1,1,\zeta_{d-1}, m^{2}, \psi,m^2);$$
 where $\beta(\cdot,\cdot,\cdot,\cdot,\cdot,\cdot,\cdot)$ is as in Lemma~\ref{lem:bundle}. This concludes the definition of $\zeta_d$ for all $d \in \poi \cup \{0\}$.
 
 Now, we claim that for every $d \in \poi \cup \{0\}$, $$\eta(d,l,m,r,r',s)=\zeta_{d+1}$$
satisfies Lemma~\ref{lem:digraph}. Suppose not. Choose $d \in \poi \cup \{0\}$ minimum such that the above value of $\eta(d,l,m,r,r',s)$ fails to satisfy Lemma~\ref{lem:digraph}.

Let $G$ be a graph, let $X\subseteq V(G)$ and let $\mf{p}$ be a $(d,\eta(d,l,m,r,r',s))$-patch for $X$ in $G$. For every $L\in \mca{L}_{\mf{p}}$, let $\partial L=\{x_L,y_L\}$ be as in \ref{P3}. We claim that:

\sta{\label{st:bundleleminpatch} There exist $\mca{P}\subseteq \mca{L}_{\mf{p}}$ with $|\mca{P}|=\zeta_d$ such that for all distinct $L,L'\in \mca{P}$, $x_L$ is anticomplete to $L'$ in $G$.}

For each $L\in \mca{L}_{\mf{p}}$, let $\mf{b}_{L}$ be the $(1,1)$-bundle in $G$ with $S_{\mf{b}_L}=\{x_L\}$ and $\mca{L}_{\mf{b}_L}=\{L\}$. Then $\mf{B}=\{\mf{b}_L:L\in \mca{L}_{\mf{p}}\}$ is a collection of $\eta(d,l,m,r,r',s)=\zeta_{d+1}$ pairwise disentangled $(1,1)$-bundles in $G$. Now, the definition of $\zeta_{d}$ for $d\in \poi$ allows for an application of Lemma~\ref{lem:bundle} to $\mf{B}$, which implies that one of the following holds.
\begin{itemize}
    \item $G$ contains either $K_{m^2}$ or $K_{m^2,m^2}$.
        \item There exist  $\mf{N}\subseteq \mf{B}$ with $|\mf{N}|=m^{2}$ as well as $S\subseteq \bigcup_{\mf{b}\in \mf{B}\setminus \mf{N}} S_{\mf{b}}$ with $|S|=\psi$, such that for every $\mf{b}\in \mf{N}$, there exists $\mca{G}_{\mf{b}}\subseteq \mca{L}_{\mf{b}}$ with $|\mca{G}_{\mf{b}}|=1$ for which $(S, \mca{G}_{\mf{b}})$ is a $(\psi,1)$-constellation in $G$. As a result, $(S,\bigcup_{\mf{b}\in \mf{N}}\mca{G}_{\mf{b}})$ is a $(\psi,m^2)$-constellation in $G$. 
        \item There exist  $\mf{M}\subseteq \mf{B}$ with $|\mf{M}|=\zeta_d$ as well as $\mca{F}_{\mf{b}}\subseteq \mca{L}_{\mf{b}}$ with $|\mca{F}_{\mf{b}}|=1$ for each $\mf{b}\in \mf{M}$, such that for all distinct $\mf{b},\mf{b}'\in \mf{M}$, $S_{\mf{b}}$ is anticomplete to $S_{\mf{b}'}\cup V(\mca{F}_{\mf{b}'})$ in $G$.
    \end{itemize}
If the first bullet above holds, then $G$ contains a $(\leq 1)$-subdivision of $K_m$ as a subgraph, and so Lemma~\ref{lem:digraph}\ref{lem:digraph_a} holds, a contradiction. Assume that the second bullet above holds. Then there exists $\mca{L}\subseteq \mca{L}_{\mf{p}}$ with $|\mca{L}|=m^{2}$ as well as $S\subseteq \{x_L:L\in \mca{L}_{\mf{p}}\setminus \mca{L}\}$ with $|S|=\psi$ such that $\mf{c}=(S,\mca{L})$ is an $(\psi,m^{2})$-constellation in $G$. Now we can apply Lemma~\ref{lem:nosubdivision_ample} to $\mf{c}$. Since Lemma~\ref{lem:digraph}\ref{lem:digraph_a} is assumed not to hold, it follows that Lemma~\ref{lem:nosubdivision_ample}\ref{lem:nosubdivision_ample_b} holds, that is, there exists $S'\subseteq S$ with $|S'|=(4r)^{l-1}(s+l-1)$ and $L\in \mca{L}$ such that $\mf{g}=(S',L)$ is a $(d+1)$-ample $((4r)^{l-1}(s+l-1))$-asterism in $G$. In particular, $\mf{g}$ is $1$-meager. So one may apply Lemma~\ref{lem:ast_to_syzygy} to $\mf{g}$. This time, since Lemma~\ref{lem:digraph}\ref{lem:digraph_b} is assumed not to hold, it follows that Lemma~\ref{lem:ast_to_syzygy}\ref{lem:ast_to_syzygy_a} holds, and so there exists a $4r$-syzygy $\mf{s}$ in $G$ with $S_{\mf{s}}\subseteq S'$ and $L_{\mf{s}}\subseteq L$. Also, $\mf{s}$ is $(d+1)$-ample, because $\mf{g}$ is. Let $u$ be the end of $L_{\mf{s}}\subseteq L$ for which $\mf{s}$ satisfies \ref{SY}. For every $i\in [r]$, let $M_i$ be the shortest path in $G[V(\mf{s}_i)]$ from  $\pi_{\mf{s}}(4i-3)$ to $\pi_{\mf{s}}(4i-1)$ with $M_i^*\subseteq L_{\mf{s}}$; in fact, we have $M_i^*\subseteq L_{\mf{s}}^*\subseteq L^*\subseteq V(G)\setminus X$. Since $\mf{s}$ is $(d+1)$-ample and from \ref{SY}, it follows that $M_1,\ldots, M_r$ are pairwise disjoint and anticomplete, each of length at least $2(d+1)>2d+1$. Thus, $\mca{M}$ is a plain $r$-polypath in $G$ such that every path in $\mca{M}$ has length at least $2(d+1)$. Also, we have $V(\mca{M})\cap X=\{\pi_{\mf{s}}(4i-3),\pi_{\mf{s}}(4i-1):i\in [r]\}=\partial \mca{M}$. From this, \ref{M1} and \ref{M2}, we conclude that $\mca{M}$ is a plain $(2(d+1),r)$-match for $X$ with $V(\mca{M})\subseteq \partial \mca{L}_{\mf{p}}\cup L^*\subseteq V(\mf{p})$ (see Figure~\ref{fig:match}). But now Lemma~\ref{lem:digraph}\ref{lem:digraph_c} holds, a contradiction. It follows that the third bullet above holds. Let $\mca{P}=\{L\in \mca{L}_{\mf{p}}:\mf{b}_L\in \mf{M}\}$. Then we have $\mca{P}\subseteq \mca{L}_{\mf{p}}$ with $|\mca{P}|=\zeta_d$, and for all distinct $L,L'\in \mca{P}$, $x_L$ is anticomplete to $L'$ in $G$. This proves \eqref{st:bundleleminpatch}.
\begin{figure}
    \centering
    \includegraphics[width=\textwidth]{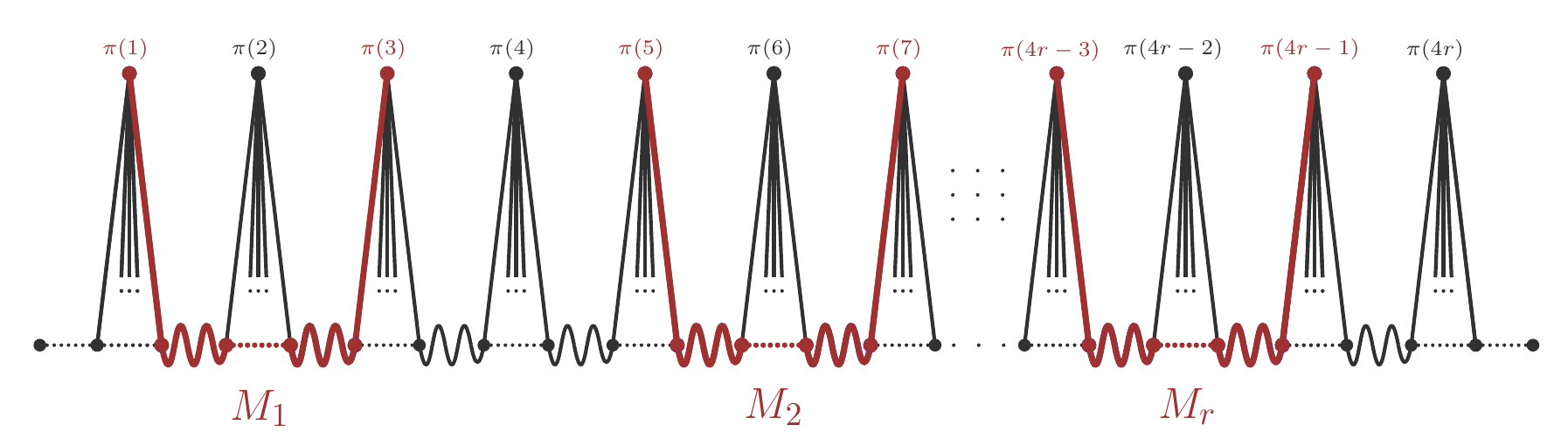}
    \caption{Proof of \eqref{st:bundleleminpatch}.}
    \label{fig:match}
\end{figure}
\medskip

Henceforth, let $\mca{P}\subseteq \mca{L}_{\mf{p}}$ be as in \eqref{st:bundleleminpatch}. Let $S=\{x_L:L\in \mca{P}\}$. Then $S$ is a stable set in $G$ of cardinality $\zeta_d$. 

\sta{\label{st:d>0} We have $d\geq 1$.}

Suppose for a contradiction that $d=0$. Then $|S|=\zeta_0=h^{8h^4}$ (recall that $h=\max\{4r,r',m\}$). Also, by \ref{P3}, $G'=G[S\cup V(\mca{P})\cup S_{\mf{p}}]$ is connected. So we can apply Theorem~\ref{connectifier} to $G'$ and $S\subseteq V(G')$. Since Lemma~\ref{lem:digraph}\ref{lem:digraph_a} does not hold, it follows that $G$ does not contain $K_m$, and so the third outcome of Theorem~\ref{connectifier} does not hold.  We deduce that there is an induced subgraph $H$ of $G'$ for which one of the following holds.
  \begin{itemize}
  \item $H$ is a path in $G$ with $|H\cap S|=4r$.  
  \item $H$ is either a caterpillar or the line graph of a caterpillar in $G$ with $|H\cap S|=4r$ and $H\cap S=\mca{Z}(H)$.
  \item $H$ is a subdivided star with root $z$ such that $|H\cap S|=r'+1$ and $\mca{Z}(H)\subseteq H\cap S\subseteq \mca{Z}(H)\cup \{z\}$.
\end{itemize}
Assume that one of the first two bullets above holds. Then since $S$ is a stable set in $G$, one may readily observe that there exists a plain $(2,r)$-match $\mca{M}$ for $S\subseteq X$ in $G$ with $V(\mca{M})\subseteq V(H)\subseteq V(G')=S\cup V(\mca{P})\cup S_{\mf{p}}$. But then $\mca{M}$ is a plain $(2,r)$-match for $X$ in $G$ with $V(\mca{M})\subseteq V(\mf{p})$, and so Lemma~\ref{lem:digraph}\ref{lem:digraph_c} holds, a contradiction. Now assume that the third bullet above holds. Then one may pick a set 
$\mca{Q}$ of $r'$ pairwise distinct components of $H\setminus z$. It follows that $\mca{Q}$ is a plain $r'$-polypath in $G$. Consequently, $\mf{q}=(\{z\},\mca{Q})$ is a plain $(0,r')$-patch for $S\subseteq X$ in $G$ with $V(\mf{q})\subseteq V(H)\subseteq V(G')=S\cup V(\mca{P})\cup S_{\mf{p}}\subseteq V(\mf{p})$, and so Lemma~\ref{lem:digraph}\ref{lem:digraph_d} holds, again a contradiction. This proves \eqref{st:d>0}.
\medskip

In view of \eqref{st:d>0}, for every $L\in \mca{P}\subseteq \mca{L}_{\mf{p}}$, $L$ has non-zero length, and so $x_L$ has a unique neighbor $x'_L$ in $L\setminus \{x_L\}$. Let $X'=\{x'_L:L\in \mca{P}\}$, and let $\mf{p}'=(X',\{L\setminus \{x_L\}: L\in \mca{P}\})$. Then it is straightforward to check that $\mf{p}'$ is a $(d-1,\zeta_d)$-patch for $X'$ in $G$. This, along with \eqref{st:d>0} and the definition of $\eta$, implies that $\mf{p}'$ is a $(d-1,\eta(d-1,l,m,r,r',s))$-patch for $X'$ in $G$. Also, we have $V(\mf{p}')\subseteq V(\mf{p})$. Now, from the minimality of $d$, it follows that 
\begin{itemize}
    \item either there is a plain $(2d,r)$-match $\mca{M}'$ for $X'$ in $G$ with $ V(\mca{M}')\subseteq V(\mf{p}')\subseteq V(\mf{p})$; or
    \item there is a plain $(d-1,r')$-patch $\mf{q}'$ for $X'$ in $G$ with $V(\mf{q}')\subseteq V(\mf{p}')\subseteq V(\mf{p})$.
\end{itemize}
In the former case, by \eqref{st:d>0}, each path in $\mca{M}'$ has non-zero length, and so for every $M\in \mca{M}'$, there are two distinct paths $K_M,L_M\in \mca{P}$ such that $\partial M=\{x'_{K_M},x'_{L_M}\}$; let $M^+=x_{K_M}\dd x'_{K_M}\dd M\dd x'_{L_M}\dd x_{L_M}$.
But then by \eqref{st:bundleleminpatch}, $\mca{M}=\{M^+:M\in \mca{M}'\}$ is a plain $2(d+1)$-match for $X$ in $G$ such that $ V(\mca{M})\subseteq V(\mca{M}')\cup \partial \mca{L}_{p}\subseteq V(\mf{p})$, and Lemma~\ref{lem:digraph}\ref{lem:digraph_c} holds, a contradiction. Moreover, in the latter case, by \ref{P3}, for each $Q\in \mca{L}_{\mf{q}'}$, there exists exactly one path $R_Q\in \mca{P}$ such that $Q\cap X'=\partial Q\cap X'=\{x'_{R_Q}\}$, and assuming $y_Q$ to be the other end of $Q$ (possibly $x'_{R_Q}=y_Q$), we have $N_Q(S_{\mf{q}'})=\{y_Q\}$. Let $Q^+=x_{R_Q}\dd x'_{R_Q}\dd Q\dd y_Q$.
Then by \eqref{st:bundleleminpatch}, $\mf{q}=(S_{\mf{q}'}, \{Q^+:Q\in \mca{L}_{\mf{q}'}\})$ is a plain $(d,r')$-patch for $X$ in $G$ with $ S_{\mf{q}}\cup V(\mca{L}_{\mf{q}})\subseteq V(\mf{q}')\cup \partial \mca{L}_{p}\subseteq V(\mf{p})$. But now Lemma~\ref{lem:digraph}\ref{lem:digraph_d} holds, again a contradiction. This completes the proof of Lemma~\ref{lem:digraph}.
\end{proof}

We can now prove the main result of this section: 

\begin{proof}[Proof of Theorem~\ref{thm:block_to_constellation}]
Let $H$ be the disjoint union of $c$ cycles, each of length exactly $o+2$ (so $H$ is the unique $2$-regular graph, up to isomorphism, with exactly $c$ components, each on $o+2$ vertices). Let $m=m(H,o+2,t)$ be as in Theorem~\ref{dvorak} (note that here $m$ only depends on $c,o$ and $t$).  Let
$$\Theta=\Theta(c,2,c,2,l,o,s,t);$$
$$\theta=\theta(c,2,c,2,l,o,s,t)$$
be as in Theorem~\ref{thm:bundlethm}, and let $$\eta=\eta(o-1,l,m,\theta,\theta,s)$$
be as in Lemma~\ref{lem:digraph}. Define 
$$\Omega=\Omega(c,l,o,s,t)=\kappa(o,\max\{2\Theta,\eta\},m)$$
where $\kappa=\kappa(\cdot,\cdot,\cdot)$ is as in Theorem~\ref{distance}. We prove that the above value of $\Omega$ satisfies Theorem~\ref{thm:block_to_constellation}. 

Let $G$ be a $(c,o)$-perforated graph and let $B$ be a strong $\Omega$-block in $G$. Suppose for a contradiction that $G$ contains neither $K_t$ nor $K_{t,t}$, and there is no plain $(s,l)$-constellation in $G$. Since $G$ is $(c,o)$-perforated, it follows that $G$ contains no induced subgraph isomorphic to a subdivision of $H$. Therefore, by Theorem~\ref{dvorak} and the choice of $m$, $G$ contains no subgraph isomorphic to a $(\leq o)$-subdivision of $K_m$. It is convenient to sum up all this in one statement:

\sta{\label{st:blockcontradiction} The following hold.
\begin{itemize}
    \item $G$ does not contain $K_{t,t}$.
    \item $G$ does not contain a subgraph isomorphic to a $(\leq o)$-subdivision of $K_m$.
    \item There is no plain $(s,l)$-constellation in $G$.
\end{itemize}}

In particular, the choice of $\Omega$  along with Theorem~\ref{distance} and the second bullet of \eqref{st:blockcontradiction}, implies that there exists $A\subseteq V(G)$ as well as $B'\subseteq B\setminus A$ such that $B'$ is both a $\max\{2\Theta,\eta\}$-block and $o$-stable set in $G\setminus A$. In particular, one may choose $2\Theta$ distinct vertices $\{x_i,y_i:i\in [\Theta]\}$ in $G$ as well as, for each $i\in [\Theta]$, a collection $\mca{P}_i$ of $\eta$ pairwise internally disjoint paths in $G$ from $x_i$ to $y_i$, such that:
\begin{itemize}
    \item for every $i\in [\Theta]$, every path $P\in \mca{P}_i$ has length at least $o+1\geq 2$; and
    \item $V(\mca{P}_1),\ldots, V(\mca{P}_{\Theta})$ are pairwise disjoint in $G$.
\end{itemize}

For each $i\in [\Theta]$, let $\mca{L}_i=\mca{P}^*_i$. Also, for every $L\in \mca{L}_i$, let $x_L$ and $y_L$ be the (unique) neighbors of $x_i$ and $y_i$ in $P$, respectively; so we have $\partial L=\{x_L,y_L\}$ (note that $x_L,y_L$ might be the same, but they are distinct from $x_i$ and $y_i$). Let $X_i=\{x_L:L\in \mca{L}_i\}$. It follows that $\mca{L}_i$ is an $\eta$-polypath in $G$ such that
\begin{itemize}
    \item $y_i\in V(G)\setminus V(\mca{L}_i)$;
    \item every path $L\in \mca{L}_i$ has length at least $o-1$; and
    \item $L\cap X_i=\{x_L\}$ and $N_L(y_i)=\{y_L\}$.
\end{itemize}
Therefore, by \ref{P1}, \ref{P2} and \ref{P3}, for every $i\in [\Theta]$, the $(1,\eta)$-bundle $\mf{p}_i=(\{y_i\},\mca{L}_i)$ is an $(o-1,\eta)$-patch for $X_i$ in $G$. This, along with the choice of $\eta$, allows for an application of Lemma~\ref{lem:digraph} to $X_i$ and $\mf{p}_i$. Since Lemma~\ref{lem:digraph}\ref{lem:digraph_a} and Lemma~\ref{lem:digraph}\ref{lem:digraph_b} violate the second and the third bullet of \eqref{st:blockcontradiction}, respectively, we deduce that for every $i\in [\Theta]$, one of the following holds.
 \begin{itemize}
         \item There is a plain $(2o,\theta)$-match $\mca{M}_i$ for $X_i$ in $G$ with $ V(\mca{M}_i)\subseteq V(\mf{p}_i)\subseteq V(G)\setminus \{x_i\}$. See Figure~\ref{fig:P&M}, Left.
        
        \item There is a plain $(o-1,\theta)$-patch $\mf{q}_i$ for $X_i$ in $G$ with $V(\mf{q}_i) \subseteq V(\mf{p}_i)\subseteq V(G)\setminus \{x_i\}$ (note that $S_{\mf{q}_i}\subseteq X_i$ is also possible). See Figure~\ref{fig:P&M},  Middle and Right.
    \end{itemize}

Now, for each $i\in [\Theta]$, we define the plain $(\leq 2,\theta)$-bundle $\mf{b}_i$ as follows. In the former case above, let $S_{\mf{b}_i}=\{x_i\}$ and let $\mca{L}_{\mf{b}_i}=\mca{M}_i$, and in the latter case, let $S_{\mf{b}_i}=S_{\mf{q}_i}\cup \{x_i\}$ and let $\mca{L}_{\mf{b}_i}=\mca{L}_{\mf{q}_i}$.

\begin{figure}
\centering
\begin{tikzpicture}[scale=0.6, auto=left, mblack]
\tikzstyle{every node}=[inner sep=1.3pt, circle]  
\centering

%%%%% MATCH
%%%%% xi

\node [label={270:$x_i$},fill] at(0,-1) {};

%%%%% X_i

\node [fill] at(-4,1) {};
\node [inner sep=2pt,circle,draw,thick] at (-4,1) {};
\node [fill] at(-2,1) {};
\node [inner sep=2pt,circle,draw,thick] at (-2,1) {};
\node [fill] at(-1,1) {};
\node [inner sep=2pt,circle,draw,thick] at (-1,1) {};
\node [fill] at(1,1) {};
\node [inner sep=2pt,circle,draw,thick] at (1,1) {};
\node [fill] at(2,1) {};
\node [inner sep=2pt,circle,draw,thick] at (2,1) {};
\node [fill] at(4,1) {};
\node [inner sep=2pt,circle,draw,thick] at (4,1) {};

%%%%%%% Paths in M

\draw[thick, snake it] (4,1) arc (0:180:4cm);
\draw[thick, snake it] (2,1) arc (0:180:2cm);
\draw[thick, snake it] (1,1) arc (0:180:1cm);
\path (0,5) -- node[auto=false, thick]{$\vdots$} (0,4);
\path (2,3) -- node[auto=false, thick]{$\iddots$} (3,3.5);
\path (-2,3) -- node[auto=false, thick]{$\ddots$} (-3,3.5);

%%%%%% Edges from xi

\draw[thick] (0,-1) -- (4,1);
\draw[thick] (0,-1) -- (2,1);
\draw[thick] (0,-1) -- (1,1);
\draw[thick] (0,-1) -- (-1,1);
\draw[thick] (0,-1) -- (-2,1);
\draw[thick] (0,-1) -- (-4,1);

%
%
%

%%%%%% PATCH 1
%%%%% xi

\node [label={270:$x_i$},fill] at(8.5,-1) {};

%%%%% Sqi

\node [label={90:$S_{\mf{q}_i}$}, fill] at(8.5,5) {};

%%%%% X_i

\node [fill] at(7,1) {};
\node [inner sep=2pt,circle,draw,thick] at (7,1) {};
\node [fill] at(8,1) {};
\node [inner sep=2pt,circle,draw,thick] at (8,1) {};
\node [fill] at(10,1) {};
\node [inner sep=2pt,circle,draw,thick] at (10,1) {};

\node [fill] at(7,4) {};
\node [fill] at(8,4) {};
\node [fill] at(10,4) {};

%%%%%%% Paths in Lqi

\draw[thick, snake it] (7,1) -- (7,4);
\draw[thick, snake it] (8,1) -- (8,4);
\draw[thick, snake it] (10,1) -- (10,4);
\path (8,2.5) -- node[auto=false, thick]{$\cdots$} (10,2.5);

%%%%%% Edges from xi

\draw[thick] (8.5,-1) -- (7,1);
\draw[thick] (8.5,-1) -- (8,1);
\draw[thick] (8.5,-1) -- (10,1);

%%%%%% Edges from Sqi
\draw[thick] (8.5,5) -- (7,4);
\draw[thick] (8.5,5) -- (8,4);
\draw[thick] (8.5,5) -- (10,4);

%%%%% PATCH 2
%%%%% xi

\node [label={270:$x_i$},fill] at(19.5,-1) {};

%%%%% Sqi

\node [label={240:$S_{\mf{q}_i}$},fill] at(14.5,1) {};
\node [inner sep=2pt,circle,draw,thick] at (14.5,1) {};

%%%%% X_i

\node [fill] at(13.2,1.9) {};

\node [fill] at(15.3,1.9) {};

\node [fill] at(16.7,1.9) {};

\node [fill] at(18,1) {};
\node [inner sep=2pt,circle,draw,thick] at (18,1) {};
\node [fill] at(19,1) {};
\node [inner sep=2pt,circle,draw,thick] at (19,1) {};
\node [fill] at(21,1) {};
\node [inner sep=2pt,circle,draw,thick] at (21,1) {};

%%%%%%% Paths in Lqi

\draw[thick, snake it] (21,1) arc (0:166:4cm);
\draw[thick, snake it] (19,1) arc (0:150:2cm);
\draw[thick, snake it] (18,1) arc (0:110:1cm);
\path (17,4) -- node[auto=false, thick]{$\vdots$} (17,5);
\path (19,3) -- node[auto=false, thick]{$\iddots$} (20,3.5);
\path (15,3) -- node[auto=false, thick]{$\ddots$} (14,3.5);

%%%%%% Edges from xi

\draw[thick] (19.5,-1) -- (18,1);
\draw[thick] (19.5,-1) -- (19,1);
\draw[thick] (19.5,-1) -- (21,1);
\draw[thick] (19.5,-1) -- (14.5,1);

%%%%%% Edges from Sqi

\draw[thick] (14.5,1) -- (13.2,1.9);
\draw[thick] (14.5,1) -- (15.3,1.9);
\draw[thick] (14.5,1) -- (16.7,1.9);

\end{tikzpicture}

\caption{Left: squiggly arcs depict the paths in $\mca{M}_i$, each of length at least $2o$. Middle and Right: squiggly arcs depict the paths in $\mca{L}_{\mf{q_i}}$, each of length at least $o-1$. In all three figures, circled nodes represent the vertices in $X_i$.}
\label{fig:P&M}
\end{figure}
    It follows that $\mf{T}=\{\mf{b}_1,\ldots, \mf{b}_{\Theta}\}$ is a collection of $\Theta$ pairwise disentangled plain $(\leq 2,\theta)$-bundles in $G$. Given the choices of $\Theta$ and $\theta$, we can apply Theorem~\ref{thm:bundlethm} to $\mf{T}$, which combined with the first and third bullet of \eqref{st:blockcontradiction} implies that there exists $I\subseteq [\Theta]$ with $|I|=c$ as well as $\mca{H}_{i}\subseteq \mca{L}_{\mf{b}_i}$ with $|\mca{H}_{i}|=2$ for each $i\in I$, such that for all distinct $i,i'\in I$, $S_{\mf{b}_i}\cup V(\mca{H}_{i})$ is anticomplete to $S_{\mf{b}_{i'}}\cup V(\mca{H}_{i'})$ in $G$. In particular, for every $i\in I$, there is a cycle $H_i$ in $G[S_{\mf{b}_{i'}}\cup V(\mca{H}_{i'})]$ of length at least $2o+2$. But now $H_1,\ldots, H_c$ comprise $c$ pairwise disjoint and anticomplete cycles in $G$, each of length at least $o+2$, which violates the assumption that $G$ is $(c,o)$-perforated. This completes the proof of Theorem~\ref{thm:block_to_constellation}.
\end{proof}
 
\section{The long-awaited conclusion}\label{sec:end}

Eventually, we can bring everything together and give a proof of Theorem~\ref{mainthmgeneral}:
\setcounter{section}{3}
\setcounter{theorem}{1}
\begin{theorem}
    For all   $c,o,t\in \poi$ and $s \in \poi \cup \{0\}$, there is a constant $\tau=\tau(c,o,s,t)\in \poi$ such that every $(c,o)$-perforated graph of treewidth more than $\tau$ contains either $K_t$ or $K_{t,t}$, or there is a full $(s,o)$-occultation in $G$.
\end{theorem}
\begin{proof}
    Let $\Sigma=\Sigma(c,o,s,t)$ and     $\Lambda=\Lambda(c,o,s,t)$ be as in Theorem~\ref{thm:const_to_fire}. Let $\Omega=\Omega(c,\Lambda,o,\Sigma,t)$ be as in Theorem~\ref{thm:block_to_constellation}. Let  
    $\tau(c,o,s,t)=\xi(c,o,\Omega,t)$, where $\xi(\cdot,\cdot,\cdot,\cdot)$ is as in Corollary~\ref{noblocksmalltw_polycyclefree}.  Let $G$ be a $(c,o)$-perforated graph of treewidth more than $\tau$. Assume that $G$ contains neither $K_t$ nor $K_{t,t}$. By Corollary~\ref{noblocksmalltw_polycyclefree}, $G$ contains a strong $\Omega$-block. Therefore, by Theorem~\ref{thm:block_to_constellation}, there exists a plain $(\Sigma,\Lambda)$-constellation in $G$. But now by Theorem~\ref{thm:const_to_fire}, $G$ contains a full $(s,o)$-occultation, as desired.
\end{proof}

\setcounter{section}{8}

\section{Acknowledgment}
  This work was partly done during the Second 2022 Barbados Graph Theory Workshop at the Bellairs Research Institute of McGill University, in Holetown, Barbados. We thank the organizers for inviting us and for creating a stimulating work environment.
  
%%% AUTHOR:
%%% Bibliography goes here. Note that the arXiv cannot process bibtex
%%% or biber bibliographies.  Example of acceptable bibliograpy format:
\bibliographystyle{amsplain}

%% AUTHOR: You can generate such a bibliography from a .bib file by 
%% running pdflatex/bibtex/pdflatex/pdflatex and then pasting the .bbl file
%% between \begin{thebibliography} and \end{bibliography}

%%% AUTHOR: Include a short description of each author following the
%%% structure below. Use the same short tags used previously.  
%%% Use \imageat{} and \imagedot{} instead of "@" and "." in
%%% email addresses-this replaces the symbols with graphics to avoid 
%%% e-mail address harvesting from the .pdf file
\begin{aicauthors}
\begin{authorinfo}[balec]
  Bogdan Alecu\\
  School of Computing, University of Leeds\\
  Leeds, UK
% b.alecu\imageat{}leeds\imagedot{}ac\imagedot{}uk \\
  % \url{https://www.cs.elte.hu/erdos}
\end{authorinfo}
\begin{authorinfo}[mchud]
 Maria Chudnovsky\\
 Princeton University\\
 Princeton, New Jersey, USA
\end{authorinfo}
\begin{authorinfo}[laci]
Sepehr Hajebi\\
Department of Combinatorics and Optimization, University of Waterloo\\
Waterloo, Ontario, Canada
\end{authorinfo}
\begin{authorinfo}[andy]
 Sophie Spirkl\\
Department of Combinatorics and Optimization, University of Waterloo\\
Waterloo, Ontario, Canada
\end{authorinfo}
\end{aicauthors}

\end{document}